\documentclass[11pt, a4paper,leqno]{amsart}
\usepackage{amsmath,amsthm,amscd,amssymb,amsfonts, amsbsy}
\usepackage{stmaryrd}
\usepackage[english]{babel}
\usepackage{latexsym}
\usepackage{txfonts}
\usepackage{exscale}
\usepackage{bbm}
\usepackage{enumitem}
\usepackage{soul}
\usepackage{mathtools}
\mathtoolsset{showonlyrefs}

\usepackage[colorlinks,citecolor=red,pagebackref,hypertexnames=false]{hyperref}
\usepackage{color}

\calclayout
\allowdisplaybreaks

\theoremstyle{plain}
\newtheorem{theorem}[equation]{Theorem}
\newtheorem{lemma}[equation]{Lemma}

\theoremstyle{definition}
\newtheorem{definition}[equation]{Definition}

\theoremstyle{remark}
\newtheorem{remark}[equation]{Remark}

\newcommand{\Xbf}{\mathbf{X}}
\newcommand{\xbf}{\mathbf{x}}
\newcommand{\Ybf}{\mathbf{Y}}

\newcommand{\Zbf}{\mathbf{Z}}

\numberwithin{equation}{section}

\newcommand{\RR}{{\mathbb{R}}}
\newcommand{\ZZ}{{\mathbb{Z}}}

\newcommand{\dist}{\operatorname{dist}}

\newcommand{\re}{\mathbb{R}}
\newcommand{\rn}{\mathbb{R}^n}

\newcommand{\N}{\mathbb{N}}
\newcommand{\dd}{\mathbb{D}}

\newcommand{\F}{\mathcal{F}}

\newcommand{\cH}{\mathcal{H}}

\newcommand{\M}{\mathcal{M}}

\newcommand{\B}{\mathcal{B}}

\newcommand{\sbf}{{\bf S}}

\newcommand{\G}{\mathcal{G}}

\newcommand{\Span}{\operatorname{span}}

\newcommand{\mut}{\mathfrak{m}}

\newcommand{\RNum}[1]{\uppercase\expandafter{\romannumeral #1\relax}}

\newcommand{\meas}{\cH_{\text{par}}^{n+1}}
\newcommand{\pcubes}{\Delta}
\newcommand{\fapprox}{\gamma}
\newcommand{\sapprox}{\beta}

\newcommand{\size}[1]{\ell({#1})}

\newcommand{\cubes}{\dd}

\newcommand{\PBMO}{P\text{-}\operatorname{BMO}}

\renewcommand{\emptyset}{\mbox{\textup{\O}}}

\DeclareMathOperator{\supp}{supp}

\DeclareMathOperator{\diam}{diam}

\title[BPRPBI and Parabolic UR]{Big Pieces of Regular Parabolic (bi-)Lipschitz Images is Equivalent to Parabolic Uniform Rectifiability} 
\author[Simon Bortz, Matthew Hyde, Mason Sharp]
{Simon Bortz, Matthew Hyde, Mason Sharp}

\address{Simon Bortz , 
Department of Mathematics and CONSERVE-AWI Group
\\
University of Alabama
\\
Tuscaloosa, AL, 35487, USA}
\email{sbortz@ua.edu}

\address{Matthew Hyde, Department of Mathematics and Statistics
\\
University of Jyv\"askyl\"a
\\
P.O. Box 35 (MaD)
\\
FI-40014 University of Jyv\"askyl\"a, Finland}
\email{matthew.j.hyde@jyu.fi}

\address{Mason Sharp, Department of Mathematics
\\
University of Maryland
\\
4176 Campus Drive, William E. Kirwan Hall
\\
College Park, MD 20742-4015, USA}
\email{mesharp3@umd.edu}

\keywords{}
\subjclass[2010]{}
\thanks{S.Bortz is supported through the National Science Foundation DMS-2555449 and the Simons Foundation MPS-TSM-00959861 and an AWI-CONSERVE Fellowship at University of Alabama. M. Hyde is supported by the Research Council of Finland via the project \textit{Quantitative differentiability and rectifiability in metric spaces}, grant no. 363800.}

\begin{document}
\allowdisplaybreaks

\begin{abstract}
	We define the notion of regular parabolic (bi-)Lipschitz images as the parabolic (bi-)Lipschitz maps from $n$-dimensional space time which, up to translation, fix the $t$ variable and whose spatial components are each regular parabolic Lipschitz functions. We show that any parabolic Ahlfors-David regular is parabolic uniformly rectifiable if and only if it has big pieces of parabolic Lipschitz images of $n$-dimensional space time if and only if it has big pieces of parabolic bi-Lipschitz images of $n$-dimensional space time. This further extends the David-Semmes theory to the parabolic setting.
	
	Our proof combines the ideas of \cite{BH-BP, BHHLN-BP} and some ideas of Azzam and Schul \cite{AS-HardSard}. The proof easily adapts (and is far less complicated) to the Euclidean case to give an alternative proof of the analogous fact.
\end{abstract}

\subjclass{28A75}
\maketitle

\tableofcontents

\section{Introduction}

 Throughout we work in $(n+1)$ dimensional space-time $\mathbb{R}^{n+1} = \{(X,t): X \in \mathbb{R}^n, t \in \mathbb{R}\}$ with $n \ge 2$. Our main result is the following. 
 
\begin{theorem}\label{main.thrm}
Suppose that $\Sigma \subseteq \RR^{n+1}$ is a parabolic Ahlfors-David regular set (see Definition \ref{ADR.def}). The following are equivalent: 
\begin{enumerate}
	\item $\Sigma$ is parabolic uniformly rectifiable (P-UR, see Definition \ref{PUR.def}),
	\item $\Sigma$ has big pieces of regular parabolic Lipschitz images of $n$-dimensional space time (BPRPLI, see Definition \ref{BPRPBI.def}),
	\item $\Sigma$ has big pieces of regular parabolic bi-Lipschitz images of $n$-dimensional space time (BPRPBI, see Definition \ref{BPRPBI.def}). 
\end{enumerate}
\end{theorem}

The study of uniform rectifiability (a quantitative notion of rectifiability) was initiated by David and Semmes in two remarkable monographs \cite{DS-Ast,DS-AMS}, drawing inspiration from the work of Jones \cite{Jones-salesman} and the boundedness of the Cauchy integral on Lipschitz graphs \cite{CMM}. In the works \cite{DS-Ast,DS-AMS}, David and Semmes connected various quantitative geometric notions with properties of integral operators. A central condition, on which the parabolic theory of uniform rectifiability rest, relies on the use of the Jones $\beta$-numbers (Definition \ref{betas.def}), which measure the flatness of a set in an $L^2$ averaged sense (more will be said about this later). 

Recall, a set is (qualitatively) $d$-rectifiable in $\mathbb{R}^n$ if it can be covered by a countable number of Lipschitz images of $\RR^d$ up to set of measure zero with respect to the $d$-dimensional Hausdorff measure. David and Semmes defined a set in $\RR^{n}$ to be \textit{uniformly rectifiable} (UR) if it has \textit{big pieces of Lipschitz images} (BPLI). Roughly, this means there are positive $\eta$ and $L$ such that, at every location and scale $(x,r)$ on the set, the set coincides with an $L$-Lipschitz image of $B_d(0,r)$ (the ball of radius $r$ in $\RR^n$) on an $\eta$-percentage with respect to $d$-dimensional Hausdorff measure. By embedding $\RR^n$ into a larger\footnote{This was not required if $d$ is sufficiently small compared to $n$, but needed, for instance, when $d = n-1$.} Euclidean space, David and Semmes showed that UR is equivalent to big pieces of bi-Lipschitz images (BPBI), i.e., one replaces Lipschitz images of $B_d(0,r)$ with bi-Lipschitz images of $\RR^d$ in the larger space. Later, Azzam and Schul \cite{AS-HardSard} showed, among other things, that one did not need to embed $\mathbb{R}^n$ into a larger dimensional Euclidean space to obtain BPBI.

In the above definitions, it is known (by an example of Hyrcak (unpublished)) that one cannot replace (bi-)Lipschitz images with Lipschitz graphs (in contrast to the qualitative setting). This example, based on the Venetian blinds construction, can be found in \cite{azzam2021semi}. However, Azzam and Schul \cite{AS-HardSard} showed that UR sets do satisfy a weaker condition called \textit{big pieces squared of Lipschitz graphs} (BP$^2$LG). An alternative proof that UR sets are BP$^2$LG in co-dimension 1 was provided afterward by the first author and Hofmann \cite{BH-BP}, and subsequently extended to metric spaces by the first author, Hoffman, Hofmann, Luna-Garcia and Nystr\"om \cite{BHHLN-BP}.

The current article concerns the notion of parabolic uniformly rectifiable (P-UR) sets. This theory is motivated by the study parabolic PDEs in domains with non-smooth, dynamic (lateral) boundaries and important singular integrals on their boundaries (e.g. layer potentials). Following the work of Dahlberg \cite{Dahl-L2}, which showed that the harmonic measure was quantitatively absolutely continuous with respect to surface measure in Lipschitz graph domains, it was conjectured that the same should be true in the parabolic setting. In particular, if $\Omega$ is the region above the graph of $\varphi$ satisfying a parabolic Lipschitz condition, then the caloric measure\footnote{The family of measures $\omega^{(X,t)}$, that give the value of the solution to the Dirichlet problem with data $f$ for the heat equation at the point $(X,t)$ via integration against $f$.} should be absolutely continuous with respect to the (parabolic) surface measure on the graph of $\varphi$. This was subsequently shown to be untrue by Kaufman and Wu \cite{KW-counter}. The natural follow-up question was ``What additional assumption on the function $\varphi$ yields the quantitative absolute continuity of the caloric measure with respect to the surface measure?" Lewis and Murray \cite{Lew-Mur-Mem} introduced the concept of {\it regular} parabolic Lipschitz (Lip(1,1/2)) graphs, where one additionally assumes that the function $\varphi$ possesses additional regularly in time, in the form of a half order time derivative in the parabolic BMO space. They showed in \cite{Lew-Mur-Mem} that the parabolic version of Dahlberg's theorem holds under this additional assumption. It took roughly 30 years to establish that this condition is, in fact, necessary for the parabolic version of Dahlberg's theorem; this was shown by the first author, Hofmann, Martell and Nystr\"om \cite{BHMN1}. 

Shortly after the work of Lewis and Murray, Hofmann showed \cite{Hof-SIO} that (homogeneous) parabolic singular integrals are bounded on the graphs of regular parabolic Lipschitz graphs (see \cite{BHHLN-SIO} for the case of non-homogeneous kernels) and Hofmann and Lewis \cite{HL-ann} proved the $L^2$ solvability of the Dirichlet and Neumann boundary value problems under the (necessary) assumption that the half order time derivative of the graph function has sufficiently small BMO norm. With these developments in mind, Hofmann, Lewis and Nystr\"om \cite{HLN1,HLN2} introduced the notion of parabolic uniformly rectifiable sets when studying the parabolic analogues of David and Jerison \cite{DJ} and Kenig and Toro \cite{KT1,KT2,KT3}. Later, under the background assumption of parabolic uniform rectifiability, Engelstein \cite{Eng-parafbp} resolved the remaining parabolic Kenig-Toro theory.

Though Hofmann, Lewis and Nystr\"om introduced the notion of parabolic uniformly rectifiable sets (in terms of a square function estimate for the parabolic beta numbers), the systematic parabolic David-Semmes theory was not undertaken until the 2020s, by the first author, Hoffman, Hofmann, Luna-Garcia and Nystr\"om \cite{BHHLN-corona, BHHLN-BP}. In these works, it was shown that some of the characterizations of uniformly rectifiable sets hold in the parabolic setting. For instance, that there exists a corona decomposition by regular Lip(1,1/2) graphs and that parabolic UR sets are big pieces squared of regular Lip(1,1/2) graphs. It was observed by Hoffman and Jaye \cite{HJ-SIO} that there can be no singular integral characterization of uniformly rectifiable sets, owing to the fact that parabolic singular integrals are merely odd in the spatial variables; on the other hand, Hoffman and Jaye show therein that if the set avoids looking like certain pathological sets (measured in transport distance) at most scales and locations then the singular integral characterization does hold. It is also important to note that several of the important characterizations of UR sets in the David-Semmes {\it fail} to hold in the parabolic setting. For instance the BWGL (`bilateral weak geometric lemma') and similar notions fail to imply parabolic uniform rectifiability; in particular, there are graphs that are very flat and have vanishing flatness at small scales that are not parabolic uniformly rectifiable. 

The purpose of this article is to prove the correct analogue of the David-Semmes characterization of uniformly rectifiable sets in terms of big pieces of Lipschitz and bi-Lipschitz images in the parabolic setting. We define the notion of regular parabolic (bi-)Lipschitz maps as the parabolic (bi-)Lipschitz maps whose spatial coordinates are regular Lip(1,1/2) functions and which act on the time coordinate by translation, and we show that P-UR sets are precisely the sets which have big pieces of regular parabolic (bi-)Lipschitz images. This fact follow from the set of implications  
\begin{align}
	\mbox{BPRPLI $\implies$ P-UR $\implies$ BPRPBI $\implies$ BPRPLI.} 
\end{align}

In this article, we prove the first and second implication (the final implication is trivial), see Theorem \ref{t:BPRPLI-implies-UR} and Theorem \ref{t:UR-implies-BPLI}, respectively. The first implication, that big pieces of regular parabolic Lipschitz images implies parabolic uniform rectifiability, proceeds as follows. First, we prove a bi-Lipschitz decomposition theorem for regular parabolic Lipschitz maps (\textit{\`a la} Jones \cite{jones1988lipschitz}) and use this to find a subset $A \subseteq B_d(0,r)$ on which the Lipschitz map $f$ from the definition of BPRPLI is bi-Lipschitz and such that $f(A)$ has large intersection with the set. By adding small expansion factors away from $A$, we produce a new map, into some higher dimensional space time $\RR^{m+1}$, which which agrees with the original map on $A$ and is bi-Lipschitz onto its image. This is similar to how David and Semmes construct bi-Lipschitz extensions in \cite{DS-Ast}. The point is, we need to do this in such a way that each spatial component is regular parabolic Lipschitz. A general transference principle  \cite{BHHLN-BP, DS-AMS, Rigot} then allows us to reduce to showing that regular parabolic bi-Lipschitz images of $n$-dimensional space time are parabolic uniformly rectifiable. This latter fact follows, essentially, from the definition of regular parabolic Lipschitz functions and the Pythagorean theorem.

The other direction, that parabolic uniform rectifiability implies big pieces of regular parabolic bi-Lipschitz images, is significantly more complicated. The general idea is an inductive scheme on the Carleson packing condition in the Corona decomposition for parabolic uniformly rectifiable sets for a fixed cube $Q_0$. To move from one packing constant to the next, we have to take several regular parabolic bi-Lipschitz images and marry them together. Two advantages we have are that we have a graphical approximation (but not coincidence) of the set under study and that the images we are gluing together look like planes at large enough scales. On the other hand, the planes for these images are not (even close to) parallel, so our task becomes rotating them in a smooth manner and gluing them together. Less smooth (Lipschitz) rotations like this are present in the work of Azzam and Schul \cite{AS-HardSard}, whose proof uses different techniques.

\begin{remark}[Simplifications in the elliptic case]
	A particular challenge in our setting is that we need to preserve regularity that is more than Lipschitz, adding complication to many of the steps involved. In the elliptic case, the several proofs simplify dramatically. Notably, we can dispense with Section \ref{s:gluing-lemma} entirely (and the part of any proof which checks the hypotheses there) and we need not worry about the producing such regular rotations (as in Lemma \ref{lem:givenrot}), but those that are merely Lipschitz. 
\end{remark}

\section{Preliminaries} 

Points in  Euclidean space-time
$ \mathbb R^{n+1}$ are denoted by $(X,t)$. However, we will often write $\Xbf$ to denote a point in space time, that is,
\[ \Xbf = (X,t).\]
When we are in $n$-dimensional space time $\mathbb{R}^{n-1} \times \mathbb{R}$, we will sometimes use the notation
\[(x,t)  \in \mathbb{R}^{n-1} \times \mathbb{R}.\] We set
\[d:= n+1.\]
Below, $\nabla$ will always denote the spatial gradient.

We let   $\langle \cdot ,  \cdot  \rangle $  denote  the standard inner
product on $ \mathbb R^{n} $ and we let  $  | X | = \langle X, X \rangle^{1/2} $ be
the  Euclidean norm of $ X\in\re^n$.  Given $(X,t), (Y,s)\in\mathbb R^{n+1}$,
we define the parabolic distance between $(X,t)$ and $(Y,s)$ as
\begin{equation}\label{e:parabolic-dist}
	\dist((X,t),(Y,s)) \coloneqq \dist(X,t,Y,s):= 
|X-Y|+|t-s|^{1/2}.
\end{equation}
We also use its $L^\infty$ parabolic distance 
\[d_\infty(X,t,Y,s) := \max\{\max_i\{|X_i - Y_i|\}, |t-s|^{1/2}\},\]
where $(X,t) = (x_1,\dots, x_n, t)$ and $(Y,s) = (y_1,\dots, y_n, s)$.

For $(X,t) \in \mathbb{R}^{n+1}$ and $R > 0$, we write
\[B((X,t), R) = \{(Y,s): \dist((X,t), (Y,s)) < R\}\]
to be the parabolic ball of radius $R$ centered at $(X,t)$. 

Given $\Xbf$ and $r>0$, we define the closed cube
\[Q(\Xbf, r) = \{\Ybf: \dist_\infty(\Ybf,\Xbf) \le r\}\]
and we let 
\[\ell(Q) = 2r\]
denote the `parabolic side-length' of $Q$.

We also sometimes employ the following function
\[\rho(X,t):= (|X|^4 + |t|^2)^{1/4}.\]
One can deduce 
\begin{equation}\label{rhodistcomp.eq}
\rho(X,t) \le \dist((X,t), (0,0)) \le  2^{3/4}\rho(X,t).
\end{equation}

We work with a suitable Hausdorff measure adapted to the parabolic scaling, which we introduce now. For $\delta>0$, $\alpha \ge 0$ and $E\subset \re^{n+1}$, we set
  \[ \cH_{\text{par},\delta}^\alpha(E):= \inf \sum_k (\diam(E_k))^\alpha,
  \]
  where the infimum runs over all countable such coverings of $\{E_k\}_k$ of $E$ with $\diam(E_k)\leq \delta$ for all $k$ (here the diameter is defined with respect to the parabolic distance). We then define
  \[
  \cH_{\text{par}}^\alpha (E) := \lim_{\delta\to 0^+} \cH_{\text{par},\delta}^\alpha(E)\,.
  \]
As is the case of classical Hausdorff measure, $ \cH_{\text{par}}^\alpha$ is a Borel regular measure.

Given a closed set $\Sigma \subset \mathbb R^{n+1}$, we
define the $(n+1)$ dimensional parabolic surface measure on
$\Sigma$ as the restriction of $\cH_{\text{par}}^{n+1}$ to $\Sigma$, i.e.,
\begin{equation}\label{sigdef}
\sigma = \sigma_\Sigma:=  \cH_{\text{par}}^{n+1}|_\Sigma\,.
\end{equation}

\begin{definition}[Parabolic Ahlfors-David regular]\label{ADR.def}
We say that a set $\Sigma\subset \rn\times \mathbb{R}$
is \textit{parabolic Ahlfors-David regular, parabolic ADR} or simply \textit{ADR} for short, with constant $C \ge 1$, if it is closed and
\[C^{-1} r^{n+1} \le \sigma(B(X,t,r)) \le  C r^{n+1},
\quad \forall (X,t) \in \Sigma, \,\, r> 0. \]
\end{definition}

A central part of the theory of uniformly rectifiable sets is the Jones $\beta$-numbers. In the parabolic setting, one needs to work with planes that are properly oriented, that is, $t$-independent planes.
\begin{definition}[$t$-independent planes]\label{tindplanes.def}
An $n$-dimensional hyperplane $P\subset \RR^{n+1}$ is called \emph{$t$-independent} if its unit normal is purely spatial, that is,
\[
    n_P=(\nu,0), \qquad \nu\in \mathbb S^{n-1}.
\]
Equivalently, there exists a constant $c$ so that 
\[
    P=P_\nu:=\{(X,t)\in\RR^n\times\RR:X\cdot \nu=c\}.
\]
In particular, a $t$-independent plane contains a line in the $t$-direction.
\end{definition}

\begin{definition}[Parabolic $\beta$-numbers]\label{betas.def}
Let $\Sigma \subset \mathbb{R}^{n+1}$ be an ADR set. For $(X,t) \in \Sigma$ and $r> 0$, we define the \textit{(parabolic) $\beta$-number} as
\[\beta_\Sigma((X,t),r): = \inf_{P \in \mathcal{P}} \left( \frac{1}{\sigma(B((X,t), r))}\int_{B((X,t), r)}  \left( \frac{\dist((Y,s), P)}{r} \right)^2 \, d\sigma(Y,s) \right)^{1/2},\]
where $\mathcal{P}$ is the collection of $t$-independent planes. In the case that $\Sigma$ is clear, we shall often simply write $\beta \coloneqq \beta_\Sigma$. 
\end{definition}

\begin{definition}[Parabolic Uniform Rectifiability, P-UR]\label{PUR.def}
Let $\Sigma \subset \mathbb{R}^{n+1}$ be an ADR set. We say that $\Sigma$ is \textit{parabolic uniformly rectifiable} (P-UR) if 
\[\beta((X,t), r)^2 \frac{d\sigma(X,t) dr}{r}\]
is a Carleson measure on $\Sigma \times (0, \infty)$. This means, there exists a constant $M$ such that, for $(Z,\tau) \in \Sigma$ and $R > 0$, 
\begin{equation}\label{betacarldef.eq}
\int_0^R \int_{B((Z,\tau), R)} \beta((X,t), r)^2 d\sigma(X,t) \frac{ dr}{r} \le MR^d.
\end{equation}
\end{definition}

The prototypical P-UR sets are regular parabolic Lipschitz graphs.

\begin{definition}[Parabolic Lipschitz functions and gamma numbers]\label{regfngraph.def}
A function $f:\mathbb{R}^{n} \to \mathbb{R}$ on $n$-dimensional space time is \textit{parabolic Lipschitz} (or Lip(1,1/2)) if there exists a constant $L$ such that
\[|f(x,t) - f(y,s)| \le L\dist((x,t),(y,s)).\]
For a Lip(1,1/2) function $f$, we define the $\gamma$-numbers to be 
\[\gamma_f(\Xbf, r) \coloneqq \inf_{A}\left(\frac{1}{|Q(\Xbf,r)|}  \iint_{Q(\Xbf,r)} \frac{|f(y,s) - A(y)|^2}{r^2} \, dy \,ds\right)^\frac{1}{2},\]
where $|Q(\Xbf,r)|$ denotes the Lebesgue measure of $Q(\Xbf,r)$ and the infimum is over all affine functions of the spatial variables, that is, $A(y) = a + \vec{b}\cdot y$. Here, $Q(\Xbf,r)$ is an $n$-dimensional space-time cube.
\end{definition}

If $f$ is Lip(1,1/2) and 
\begin{align}
	\Gamma = \textrm{graph}(f) \coloneqq \{(x_1, x,t): x_1 = f(x,t)\},
\end{align}
it is known that 
\[\beta_\Gamma(f(\Xbf),r) \lesssim  \gamma_f(\Xbf, r)  \lesssim \beta_\Gamma(f(\Xbf),(1+L)r),\, \] 
where the implicit constants depend on dimension and the Lip(1,1/2) constant of $f$. 
Thus, $\beta_\Gamma(f(\Xbf),r)^2 r^{-1} \, d\Xbf \, dr$ is a Carleson measure on $\Gamma \times (0,\infty)$ if and only if $\gamma_f(\Xbf,r)^2 r^{-1} \, d\Xbf \, dr$ is a Carleson measure on $\mathbb{R}^n \times (0,\infty)$. This leads to the following definition. 

\begin{definition}[Regular parabolic Lipschitz functions]
We say $f$ is \textit{{\bf regular} parabolic Lipschitz} (or regular Lip(1,1/2)) if the graph $\Gamma = \textrm{graph}(f)$ is P-UR. We say  $f:\mathbb{R}^n \to \mathbb{R}$ is an $(L,M)$-regular Lip(1/1/2) function and $\Gamma$ is $(L,M)$-regular parabolic Lipschitz, if $f$ is parabolic Lipschitz with constant $L$ and $\gamma_f(\Xbf,r)^2 r^{-1} \, d\Xbf \, dr$ is a Carleson measure with constant $M$, that is,
\begin{align}\label{e:gamma-carleson}
	\int_0^R \int_{Q(\Xbf,R)} \gamma_f(\Ybf,r)^2 r^{-1} \, d\Ybf \, dr \le M|Q(\Xbf,R)|, \quad \forall Q(\Xbf,R).
\end{align}
\end{definition}

\begin{remark}
	Regular parabolic Lipschitz is usually defined by asking that $D_t^{1/2}f$ belongs to the parabolic BMO space\footnote{This means replacing standard Euclidean cubes by parabolic cubes.}, where
	\[D_t^{1/2}f(x,t): = c\  p.v.\int_{\mathbb{R}} \frac{f(x,s) - f(x,t)}{|s-t|} \, ds\]
	and $c$ is an absolute constant. Here, $D_t^{1/2}f$ should be viewed as an element of $\mathcal{D}_0'$, where $\mathcal{D}_0$ is the collection of smooth compactly supported functions with zero mean. The Carleson measure property for Lip(1,1/2) functions is equivalent to this. Quantitatively, denoting $\mu_f \coloneqq \gamma_f(\Xbf,r)^2 r^{-1} d\Xbf dr$, we have the estimates 
	\begin{equation}\label{eq:dtbdbycarl}
		\| D_t^{1/2}  f \|_{\PBMO} \lesssim_n L + \| \mu_f \|_{\mathcal{C}}^{1/2}
	\end{equation}
	and 
	\begin{equation}\label{eq:carlbdbydt}
	\| \mu_f \|_{\mathcal{C}}^{1/2}  \lesssim_n L + 	\| D_t^{1/2}  f \|_{\PBMO} , 
	\end{equation}
	where $\|\cdot\|_{\PBMO}$ is the parabolic BMO norm and $\|\mu_f\|_{\mathcal{C}}$ is the Carleson norm of $\mu_f$, i.e., the smallest $M$ for which \eqref{e:gamma-carleson} holds. 
	
The bound \eqref{eq:dtbdbycarl} can be deduced from the methods in \cite{HLN1,HLN2}, see \cite[Lemma 5.5]{BHHLN-corona}. The bound \eqref{eq:carlbdbydt} can be found in \cite[Section 5]{Hof-SIO}; however, we should note that in that manuscript Hofmann works with objects called $D_n$ and $D_{par}$, which are not $D_t^{1/2}$. In particular, in \cite[Section 5]{Hof-SIO} it is shown that $\| \mu_f \|_{\mathcal{C}}^{1/2}$ is bounded by $\|D_{par}\|_{\PBMO}$ up to a dimensional constant.  On the other hand, in \cite[Section 3]{Hof-SIO} it is observed that
\[D_{par}  = \sum_{j =1}^{n-1} \partial_{x_j}R_j + D_nR_n,\]
where $R_j$ are certain {\it parabolic} Riesz transforms, that is, the Fourier symbol of $R_j$ is  $\xi_j/\|(\xi, \tau)\|$ for $j \in \{1,2,\dots, n-1\}$ and the Fourier symbol of  $R_n$ is $\tau/\|(\xi, \tau)\|^2$ (here, the norm $\|\cdot\|$ is different, but equivalent to the one we use, so that the identity above holds). Since $R_j$ maps $\PBMO$ to $\PBMO$, it holds that
\[\|D_{par} f\|_{\PBMO} \lesssim \sum_{j = 1}^{n-1}\|\partial_{x_i} f\|_{\PBMO} + \|D_n f\|_{\PBMO} \lesssim \|\nabla  f\|_{L^\infty} + \|D_n f\|_{\PBMO}.\]
The final step is then to use 
\cite[Equation (0.19)]{HL-ann}, which states
\[ \|\nabla  f\|_{L^\infty} + \|D_n f\|_{\PBMO} \approx \|\nabla  f\|_{L^\infty} + \|D_t^{1/2} f\|_{\PBMO} .\] 
\end{remark}

Now we define the notion of regular parabolic (bi)-Lipschitz \textit{maps}.

\begin{definition}[Parabolic regular (bi-)Lipschitz maps]\label{d:p-regular-bi-lip}
We say $F: \mathbb{R}^n \to \mathbb{R}^{n+1}$, defined as
\[F(x,t) = (F_1(x,t), \dots, F_n(x,t), F_{n+1}(x,t)),\]
is \textit{parabolic regular Lipschitz} if there exists a constant $L\ge 0$ such that $F$ is $L$-Lipschitz, that is,
\begin{align}
	\dist(F(x,t), F(y,s)) \leq L \dist((x,t),(y,s)) 
\end{align}
for all $(x,t), (y,s) \in \mathbb{R}^n$, if the action of $F$ on the $t$-coordinate is translation, that is,
\[F_{n+1}(x,t) = c + t,\]
and $F_1,\dots, F_n$ are regular Lip(1,1/2) functions. We say $F$ is an \textit{$(L,M)$-parabolic regular Lipschitz map} if it is $L$-Lipschitz and each $F_i$, $i = 1,\dots, n$, is $(L,M)$- regular parabolic Lipschitz. We define \textit{regular parabolic bi-Lipschitz} and \textit{$(L,M)$-regular parabolic bi-Lipschitz} similarly, except that we require $F$ to be $L$-bi-Lipschitz, that is, 
\[L^{-1}\dist((x,t),(y,s)) \le \dist(F(x,t), F(y,s)) \le L \dist((x,t),(y,s)). \]
\end{definition}

\begin{remark}[Graphs and Images of Planes]
We shall often write that a function $F:P \to P^\perp$ is a regular Lip(1,1/2) function or a map $F:P \to \mathbb{R}^{n+1}$ is a regular  bi-Lipschitz map, where $P$ is a $t$-independent plane. In either case, these are viewed in the most natural way, that is, by verifying that the requisite properties hold after performing a change of coordinates. It is easy to see that these transformations (in the domain or co-domain) can only increase the Lip(1,1/2) character and the Carleson norm of the $\gamma$ numbers of the coordinate functions by at most dimensional constants. Since these constants have no bearing on the conclusions of our results, we do not carefully track them.
\end{remark}

\begin{definition}[Big pieces of regular parabolic (bi-)Lipschitz images]\label{BPRPBI.def}
Let $\Sigma \subset \mathbb{R}^{n+1}$ be ADR. 
\begin{enumerate}
	\item We say $\Sigma$ has \textit{big pieces of regular parabolic Lipschitz images}, or $\Sigma \in BPRPLI$, if there exist constants $L, M \geq 0$ and $\theta > 0$ such that, for every $(X,t) \in \Sigma$ and $r> 0$, there exists an $(L,M)$-regular parabolic Lipschitz map $F_{X,t,r} \colon \RR^n \to \RR^{n+1}$ such that
	\[\sigma(F_{X,t,r}(B_n(\mathbf{0},r))) \cap B((X,t),r)) \ge \theta r^d,\]
	where $B_n(\mathbf{0},r)$ is the parabolic ball in $\RR^n$ centered at the origin with radius $r$. 
	
	\item We say $\Sigma$ has \textit{big pieces of regular parabolic bi-Lipschitz images}, or $\Sigma \in BPRPBI$, if there exist constants $L, M \geq 0$ and $\theta > 0$ such that, for every $(X,t) \in \Sigma$ and $r> 0$, there exists an $(L,M)$-regular parabolic Lipschitz map $F_{X,t,r} \colon \RR^n \to \RR^{n+1}$ such that
	\[\sigma(F_{X,t,r}(\RR^n) \cap B((X,t),r)) \ge \theta r^d.\]
\end{enumerate}
Equivalently, 

\begin{enumerate}
	\item [(1')] $\Sigma \in BPRPLI$ if there exist constants $L, M \geq 0$ and $\theta > 0$ such that, for every $Q \in \dd$, there exists an $(L,M)$-regular parabolic Lipschitz map $F_{X,t,r} \colon \RR^n \to \RR^{n+1}$ such that
	\[\sigma(F_{Q}(B_n(\mathbf{0},\ell(Q)))) \cap Q) \ge \theta \sigma(Q).\]

	\item [(2')] $\Sigma \in BPRPBI$ if there exist constants $L, M \geq 0$ and $\theta > 0$ such that, for every $Q \in \dd$, there exists an $(L,M)$-regular parabolic bi-Lipschitz map $F_{X,t,r} \colon \RR^n \to \RR^{n+1}$ such that
	\[\sigma(F_{Q}(\RR^n) \cap Q) \ge \theta \sigma(Q).\]
\end{enumerate}
\end{definition}

\begin{definition}[Parabolic isometry]
Write
\[
    \mathbb R^{n+1}=\mathbb R^n \times\mathbb R
\]
and equip \(\mathbb R^{n+1}\) with the parabolic metric from \eqref{e:parabolic-dist}. A \emph{parabolic isometry} is an affine map
\[
     A:\mathbb R^{n+1}\to\mathbb R^{n+1}
\]
of the form
\[
     A(X,t)
    =
    \bigl(OX+b,\ t+\tau\bigr),
\]
where \(O\in O(n)\) is an orthogonal transformation,
\(b\in\mathbb R^n\), and \(\tau\in\mathbb R\).

Thus, \( A\) acts on the spatial variables by a Euclidean
isometry and acts on the time variable by translation. In particular,
\[
    \dist\bigl( A(X,t), A(Y,s)\bigr)
    =
    \dist\bigl((X,t),(Y,s)\bigr)
\]
for every \((X,t),(Y,s)\in\mathbb R^{n+1}\).
\end{definition}

To discretize our analysis, we use a standard construction of ``dyadic cubes." Such constructions are now ubiquitous in harmonic analysis,
and for the following lemma we refer to \cite{Christ-cubes,David-cubes,HK-cubes}.

\begin{lemma}\label{cubes}  Assume that $\Sigma  \subset \mathbb R^{n+1}$ is (parabolic) ADR  in the sense of Definition \ref{ADR.def} with constant $M$. Then $\Sigma$ admits a parabolic dyadic decomposition in the sense that there exist positive, finite
constants $c_0$, $\zeta$, and $c_*\,$, depending only on dimension
and the ADR constant, such that the following holds. For each $k \in \mathbb{Z}$,
there exists a collection of Borel sets, $\mathbb{D}_k$,  which we will call (dyadic) cubes, such that
$$
\mathbb{D}_k:=\{Q_{j}^k\subset \Sigma: j\in \mathfrak{I}_k\},$$ where
$\mathfrak{I}_k$ denotes some (countable)  index set depending on $k$, and such that the following hold:
\begin{list}{$(\theenumi)$}{\usecounter{enumi}\leftmargin=1cm \labelwidth=1cm \itemsep=0.2cm \topsep=.2cm \renewcommand{\theenumi}{\roman{enumi}}}
\item $\Sigma=\cup_{j}Q_{j}^k\,$ for each $k\in{\mathbb Z}$.

\item If $m\geq k$ then either $Q_{i}^{m}\subseteq Q_{j}^{k}$ or $Q_{i}^{m}\cap Q_{j}^{k}=\emptyset$.

\item For each $(j,k)$ and each $m<k$, there is a unique $i$ such that $Q_{j}^k\subset Q_{i}^m$.
\item $\diam\big(Q_{j}^k\big)\leq c_* 2^{-k}$.
\item Each $Q_{j}^k$ contains $\Sigma\cap B(Z^k_{j},t^k_j, c_02^{-k})$ for some $(Z^k_{j},t^k_j)\in \Sigma$.
\item $\sigma\left(\left\{(Z,t)\in Q^k_j: \dist(Z,t,\Sigma\setminus Q^k_j)\leq \varrho \,2^{-k}\right\}\right)\leq c_*\,\varrho^\zeta\,\sigma(Q^k_j),$
for all $k,j$, and all $\varrho\in (0,1)$.
\end{list}
\end{lemma}

\begin{remark}\label{remarkddcubes} We denote by
$\mathbb{D}=\mathbb{D}(\Sigma)$ the
collection of all $Q^k_j$, i.e.,
$$\mathbb{D} := \cup_{k} \mathbb{D}_k.$$ For a dyadic cube
$Q = Q^k_j \in \mathbb{D}_k$, we set 
\[\Xbf_Q \coloneqq (X_Q,t_Q) \coloneqq (Z_j^k,t_k^j)\quad \mbox{ and } \quad \ell(Q)\coloneqq 2^{-k},\]
which we will refer to as the \textit{center} and \textit{size}
of $Q$, respectively.  Evidently, $\ell(Q)\sim\diam(Q)$. Furthermore,
given $Q\in \mathbb{D}$ and $k\in\mathbb Z_+$, we let
$Q'\in \mathbb{D}_k(Q)$ if $Q'\subset Q$ and $\ell(Q') := 2^{-k}\ell(Q)$,
and we set $\dd(Q):=\{Q'\in\dd(\Sigma): Q'\subset Q\}$.
\end{remark}

Given $Q \in \mathbb{D}(\Sigma)$ and $\lambda > 0$, we set
\[\lambda Q = \{\Xbf: \dist(\Xbf, Q) \le (\lambda - 1)\diam{Q}\}.\]
On the other hand, if $Q \in \mathbb{D}(\mathbb{R}^n)$ or $Q \in \mathbb{D}(\mathbb{R}^{n+1})$ with $Q = Q(\Xbf,r)$ and $\lambda > 0$, we set
\[\lambda Q = Q(\Xbf,\lambda r)\]
to be the usual dilate of $Q$.

\begin{definition}[\cite{DS-AMS}]\label{d3.11}
Suppose $E$ is ADR with dyadic cubes $\dd(E)$. Let $\sbf\subset \dd(E)$. We say that $\sbf$ is
\textit{coherent} if the following conditions hold:
\begin{itemize}\itemsep=0.1cm

\item[$(a)$] $\sbf$ contains a unique maximal element $Q(\sbf)$ which contains all other elements of $\sbf$ as subsets.

\item[$(b)$] { If $Q$  belongs to $\sbf$, $\widetilde{Q} \in \mathbb{D}$ and  $Q\subset \widetilde{Q}\subset Q(\sbf)$, then $\widetilde{Q}\in {\bf S}$.}

\item[$(c)$] Given a cube $Q\in \sbf$, either all of its children belong to $\sbf$, or none of them do.

\end{itemize}
We say that $\sbf$ is \textit{semi-coherent} if only conditions $(a)$ and $(b)$ hold.
\end{definition}

It will be convenient for us to discretize the Carleson conditions for the $\gamma$ and $\beta$ numbers in terms of dyadic decompositions. The following is well-known \cite{DS-AMS}, see for example \cite{Rigot}.

\begin{lemma}
	Let $f \colon \RR^n \to \RR$ be parabolic Lipschitz and $\Sigma \subset \RR^{n+1}$ be ADR with constant $C_A$, and let $\dd = \dd(\RR^n)$ and $\dd' = \dd(\Sigma)$. Then, we have the following equivalences.
	\begin{enumerate}
		\item The inequality \eqref{e:gamma-carleson} holds with constant $M$ if and only 
		\begin{align}
				\sum_{Q \in \dd(Q_0)} \gamma_f(\Xbf_Q,3\ell(Q))^2\ell(Q)^{n+1} \leq M' \ell(Q_0)^{n+1} \quad \mbox{ for all } Q_0 \in \dd, 
		\end{align}
		where $M,M'$ satisfy $M/M' \sim_n 1$. In particular, $f$ is regular parabolic Lipschitz if and only if the above property holds. \\
		
		\item The inequality \eqref{betacarldef.eq} holds with constant $M$ if and only 
		\begin{align}
			\sum_{Q \in \dd'(Q_0)} \beta(\Xbf_Q,3\ell(Q))^2\ell(Q)^{n+1} \leq M' \ell(Q_0)^{n+1} \quad \mbox{ for all } Q_0 \in \dd', 
		\end{align}
		where $M,M'$ satisfy $M/M' \sim_{C_A,n} 1$. In particular, $\Sigma$ is P-UR if and only if the above property holds. 
	\end{enumerate}
	We refer to the above conditions as ``geometric lemmas'' for the remainder of the paper. 
\end{lemma}

It is known \cite{BHHLN-BP, BHHLN-corona} that a set is P-UR if and only if it admits a (bilateral) corona decomposition by regular parabolic Lipschitz graphs. We will only need one direction of this characterization.

\begin{theorem}[{\cite[Theorem 3.2]{BHHLN-corona}}]\label{PURiffcorona.thrm} 
Let $\Sigma \subset \mathbb{R}^{n+1}$ be an P-UR set with ADR constant $C_A$ and let $M_\Sigma$ be the smallest constant $M$ so that \eqref{betacarldef.eq} holds.
There exists  $\kappa =\kappa(n,C_A, M_\Sigma)$ and  $M = M(n,C_A, M_\Sigma)$ such that, for every $L \in (0,1]$, the following holds. There exists a disjoint decomposition,
$\dd(\Sigma) = \B \cup\G$, satisfying the following properties.
\begin{enumerate}
\item The collection $\B$ satisfies the Carleson
packing condition:
$$\sum_{Q'\subseteq Q, \,Q'\in\B} \sigma(Q')
\leq\, c(n,M,L)\, \sigma(Q)\,,
\quad \forall Q\in \dd(\Sigma)\,.$$
\item  The collection $\G$ is further subdivided into
disjoint stopping time regimes $\mathcal{S}$, such that each such regime $\sbf \in \mathcal{S}$ is coherent. 

\item The maximal cubes $Q(\sbf)$ for the stopping time regimes satisfy the Carleson
packing condition, that is, if 
\[\mathcal{M} = \{Q(\sbf)\}_{\sbf \in \mathcal{S}}\]
then
$$\sum_{Q(\sbf) \in \mathcal{M}: Q(\sbf) \subseteq Q}\sigma\big(Q(\sbf)\big)\,\leq\, c(n,M,L)\, \sigma(Q)\,,
\quad \forall Q\in \dd(\Sigma)\,.$$
\item For each $\sbf$, if $\sbf'$ is a semi-coherent stopping time regime such that $\sbf' \subseteq \sbf$,
 there exists a coordinate system $P^\perp \times P$ for $P$ a $t$-independent plane and an
$(L,M)$-regular Lip(1,1/2) function $ \psi_{\sbf'}:=\psi : \mathbb R^{n-1}\times\mathbb R\to \mathbb R$
 such that, if we define
 $\Gamma_{\sbf'}:=\Gamma_\psi:=\{(\psi(x,t),x,t):\ (x,t)\in \mathbb R^{n-1}\times\mathbb R\}$, then
\begin{equation}\label{closegraphcorona.eq}
\sup_{(X,t)\in 10 Q} \dist(X,t,\Gamma_{\sbf'} ) +   \sup_{(Y,s)\in B((X_Q,t_Q),10 \diam(Q)) \cap \Gamma_{\sbf'}} \dist(Y,s, \Sigma) 
\leq  L \diam(Q),
\end{equation}
for all $Q\in \sbf'$.
Moreover, we have that 
\[P \cap B((X_{Q(\sbf')},t_{Q(\sbf')}),\kappa \diam({Q(\sbf')})) \neq \emptyset\]
and
\[\supp(\psi) \subseteq P \cap B((X_{Q(\sbf')},t_{Q(\sbf')}),\kappa \diam({Q(\sbf')})).\]
\end{enumerate}
\end{theorem}
\begin{remark} 
Item (4) is not explicitly stated in \cite[Theorem 3.2]{BHHLN-corona}. The proof in \cite[Theorem 3.2]{BHHLN-corona} (which is really re-running \cite[Theorem 3.1]{BHHLN-corona} with additional, bilateral approximation) relies on a Carleson estimate on the $\beta$ numbers that continues to hold for every sub-regime of $\sbf$. Moreover, the statement concerning
\[``P \cap B((X_{Q(\sbf')},t_{Q(\sbf')}),\kappa \diam({Q(\sbf')})) \neq \emptyset\]
and
\[\supp(\psi) \subseteq P \cap B((X_{Q(\sbf')},t_{Q(\sbf')}),\kappa \diam({Q(\sbf')}))''\]
does not appear in \cite[Theorem 3.2]{BHHLN-corona} explicitly, but can be easily deduced from the construction.  
\end{remark}

\section{A gluing lemma for regular Lipschitz functions}\label{s:gluing-lemma}

At several points in the proof of Theorem \ref{main.thrm}, we will need to glue families of functions which are locally regular parabolic Lipschitz and supported on a well-separated collection of cubes. In this section, we formulate a lemma via which this will be implemented. First, we need to say what we mean by \textit{locally} regular parabolic Lipschitz.

\begin{definition}\label{d:local-regular-lipschitz}
	Let $L,M \geq 0$ and $Q_0 \in \Delta$.  We say that a function $\theta \colon \RR^n \to \RR$ is \textit{$(L,M)$-regular parabolic Lipschitz in $10Q_0$} if it is parabolic $L$-Lipschitz and satisfies 
	\begin{align}\label{e:loca-regular-lipschitz}
		\sum_{Q \subseteq 4Q_1} \gamma_\theta(3Q)^2 \ell(Q)^{n+1} \leq  M \ell(Q_1)^{n+1}
	\end{align}
	for all $Q_1 \in \Delta$ such that $6Q_1 \subseteq 10Q_0$.
\end{definition}

The main result of this section is the following. 

\begin{lemma}\label{l:gluing-for-regular-functions-prime}
	Let $K,L,M \geq 0$. Let $\{Q_i\}_{i \in I} \subseteq \dd = \dd(\RR^n)$ be a collection of cubes satisfying the separation condition 
	\begin{align}\label{e:separation-gluing-prime}
		d_\infty(Q_i, Q_j) \ge 10^5\max\{\diam_\infty(Q_i), \diam_\infty(Q_j)), \quad \forall i \neq j,
	\end{align}
	where $\diam_\infty$ is the diameter measured in $d_\infty$. Let $\{h_i\}_{i \in I}$ be a collection of functions such that at least one of the following conditions holds: 
	\begin{enumerate}
		\item [(H1)] For each $i \in I$, the function $h_i$ is $(L,M)$-regular parabolic Lipschitz in $10Q_i$ and satisfies $\supp(h_i) \subseteq 3Q_i$. \\
		\item [(H2)] For each $i \in I$, we have 
		\begin{align*}
			h_i = \eta_i \theta_i, 
		\end{align*}
		where $\eta_i$ satisfies 
			\begin{align*}
			0 \leq \eta_i \leq 1, \  \supp(\eta_i) \subseteq 3Q_i \ \mbox{ and } \ 	|\nabla_x^\alpha \partial_t^m\eta_i|
			\le C_{\alpha,m} \ell(Q_i)^{-|\alpha|-2m}, 
		\end{align*} 
		and $\theta_i$ is $(L,M)$-regular parabolic Lipschitz such that 
		\begin{align*}
				|\theta_i| \leq K \ell(Q_i) \quad \mbox{ on } 10Q_i. 
		\end{align*}
	\end{enumerate} 
	Then, the function 	
	\begin{align*}
		h \coloneqq \sum_{i \in I} h_i 
	\end{align*} 
	is regular parabolic Lipschitz. Under the assumptions of (H1), this holds with constants $(C_nL,C_n(M +L^2))$; under the assumptions of (H2), this holds with constants $(C_n(L+K),C_n(M+L^2+K^2))$. 
\end{lemma}

\begin{proof}
	The proof is very similar whether we assume (H1) or (H2); however, (H2) is the slightly more complicated case which we will focus on. Once we have completed the proof under this assumption, we will describe the necessary changes for the (H1) case.  For brevity, let us denote 
	\begin{align}
		r_i \coloneqq \ell(Q_i) \quad \mbox{ for } i \in I. 
	\end{align}
	
	We start by showing that $h$ is parabolic $C(L+K)$-Lipschitz. Indeed, it follows by assumption that $h_i$ is a Lip(1,1/2) function with
	\begin{align}\label{e:h_i-Lip}
		\|h_i\|_{\rm{Lip}(1,1/2)} \lesssim L + K .
	\end{align}
	We take two points $(x,t)$ and $(y,s)$ and break into cases to prove that $h$ is Lip(1,1/2). If $(x,t),(y,s) \in (\cup_i 3Q_i)^c$, then $h(x,t) = h(y,s) = 0$ and there is nothing to show.  If $(x,t) \in Q_i$ for some $i$ and $(y,s) \in Q_i \cup ( \cup_j 3Q_j)^c$ then $h(x,t) = h_i(x,t)$ and $h(y,s) = h_i(y,s)$, and the Lipschitz bound for $h$ follows from that of $h_i$. Finally, if $(x,t) \in Q_i$ and $(y,s) \in Q_j$ with $i \neq j$ then, by definition of $h_i$, it holds that
	\begin{align}\label{e:size-h}
		|h(x,t)| \le Kr_i \quad \mbox{ and } \quad |h(y,s)| \le K r_j,
	\end{align}
	and the separation of the $Q_i$'s, \eqref{e:separation-gluing-prime}, gives 
	\[\dist((x,t),(y,s)) \gtrsim_n \max\{r_i,r_j\}.\] 
	Thus,
	\[|h(x,t) - h(y,s)| \le 2K \max\{r_i,r_j\} \lesssim K \dist((x,t),(y,s)).\]

	We will now check that $h$ is regular parabolic Lipschitz with the above stated constant. Fix $Q_0 \in \dd$. We are required to show
	\begin{align}\label{e:h-Carleson} 
		\sum_{Q \in \dd(Q_0)} \gamma_h(3Q)^2\ell(Q)^{n+1} \le C_n (M + L^2 + K^2) \ell(Q_0)^{n+1} ,
	\end{align} 
	where the implicit constant is independent of $Q_0$. We will split the sum into three separate pieces. First, let $\{R_j\}$ be a collection of maximal cubes in $\dd(Q_0)$ such that 
	\begin{align*}
		3R_j \cap \bigcup_{i} 3Q_i = \emptyset. 
	\end{align*}
	By definition, $h \equiv 0$ outside of $\bigcup_i 3Q_i$. In particular, we have 
	\begin{align}\label{e:R_j}
		\sum_j \sum_{Q \subseteq R_j} \gamma_h(3Q)^2 \ell(Q)^{n+1} = 0. 
	\end{align}
	Thus, it only remains to control the sum over cubes in $S \coloneqq \dd(Q_0) \setminus \bigcup_j \dd(R_j)$. Below, we will assume $S \neq \emptyset$ or else there is nothing to show. We split this into two families of cubes: those which are small with respect to the $Q_i$ which intersect it, and those which are large. Indeed, let 
	\begin{align*}
		S_1 &\coloneqq \{Q \in S \colon \mbox{ there exists } i \mbox{ such that } 3Q \cap 3Q_i \neq \emptyset \mbox{ and } \ell(Q_i) \geq 3\ell(Q) \}; \\
		S_2 &\coloneqq S \setminus S_1. 
	\end{align*}  
	
	We immediately note the following. 
	
	\noindent\textbf{Claim 1:} 
	For each $Q \in S_1$ there is a unique $i \in I$ such that $6Q \cap 3Q_i \neq \emptyset$, if such an index exists. In the case that it does exist, we have
	\begin{align}\label{e:in-5-time}
		Q \subseteq 4Q_{i} \quad \mbox{ and } \quad  6Q \subseteq 10Q_{i} . 
	\end{align}

		Indeed, let us assume that such an $i \in I$ exists (otherwise the claim is vacuous). The estimates in \eqref{e:in-5-time} follow directly from the triangle inequality and the definition of $S_1$. The uniqueness follows from the first inequality in \eqref{e:in-5-time} and the separation condition \eqref{e:separation-gluing-prime}. 
	
	Our first goal is to show
	\begin{align}\label{e:S_1}
		\sum_{Q \in S_1} \gamma_h(3Q)^2\ell(Q)^{n+1} \lesssim (M + L^2 + K^2) \ell(Q_0)^{n+1}. 
	\end{align}
	We claim that it suffices to show 
	\begin{align}\label{e:S_1-prelim}
		\sum_{\substack{Q \subseteq 4Q_1}} \gamma_h(3Q)^2\ell(Q)^{n+1} \lesssim (M + L^2 + K^2) \ell(Q_1)^{n+1} \mbox{ for all } Q_1 \in S_1. 
	\end{align}
	Indeed, if $Q_0 \in S_1$, then clearly \eqref{e:S_1-prelim} implies \eqref{e:S_1}. Suppose instead that $Q_0 \in S_2$. Let $Q \in S_1$ and let $i$ such that $3Q \cap 3Q_i \neq \emptyset$. Since $Q \subseteq Q_0$, we know that $3Q_0 \cap 3Q_i \neq \emptyset$, which (since $Q_0 \in S_2$) implies 
	\begin{align}\label{e:5Q_0}
		Q_i \subseteq 20Q_0.
	\end{align}
	In particular, since we are assuming \eqref{e:S_1-prelim} for each $Q_i$ (which is clearly in $S_1$), it follows from \eqref{e:in-5-time}, \eqref{e:5Q_0} and the fact that cubes $\{Q_i\}$ are disjoint, that  
	\begin{align*}
		\sum_{Q \in S_1} \gamma_h(3Q)^2\ell(Q)^{n+1} &\leq \sum_{Q_i \subseteq 20Q_0} \sum_{Q \subseteq 4Q_i} \gamma_h(3Q)^2\ell(Q)^{n+1} \\
		&\lesssim (M + L^2 + K^2) \sum_{Q_i \subseteq 20Q_0} \ell(Q_i)^{n+1} \\
		&\lesssim (M + L^2 + K^2) \ell(Q_0)^{n+1}, 
	\end{align*}  
	as required.

	We now prove \eqref{e:S_1-prelim}. Let $Q_1 \in S_1$ and fix, momentarily, $Q \subset 4Q_1$. We seek an estimate on $\gamma_h(3Q)^2$ and, therefore, we may assume $3Q \cap 3Q_i \neq \emptyset$ (since otherwise $\gamma_h(3Q)^2 = 0$).  Since each $\eta_j$ is supported on $3Q_j$, it follows from the separation condition \eqref{eq:sepcubesprop} and the second inclusion in \eqref{e:in-5-time} that 
	\begin{align}\label{e:reference}
		h \equiv \eta_i \theta_i \quad \mbox{ on } 3Q. 
	\end{align}
	Note, by hypothesis,
	\[|\theta_i| \le K r_i, \quad \text{ on } 3Q.\]
	
	Let $A_1,A_2 \colon \RR^n \to \RR$ be the infimising affine function for $\gamma_{\eta_i}(3Q)$ and $\gamma_{\theta_i}(3Q)$. Since $\eta_i$ and $\theta_i$ are Lip(1,1/2) with constants $C r_i^{-1}$ and $L$, respectively, we have that $|\nabla A_1| \lesssim r_i^{-1}$ and $|\nabla A_2| \lesssim L$. By Taylor's Theorem, there is a third affine function $A_3 \colon \RR^n \to \RR$ such that 
	\begin{align*}
		| A_1A_2(x) - A_3(x) | \lesssim Lr_i^{-1}|x-x_Q|^2 \lesssim L r_i^{-1} \ell(Q)^2 
	\end{align*}
	for all $(x,t) \in 3Q$. Using this, \eqref{eq:closeness} and \eqref{eq:eta}, we now have
	\begin{align*}
		|h(x,t) - A_3(x)| &\leq |\eta_i(x,t)\theta_i(x,t) - A_1(x)A_2(x)| + |A_1(x)A_2(x) - A_3(x)|  \\
		&\leq |\eta_i(x,t) - A_1(x)||\theta_i(x,t)| + \||A_1\|_{L^\infty(3Q)}|\theta_i(x,t) - A_2(x)| +CLr_i^{-1}\ell(Q)^2 \\
		&\lesssim |\eta_i(x,t) - A_1(x)| Kr_i + |\theta_i(x,t) - A_2(x)| +CLr_i^{-1}\ell(Q)^2
	\end{align*}
	for each $(x,t) \in 3Q$. Here we used that $\|A_1\|_{L^\infty(3Q)} \lesssim 1$, since $A_1$ is the $L^2(3Q)$ projection of $\eta_i$ onto the space of spatially affine functions, $\|\eta\|_{Lip(1,1/2)} \lesssim r_i^{-1}$ and $\|\eta\|_{L^\infty} \le 1$. After dividing by $\ell(Q)$, squaring, and taking the average integral of $3Q$, the above estimate implies 
	\begin{align}\label{e:omega-est}
		\gamma_h(3Q)^2 \lesssim K^2 r_i^{2} \gamma_{\eta_i}(3Q)^2 + \gamma_{\theta_i}(3Q)^2 + L^2r_i^{-2}\ell(Q)^2. 
	\end{align}
	
	Since \eqref{e:omega-est} holds for arbitrary $Q \subseteq 6Q_1$, and $\eta_i$ and $\theta_i$ are regular parabolic Lipschitz with constant $(Cr_i^{-1},Cr_i^{-2})$ and $(L,M)$, respectively, we see that
	\begin{align*}
		\sum_{Q \subseteq 4Q_1}  \gamma_h(3Q)^2 \ell(Q)^{n+1} &\lesssim K^2r_i^2\sum_{Q \subseteq 4Q_1} \gamma_{\eta_i}(3Q)^2\ell(Q)^{n+1} +  \sum_{Q \subseteq 4Q_1} \gamma_{\theta_i}(3Q)^2\ell(Q)^{n+1} \\
		&\hspace{2em} + L^2r_i^{-2}\sum_{Q \subseteq 4Q_1} \ell(Q)^{n+3} \lesssim (K^2 +M + L^2)\ell(Q_1)^{n+1}
	\end{align*}
	(note, the third term is estimated by summing a geometric series), as required for \eqref{e:S_1}.  
	
	We now deal with the sum over $S_2$. We can assume that $Q_0 \in S_2$ as otherwise $S_2 = \emptyset$ and there is nothing to show. In this case, similar to how we argued for \eqref{e:in-5-time} and \eqref{e:5Q_0}, we have 
	\begin{align}\label{e:i-in-Q_0}
		3Q_i \subseteq 50Q_0  \mbox{ for all } i \mbox{ such that } 3Q_i \cap 3Q_0 \neq \emptyset. 
	\end{align} 
	By \eqref{eq:closeness} and \eqref{eq:eta}, we have $\eta_i\theta_i \le C_nK r_i$ on $3Q_i$ for each $i$.	Since $h \equiv 0$ outside of $\bigcup_i 3Q_i$, this implies
	\begin{align}
		\begin{split}\label{e:S_2-prime}
			\sum_{Q \in S_2} \gamma_h(3Q)^2\ell(Q)^{n+1} &= \sum_{Q \in S_2} \sum_{\substack{i \\ 3Q_i \cap 3Q \neq\tiny{\emptyset}}} \int_{3Q_i} \left( \frac{\eta_i\theta_i}{\ell(Q)} \right)^2 \, d\cH^{n+1}_{\rm par} \\
			&\lesssim K^2\sum_{Q \in S_2}   \sum_{\substack{i \\ 3Q_i \cap 3Q \neq\tiny{\emptyset}}} \left(\frac{r_i}{\ell(Q)} \right)^2 r_i^{n+1} \\
			&\leq K^2 \sum_{\substack{i \\ 3Q_i \subseteq 50Q_0}} r_i^{n+1} \sum_{\substack{Q \in S_2 \\ 3Q_i \cap 3Q \neq\tiny{\emptyset}}} \left(\frac{r_i}{\ell(Q)} \right)^2,
		\end{split}
	\end{align}
	where in the final inequality we used \eqref{e:i-in-Q_0}. Fix some $i$ such that $3Q_i \subseteq 50Q_0$ and let $k \in \ZZ$ such that $Q_i \in \dd_k$. By the definition of $S_2$, 
	\begin{align*}
		\{Q \in S_2 \colon 3Q_i \cap 3Q \neq \emptyset \} \subseteq \bigcup_{m=-\infty}^k \dd_m.
	\end{align*}
	By doubling, 
	\begin{align*}
		|\{Q \in S_2 \cap \dd_m \colon 3Q_i \cap 3Q \neq \emptyset \}| \lesssim 1 \quad \mbox{ for all } m \leq k. 
	\end{align*}
	Thus, we can estimate
	\begin{align*}
		\sum_{\substack{Q \in S_2 \\ 3Q_i \cap 3Q \neq\emptyset}}  \left(\frac{r_i}{\ell(Q)} \right)^2 = \sum_{m=-\infty}^k \sum_{\substack{Q \in S_2 \cap \dd_m \\ 3Q_i \cap 3Q \neq\tiny{\emptyset}}} 2^{-m+k} \lesssim \sum_{m=-\infty}^k 2^{-m+k}  \lesssim 1. 
	\end{align*}
	Plugging this estimate back in \eqref{e:S_2-prime} gives 
	\begin{align}\label{e:S_2}
		\sum_{Q \in S_2} \gamma_h^2(3Q)^2\ell(Q)^{n+1} \lesssim K^2 \sum_{\substack{i \\ 3Q_i \subseteq 50Q_0}} r_i^{n+1} \lesssim K^2 \ell(Q_0)^{n+1}. 
	\end{align}
	
	The estimates \eqref{e:R_j}, \eqref{e:S_1} and \eqref{e:S_2} complete the proof of \eqref{e:h-Carleson}.

	We have now verified the conclusion of the lemma under the assumption (H2). Let us describe how to modify the proof to obtain the conclusion under the assumption (H1). The proof that $h$ is $CL$ Lipschitz essentially the same as in the case (H2) except that the estimate \eqref{e:h_i-Lip} becomes 
	\begin{align}
		\|h_i\|_{Lip(1,1/2)} \lesssim L
	\end{align}
	and \eqref{e:size-h} becomes 
	\begin{align}
			|h(x,t)| \lesssim Lr_i \quad \mbox{ and } \quad |h(y,s)| \lesssim L r_j. 
	\end{align} 
	To check that $h$ is regular parabolic Lipschitz one proves \eqref{e:h-Carleson} with constant $C_n(M+L^2)$. Splitting again into $S_1$ and $S_2$, this follows (by the same argument as after \eqref{e:S_1-prelim}) once we have the estimates  
	\begin{align*}
		\sum_{Q \subseteq 4Q_1} \gamma_h(3Q)^2\ell(Q)^{n+1} \lesssim M \ell(Q_1)^{n+1} \mbox{ for all } Q_1 \in S_1. 
	\end{align*}
	and 
	\begin{align*}
		\sum_{Q \in S_2} \gamma_h(3Q)^2\ell(Q)^{n+1} \lesssim K^2 \ell(Q_0)^{n+1}. 
	\end{align*}
	Since $h$ is supported in $\cup_i 3Q_i$ and $3Q \subseteq 6Q_1$ for any $Q \subseteq 4Q_1$ with $\ell(Q) \leq \ell(Q_1)$, the first estimate is trivial if $6Q_1 \cap \cup_i 3Q_i = \emptyset$. If $6Q_1 \cap \cup_i 3Q_i \neq \emptyset$, it follows from Claim 1 that there is a unique $i \in I$ such that $6Q_1 \cap 3Q_i \neq \emptyset$ and, moreover, $6Q_1 \subseteq 10Q_i$. In particular, $\gamma_h(3Q) = \gamma_{h_i}(3Q)$ for any $Q \subseteq 4Q_1$. The first estimate now follows directly from the fact that $h_i$ is $(L,M)$-regular in $10Q_i$. The second estimate follow from the same argument as in the (H2) case (up to notational changes). 
\end{proof}

\section{BPRPLI implies P-UR}

The goal for this section is to prove the following.

\begin{theorem}\label{t:BPRPLI-implies-UR}
	Suppose $\Sigma \subseteq \RR^{n+1}$ is parabolic Ahlfors-David regular and has big pieces of regular parabolic Lipschitz images of $n$-dimensional space, then $\Sigma$ is parabolic uniformly rectifiable.
\end{theorem}

Before getting to this, we need a quantitative bi-Lipschitz decomposition result for parabolic Lipschitz functions in the spirit of Jones \cite{jones1988lipschitz}. We provide this below.

\subsection{Bi-Lipschitz decompositions of regular parabolic maps} 

In \cite{jones1988lipschitz}, Jones proved that the domain of any Lipschitz map $f \colon [0,1]^n \to \RR^m$ ($m \geq n$) can be decomposed into finitely many pieces on which $f$ is bi-Lipschitz and a garbage set whose image under $f$ has small Hausdorff content. Here, we prove a version of this result for regular parabolic Lipschitz functions. Our strategy is based on the original strategy of Jones and the generalization due to David and Semmes \cite{david1993quantitative}. Below, $\Delta$ denotes that standard grid of parabolic dyadic cubes on $n$-dimensional space time $\mathbb{R}^n$.

\begin{theorem}\label{t:lip-bi-lip}
	For each $L,M \geq 1$ and $\varepsilon > 0$, there are $K \geq 1$ and $\tau > 0$ such that the following holds. Let $I_0 \in \Delta$ and $\psi \colon \RR^n \to \RR^{n+1}$ be $(L,M)$-regular parabolic Lipschitz.
	Then, there is a decomposition $I_0 = B \cup F_1 \cup \dots \cup F_K$ such that image of $B$ under $\psi$ has small parabolic Hausdorff content, i.e., 
	\begin{align}\label{e:B-small-image}
		\mathcal{H}_{par,\infty}^{n+1}\left( \psi(B) \right) < \varepsilon \mathcal{H}_{par}^{n+1}(I_0), 
	\end{align} 
	and such that each $F_j$ is compact and $\psi|_{F_{j}}$ is bi-Lipschitz with upper constant $L$ and lower constant $\tau$, i.e.,  
	\begin{align}
	 \tau \dist((x,t),(y,s)) \leq 	\dist(\psi(x,t) , \psi(y,s)) \leq  L \dist((x,t),(y,s))
	\end{align}
	for all $(x,t) , (y,s) \in F_j$. 
\end{theorem}

Before embarking on the proof of Theorem \ref{t:lip-bi-lip}, we collect two useful auxiliary lemmas. The first introduces a convenient geometric constant, and the second provides a sort of weak bi-Lipschitz behavior of Lipschitz function on cubes where it is approximately affine and has large image. 

\begin{lemma}\label{l:b}
	Let $(x,t) , (y,s) \in \RR^n$ and let $I \in \Delta$ be the smallest parabolic cube containing $(x,t)$ such that $(y,s) \in 2I$. Then, 
	\begin{align}
		\dist((x,t),(y,s)) \geq b\diam(I)
	\end{align}
	for some $b > 0$ which depends only on dimension. 
\end{lemma}

\begin{proof}
	Let $J \in \Delta$ be the child of $I$ such that $(x,t) \in J$. By minimality, $(y,s) \not\in 2J$, and the estimate follows directly from this. 
\end{proof}

We define the $\gamma$-numbers for mappings $f \colon \RR^n \to \RR^m$ (equipping $\RR^m$ with the Euclidean metric) in the obvious way. Indeed, we let 
\[\gamma_f(\Xbf, r)^2 = \inf_{A}\frac{1}{|Q(\Xbf,r)|}  \iint_{Q(\Xbf,r)} \frac{|f(y,s) - A(y)|^2}{r^2} \, dy \,ds,\]
where the infimum is over all affine mappings $A \colon \RR^{n-1} \to \RR^m$ of the spatial variables, that is, each component of $A_i$ of $A$ is of the form $A_i(y) = a_i + \vec{b}_i\cdot y$.

\begin{lemma}\label{l:coarse-bi-lip}
	For each $L \geq 1$ and $\varepsilon > 0$, there are $\alpha, \tau > 0$ such that the following holds. Suppose $f \colon \RR^n \to \RR^{n+1}$ is parabolic $L$-Lipschitz map of the form 
	\begin{align}
		f(x,t) = (\psi(x,t), c + t) \in \RR^n \times \RR
	\end{align} 
	and $I \in \Delta$ is such that 
	\begin{align}
		\gamma_\psi(10I) < \alpha \quad \mbox{ and } \quad \mathcal{H}_{par,\infty}^{n+1}(f(I)) > \varepsilon \mathcal{H}^{n+1}_{par}(I).
	\end{align}
	Then, 
	\begin{align}
		\dist(f(x,t)  ,  f(y,s) ) \geq \tau \dist((x,t),(y,s)) 
	\end{align}
	for all $(x,t) , (y,s) \in 2I$ such that $\dist((x,t),(y,s)) \geq b \diam(I)$. 
\end{lemma}

\begin{proof}
	We prove this by compactness. Let $L \geq 1$, $\varepsilon > 0$ and suppose towards a contradiction that, for each $j \in \N$, there is a parabolic $L$-Lipschitz map $f_j \colon \RR^n \to \RR^{n+1}$ map of the form 
	\begin{align}
		f_j(x,t)  = (\psi_j(x,t) , c_j +t)
	\end{align}
	and a cube $I_j \in \Delta$ such that 
	\begin{align}\label{e:omega-sequence} 
		\gamma_{\psi_j}(10I_j) < 1/j \quad \mbox{ and } \quad \mathcal{H}_{par,\infty}^{n+1}(f_j(I_j)) > \varepsilon \mathcal{H}^{n+1}_{par}(I_j)
	\end{align}
	but 
	\begin{align}\label{e:fails-Lip-sequence}
		\dist( (x_j,t_j) , f(y_j,s_j) ) < \dist((x_j,t_j),(y_j,s_j))/j
	\end{align}
	for points $(x_j,t_j) , (y_j,s_j) \in 2I_j$ which satisfy $\dist((x_j,t_j),(y_j,s_j)) \geq b \diam(I_j)$. By scaling and translating we may assume that there is a fixed parabolic cube $J_0$ such that 
	\begin{align}
		I_j = J_0 \mbox{ for all } j \in \N. 
	\end{align}
	Moreover, since the properties are local and invariant under shifts in the $t$-component, we can assume that 
	\begin{align}
		f_j \mbox{ has compact support in } 10J_0 \quad \mbox{and} \quad c_j = 0 \quad \mbox{ for all } j \in \N. 
	\end{align}
	By Arzel\`a-Ascoli and the compactness of $2J_0$, we may suppose (by taking a sub-sequence, if necessary) that there is a parabolic Lipschitz map $f \colon \RR^n \to \RR^{n+1}$ and points $(x,t), (y,s) \in 2J_0$ such that $f_j \to f$ uniformly on $\RR^n$, $(x_j,t_j) \to (x,t)$ and $(y_j,s_j) \to (y,s)$. It is easy to check that 
	\begin{align}
		f(z,\sigma) = (\psi(z,\sigma) , t) \mbox{ for all } (z,\sigma) \in \RR^n,
	\end{align}
	for some $\psi \colon \RR^n \to \RR^{n}$.

	Using \eqref{e:omega-sequence} and the upper semi-continuity of $\mathcal{H}^{n+1}_{par,\infty}$ along sequences of compact subsets of a compact metric space (with respect to Hausdorff convergence) \cite{mattila1997measure}, it is not difficult to show that  
	\begin{align}\label{e:f-omega}
		\gamma_\psi(10J_0) = 0 \quad \mbox{ and } \quad \mathcal{H}^{n+1}_{par,\infty}(f(J_0)) > \varepsilon \mathcal{H}_{par}^{n+1}(J_0). 
	\end{align}
	Moreover, it follows from \eqref{e:fails-Lip-sequence} that 
	\begin{align}\label{e:f-squish}
		\dist(f(x,t) , f(y,s) ) = 0 \quad \mbox{ and } \quad \dist((x,t),(y,s)) \geq b \diam(J_0). 
	\end{align}
	
		Our goal now is to show that the first part of \eqref{e:f-omega} and \eqref{e:f-squish} imply 
	\begin{align}
		\mathcal{H}^{n+1}_{par,\infty}(f(J_0)) = 0, 
	\end{align}
	thus obtaining a contradiction to the second part of \eqref{e:f-omega}. Let $J_{0,X}$ denote the projection of $J_0$ on to the spatial components and $J_{0,t}$ denotes the projection of $J_0$ onto the $t$-component. Since $f$ fixes the $t$-component and (the first property of \eqref{e:f-omega} implies) $f(J_0)$ is a subset of some $t$-independent plane, it suffices to show (see \cite[Remark 2.8]{BHHLN-corona}) that 
	\begin{align}\label{e:time-slice}
		\int_{\RR^n} \mathbbm{1}_{f(J_0)}(X,s) \, d\mathcal{H}^{n-1}(X) = \mathcal{H}^{n-1}(f(J_0) \times \{s\}) = 0 \quad \mbox{ for all } s \in J_{0,t}.
	\end{align}
	
	Note to begin with that, since $f$ fixes the $t$-component, \eqref{e:f-squish} implies 
	\begin{align}\label{e:psi-squish}
		 (x,t) = (y,t), \quad |\psi(x,t) - \psi(y,t)| = 0 \quad \mbox{ and } \quad |x-y| \geq b \diam(J_0). 
	\end{align}
	The first part of \eqref{e:f-omega} provides an affine map $A \colon \RR^{n-1} \to \RR^n$ such that $\psi(z,\sigma) = A(z)$ for all $(z,\sigma) \in 10J_0$. By \eqref{e:psi-squish}, this map satisfies 
	\begin{align}
		|A(x) - A(y)| = 0 \quad \mbox{ and } \quad |x-y| \geq b \diam(J_0).  
	\end{align}
	In particular, by the area formula, 
	\begin{align}
	 \mathcal{H}^{n-1}(f(J_0) \times \{s\}) = \mathcal{H}^{n-1}( f(J_{0,X} \times \{s\})) = 	\mathcal{H}^{n-1}(A(10J_{0,X}) \times \{s\}) = 0
	\end{align}
	for all $s \in J_{0,t}$. This finishes the proof of \eqref{e:time-slice} and the lemma. 
\end{proof}

We will now prove Theorem \ref{t:lip-bi-lip}. Fix $L \geq 1$, $\varepsilon > 0$, and let $\psi$ and $I_0$ be as in the statement of Theorem \ref{t:lip-bi-lip}. Assume without loss of generality that $c = 0$ (the constant in the definition of regular parabolic Lipschitz mappings). Furthermore, let $\alpha,\tau > 0$ be the constants obtained by applying Lemma \ref{l:coarse-bi-lip} with constants $L$ and $\varepsilon$. 

We start by defining the bad set, that is, the set whose image has small Hausdorff content. For this, we will need to introduce two families of cubes; those with relatively small image and those where $\psi$ fails to be approximately affine. Indeed, let 
\begin{align}
	\mathcal{B}_1 &\coloneqq \left\{I \subseteq I_0 \colon \mathcal{H}_{par,\infty}^{n+1}(\psi(I)) < (\varepsilon/2) \mathcal{H}_{par}^{n+1}(I)\right\}; \\
	\mathcal{B}_2 &\coloneqq \left\{ I \subseteq I_0 \colon \gamma_{\psi}(3I) > \alpha \right\}. 
\end{align}
We define the bad set as the those points which are contained in at least one cube from $\mathcal{B}_1$ or at least $N$ dilated cubes from $\mathcal{B}_2$, where $N \in \N$ is some number to be chosen sufficiently large. More precisely, 
\begin{align}
	B \coloneqq R_1 \cup R_2 \coloneqq  \left(\bigcup_{I \in \mathcal{B}_1} I \right) \cup \left\{(x,t) \in I_0 \colon \sum_{I \in \mathcal{B}_2} \mathbbm{1}_{10I}(x,t)  \geq N \right\}. 
\end{align}

Now that $B$ is defined, we will check that its image under $\psi$ has small content.

\begin{lemma}
	For $N = N(\alpha, \varepsilon , L,M)$ large enough, the estimate \eqref{e:B-small-image} holds.
\end{lemma}
	
\begin{proof}
	First, it is clear from the definition of $\mathcal{B}_1$ that 
\begin{align}\label{e:psi-R_1}
	\mathcal{H}_{par,\infty}^{n+1}\left( \psi\left( R_1 \right) \right) < (\varepsilon/2) \mathcal{H}_{par}^{n+1}(I_0). 
\end{align}
Second, it follows from Chebyshev's inequality, the definition of regular Lip(1,1/2) functions and the definition of $\mathcal{B}_2$ that, for $N = N(\alpha,\varepsilon,L,M)$ large enough,
\begin{align}
	\mathcal{H}^{n+1}_{par,\infty}(R_2) \leq \frac{C}{N} \sum_{I \in \mathcal{B}_2} \ell(I)^{n+1}  \leq \frac{C}{N \alpha^2 } \sum_{I \in \mathcal{B}_2} \gamma_\psi(3I)^2\ell(I)^{n+1} \leq (\varepsilon/2L) \mathcal{H}^{n+1}_{par}(I_0). 
\end{align}
In particular, since $\psi$ is $L$-Lipschitz, this gives $\mathcal{H}^{n+1}_{par,\infty}(\psi(R_2)) < (\varepsilon/2) \mathcal{H}^{n+1}_{par}(I_0)$. This estimate and \eqref{e:psi-R_1} prove the lemma. 
\end{proof}

To obtain bi-Lipschitz estimates on the complement of $B$, we need weak bi-Lipschitz estimates on a collection of \textit{good cubes} $\mathcal{G}$, defined as  
\begin{align}
	\mathcal{G} \coloneqq \Delta \setminus \left( \mathcal{B}_1 \cup \mathcal{B}_2 \right). 
\end{align}
The following is a direct corollary of Lemma \ref{l:coarse-bi-lip} and our choice of $\alpha$ and $\tau$ in terms of $L$ and $\varepsilon$. 

\begin{lemma}\label{l:good-bi-lip}
	If $I \in \mathcal{G}$ then 
	\begin{align}
		\dist( \psi(x,t)  , \psi(y,s) ) \geq \tau  \dist((x,t) , (y,s)) 
	\end{align}
	for all $(x,t), (y,s) \in 2I$ such that $\dist((x,t) , (y,s)) > b \diam(I)$. 
\end{lemma}

By Lemma \ref{l:b} and Lemma \ref{l:good-bi-lip}, it is enough to construct sets $F_1,\dots,F_K$ which satisfy the following property: for each $x,y \in F_j$, if $I$ is the smallest parabolic cube containing $x$ such that $y \in 2I$, then $I \in \mathcal{G}$. The construction of such sets follows from a (now standard) coding argument introduced in \cite{jones1988lipschitz}. Since this argument only relies on the structure of dyadic cubes (i.e., is independent of the parabolic structure), we shall not include it here. We direct the reader to \cite{jones1988lipschitz} or \cite{david1993quantitative} for the details. 

\subsection{The proof of Theorem \ref{t:BPRPLI-implies-UR}} With the bi-Lipschitz decomposition in hand, we can commence with the proof of Theorem \ref{t:BPRPLI-implies-UR} more directly. Our first goal is to show that any set which satisfies the hypotheses of that theorem has big pieces of regular parabolic \textit{bi-Lipschitz} images of $n$-dimensional space in some larger space-time $\RR^{m+1}$ (see Lemma \ref{l:lip-to-bi-lip}). Before getting to this, we prove the following useful auxiliary result, which gives us a way to stretch the map in the places where it may fail the lower Lipschitz estimate. 

\begin{lemma}\label{l:bubbles}
	Let $I \in \Delta$. There is a mapping $\phi \colon \RR^{n-1} \times \RR \to \RR^{n}$ (whose image we equip with the Euclidean distance) such that the following conditions hold. 
	\begin{enumerate}
		\item \label{l:bubbles-derivatives} The map $\phi$ is supported in $3I$, satisfies $|\phi| \lesssim \ell(I)$, and each component $\phi_i$ of $\phi$ has parabolic derivative bounds 
		\begin{align*}
			|\nabla_x^\alpha \partial_t^m \phi_i| \leq C_{\alpha,m} \ell(I)^{-|\alpha|-2m+1}. 
		\end{align*}
		\item \label{l:bubbles-lower-lip} For each $(x,t),(y,s) \in 2I$, 
		\begin{align*}
			| \phi(x,t) - \phi(y,s) | \gtrsim |x-y| .
		\end{align*}
		\item \label{l:bubbles-small} For each $(x,t) \in (2I)^c$, 
		\begin{align*}
			| \phi(x,t) | \leq \ell(I)/2 .
		\end{align*}
		\item \label{l:bubbles-large} For each $(x,t) \in I$, 
		\begin{align*}
			| \phi(x,t) | \geq \ell(I). 
		\end{align*}
	\end{enumerate}
\end{lemma} 

\begin{proof}	
	By translating, we can assume that $I$ is centred at the origin. Let $\eta_i \in C_0^\infty(\RR^n)$, $i \in \{1,2\}$, be smooth bump functions with support in $3I$ such that 
	\begin{align}
		0 \leq \eta_1 \leq 1/8 \quad \mbox{and} \quad  \eta_1 \equiv 1/8 \mbox{ on } 2I, 
	\end{align} 
	and 
	\begin{align}
		0 \leq \eta_2 \leq \ell(I), \quad  \eta_2 \equiv \ell(I) \mbox{ on } I \quad \mbox{and} \quad \eta_2 \leq \ell(I)/8 \mbox{ on } (2I)^c. 
	\end{align}
	Furthermore, assume the parabolic derivative bounds 
	\begin{align}\label{e:derivative-bounds}
		| \nabla^\alpha \partial_t^m \eta_1 | \lesssim_{\alpha,m} \ell(I)^{-|\alpha|-2m} \quad \mbox{ and } \quad | \nabla^\alpha \partial_t^m \eta_2 | \leq_{\alpha,m} \ell(I)^{-|\alpha| -2m +1} . 
	\end{align}
	We define the map $\phi$ by 
	\begin{align}
		\phi(x,t) \coloneqq \left(x \eta_1(x,t) , \eta_2(x,t) \right) . 
	\end{align}
	
	We now check the conditions of the lemma. First, if $i = n$ then (1) follows directly from \eqref{e:derivative-bounds}. If $i \in \{1,\dots,n-1\}$ then (1) follows from the chain rule, \eqref{e:derivative-bounds} and the fact that $|x| \lesssim \ell(I)$. Item (2) holds by the fact that the first component of $\phi$ is $x \eta_1(x,t) \equiv x/8$ on $2I$. Item (3) holds by the fact that $|x \eta_1(x,t)| \leq \ell(I)/4$ on $3I$, $\supp(\eta_1) \subseteq 3I$ and $\eta_2 \leq \ell(I)/8$ on $(2I)^c$. Finally, (4) follows from the fact that $\eta_2 \equiv \ell(I)$ on $I$. 
\end{proof}

\begin{lemma}\label{l:lip-to-bi-lip}
	Suppose $\Sigma \subseteq \RR^{n+1}$ is parabolic Ahlfors-David regular and has big pieces of regular parabolic Lipschitz images of $n$-dimensional space. There is $m \in \N$, depending only on $n$, such that $\Sigma$ has big pieces of regular parabolic bi-Lipschitz images of $n$-dimensional space in $\RR^{m+1}$. 
\end{lemma}

\begin{proof}
	Let $L,M ,\theta > 0$ be the constants for which $\Sigma$ satisfies the BPRPLI. Fix $\Xbf \in \Sigma$, $R > 0$ and let $f \colon \RR^n \to \RR^{n+1}$ be the $(L,M)$-regular parabolic Lipschitz map such that 
	\begin{align}
		\sigma\left( \Sigma \cap B(\Xbf,R) \cap f(B_n(\mathbf{0},R))\right) \geq \theta R^{n+1},
	\end{align}
	where $B_n(\mathbf{0},R)$ denotes the parabolic ball in $\RR^n$ centered at the origin with radius $R$. By ADR, this estimate implies 
	\begin{align}\label{e:image-large-content}
		\mathcal{H}_{par,\infty}^{n+1}(\Sigma \cap B(\Xbf,R) \cap f(B_n(\mathbf{0},R))) \geq \varepsilon \mathcal{H}^{n+1}_{par}(B_n(\mathbf{0},R))
	\end{align}
	for some $\varepsilon$ depending on the $\theta$ and the ADR constant. By Theorem \ref{t:lip-bi-lip}, there are constants $N \geq 1$ and $\tau > 0$ (depending on $L,M$ and $\varepsilon$) and a decomposition 
	\begin{align}\label{e:decomposition}
		B_n(\mathbf{0},R) = B \cup F_1 \cup \dots \cup F_N 
	\end{align}
	such that 
	\begin{align}\label{e:image-small-content}
		\mathcal{H}_{par,\infty}^{n+1}(f(B)) < (\varepsilon/2) \mathcal{H}_{par}^{n+1}(B_n(\mathbf{0},R))
	\end{align}
	and such that $f|_{F_j}$ is bi-Lipschitz with upper constant $L$ and lower constant $\tau$ for each $j \in \{1,\dots,N\}$. It follows from \eqref{e:image-large-content}, \eqref{e:decomposition} and \eqref{e:image-small-content} (by pigeonholing) that there exists some $j$ such that 
	\begin{align}\label{e:big-piece-F_j}
		\mathcal{H}_{par,\infty}^{n+1}(\Sigma \cap B(\Xbf,R) \cap f(F_j)) \geq (\varepsilon/2N)\mathcal{H}^{n+1}_{par}(B_n(\mathbf{0},R)) \gtrsim \varepsilon R^{n+1} . 
	\end{align}
	
	Our goal now is to construct a regular parabolic bi-Lipschitz map $g \colon \RR^{n} \to \RR^{m+1}$ (for some $m \in \N$) whose restriction to $F_j$ agrees with $f$. By \eqref{e:big-piece-F_j}, this is sufficient to conclude the lemma. For brevity, we will denote 
	\begin{align}
		A \coloneqq F_j .
	\end{align}
	We will need a regularized (i.e., regular parabolic Lipschitz) distance function to $A$. By \cite[Lemma 3.24]{BHHLN-CME}, there are constants $C_1 \geq 1$ and $L_1,M_1 \geq 0$, depending on dimension, and a regular Lip(1,1/2) function $g_A$ such that 
	\begin{align}
		C_1^{-1} \dist((x,t),A) \leq g_A(x,t) \leq C_1 \dist((x,t),A) \quad \mbox{ for all } (x,t) \in \RR^n. 
	\end{align} 
	We also need a Whitney decomposition. Let $\mathcal{I}$ be a Whitney decomposition (by parabolic cubes) of $\RR^n \setminus A$ such that 
	\begin{align}
		 \diam(I) \leq 10^{-3}\dist(I,A) \leq 3\diam(I)  \quad \mbox{ for all }I \in \mathcal{I}
	\end{align}
	and decompose $\mathcal{I}$ into $K \in \N$ sub-families $\{\mathcal{I}_k\}_{k=1}^K$ such that 
	\begin{align}\label{e:separation-bi-lip}
		d_\infty(I, J) \ge 10^6L\tau^{-1}\max\{\diam_\infty(I), \diam_\infty(J)\}, 
	\end{align}
	for all $k \in \{1,\dots,K\}$ and $I \neq J \in \mathcal{I}_k$, 
	where $\diam_\infty$ is the diameter measured in $d_\infty$.
	This can be done in such a way that $K$ depends only on dimension.
	Then, we define $g \colon \RR^n \to \RR^{n+1} \times \RR \times (\RR^{n})^K$ by
	\begin{align*}
		g_{-1} &= f & &\mbox{ if $g_{-1}$ is the $\RR^{n+1}$ coordinate of $g$;} \\
		g_0 &= g_A & &\mbox{ if $g_0$ is the $\RR$ coordinate of $g$;} \\
		g_k &=  \sum_{I \in \mathcal{I}_k} \phi_I & &\mbox{ if $g_k$ is the $k^{th}$ copy of the $\RR^{n}$ coordinate of $g$.} 
	\end{align*}
Here $\phi_I$ is from Lemma \ref{l:bubbles}.	
	It is clear that $g|_A = f|_A$ since each of the mappings $g_k$, $k \in \{0,\dots,K\}$, are identically zero on $A$. It is also clear that $g$ is parabolic Lipschitz since it is the sum of finitely many parabolic Lipschitz mappings. Thus, to check that $g$ is bi-Lipschitz we need only prove the lower Lipschitz bound, i.e.,
	\begin{align}\label{e:lower-lip}
		\dist( g(x,t) ,  g(y,s) ) \gtrsim \dist((x,t) , (y,s) ) . 
	\end{align}
	
	Let $(x,t),(y,s) \in \RR^n$ and suppose to begin with that there exists $I,J \in \mathcal{I}$ such that $(x,t) \in I$ and $(y,s) \in J$. Assume without loss of generality that $\diam(I) \geq \diam(J)$. Under the above assumptions, the proof of \eqref{e:lower-lip} splits into three cases.

	\noindent\textbf{Case 1:} Suppose that $\dist(I,J) \geq 10^5L\tau^{-1} \diam(I)$. 
	
	Let $(v,\rho) \in A$ and $(w,\sigma) \in A$ be points which are closest to $(x,t)$ and $(y,s)$, respectively. In this way, we have 
	\begin{align*}
		\dist((x,t),(v,\rho)) \leq 3\diam(I)  \quad \mbox{ and } \quad \dist((y,s) , (w,\sigma))  \leq 3\diam(J) .
	\end{align*}
	The assumption of Case 1 tell us that 
	\begin{align*}
		\dist((x,t),(y,s)) \geq 10^5L\tau^{-1} \diam(I) \geq 10^5L\tau^{-1} \diam(J). 
	\end{align*}
	Thus, by the triangle inequality, 
	\begin{align*}
		\dist((v,\rho) , (w,\sigma)) \geq (10^5L\tau^{-1}/2)\diam(I) \geq (10^5L\tau^{-1}/2)\diam(J) . 
	\end{align*}
	Using the above estimates, along with the fact that $f$ is bi-Lipschitz on $A$ with constant $(L,\tau)$, it follows from the triangle inequality that 
	\begin{align*}
		\dist( f(x,t)  , f(y,s)) &\geq \dist( f(v,\rho) ,  f(w,\sigma) )  - 3L\diam(I) - 3L\diam(J) \\
		&\geq \tau \dist((v,\rho) ,(w,\sigma)) - 3L\diam(I) - 3L\diam(J) \\
		&\geq (\tau/2) \dist((v,\rho) ,(w,\sigma)) \geq (\tau/4) \dist((x,t),(y,s)). 
	\end{align*}
	This completes the proof of \eqref{e:lower-lip} in Case 1. 
	
	\noindent\textbf{Case 2:} Suppose that $\dist(I,J) < 10^5L\tau^{-1}\diam(I)$ and $\diam(J) \leq \diam(I)/(5C_1^2)$. 
	
	We note that the assumption $\dist(I,J) < 10^5L\tau^{-1}\diam(I)$ implies 
	\begin{align*}
		\dist((x,t),(y,s)) \lesssim \diam(I). 
	\end{align*}
	Furthermore, after recalling the estimates on $g_A$ and the definition of the Whitney cubes, the second assumption in Case 2 gives
	\begin{align*}
		g_A(x,t) \geq C_1^{-1}\dist((x,t),A)  \geq 10^3C_1^{-1}\diam(I)
	\end{align*}
	and 
	\begin{align*}
		g_A(y,s) \leq C_1 \dist((y,s),A) \leq 3\cdot 10^5 C_1\diam(J)   \leq (3/5)10^3C_1^{-1}\diam(I). 
	\end{align*}
	The above three estimates then imply 
	\begin{align*}
		\dist( g(x,t)  , g(y,s))  \geq |g_A(x,t) - g_A(y,s)| \gtrsim \diam(I) \gtrsim \dist((x,t),(y,s)),
	\end{align*}
	which finishes the proof of \eqref{e:lower-lip} in Case 2.

	\noindent\textbf{Case 3:} Suppose that $\dist(I,J) < 10^5L\tau^{-1}\diam(I)$ and $\diam(J) > \diam(I)/(5C_1^2)$. 
	
	Let $k \in \{1,\dots,K\}$ such that $I \in \mathcal{I}_k$. By the first assumption in Case 3 and the separation condition \eqref{e:separation-bi-lip} (and the fact that $\diam(J) \geq \diam(I)$), we know that $I \cap 3Q = J \cap 3Q = \emptyset$ for all $Q \in \mathcal{I}_k$ not equal to $I$. In particular, since each $\phi_Q$ is supported in $3Q$, it follows that 
	\begin{align}
		g_k(x,t) = \phi_I(x,t) \quad \mbox{ and } \quad g_k(y,s) = \phi_I(y,s) .
	\end{align}
	By definition, we always have $(x,t) \in I$. If $(y,s) \in 2I$, then we use Lemma \ref{l:bubbles} \eqref{l:bubbles-lower-lip} and the fact that $f_{n+1}$ (which is the $t$-component of the image of $g$) fixes the $t$ component to get 
	\begin{align*}
		\dist( g(x,t) , g(y,s) ) &\geq |\phi_I(x,t) - \phi_I(y,s)| + |f_{n+1}(x,t) - f_{n+1}(y,s)|^\frac{1}{2} \\
		&\gtrsim |x-y| + |t-s|^\frac{1}{2}, 
	\end{align*} 
	as required for \eqref{e:lower-lip}. If $(y,s) \not\in 2I$, then it follows from Lemma \ref{l:bubbles} \eqref{l:bubbles-small} and \eqref{l:bubbles-large} that 
	\begin{align}
		\dist( g(x,t)  , g(y,s) )  \geq |g_k(x,t) - g_k(y,s)| \gtrsim \diam(I). 
	\end{align}
	The first assumption in Case 3 implies $\dist((x,t),(y,s)) \lesssim \diam(I)$ (as in Case 1) so the above estimate is enough to conclude \eqref{e:lower-lip}. 
	
	 To complete the bi-Lipschitz estimates, there are (up to symmetry) two more cases to consider. First, if $(x,t),(y,s) \in A$ then there is nothing to show since $g|_A = f|_A$ and $f$ is bi-Lipschitz on $A$. Lastly, one has to deal with the situation that $(x,t) \in I$ for some $I \in \mathcal{I}$ and $(y,s) \in A$. This can be dealt with in the same way as Case 1 and Case 2, conditioning instead on whether $\dist((y,s),I)$ is larger or smaller than $10^5L\tau^{-1}\diam(I)$. We leave the details to the reader. 
	
	The only thing left to prove is that each spatial component of $g$ is a regular parabolic Lipschitz function. Let us fix such a component, which we will denote by $h$. Since each component of $g_{-1}$ and $g_0$ are regular parabolic Lipschitz (by assumption or construction), we may assume that 
	\begin{align}
		h = \sum_{I \in \mathcal{I}_k} \phi_{I,i} 
	\end{align}
	for some $k \in \{1,\dots,K\}$ and $i \in \{1,\dots,n\}$, where $\phi_{I,i}$ denotes the $i^{th}$ component of $\phi_I$. Writing
	\begin{align}
		\eta_I \coloneqq \ell(I)^{-1} \phi_{I,i} \quad \mbox{ and } \quad \theta_I \coloneqq \ell(I),
	\end{align}
	so that 
	\begin{align}
		h = \sum_{I \in \mathcal{I}_k} \eta_I \theta_I, 
	\end{align}
	it is clear from Lemma \ref{l:bubbles} \eqref{l:bubbles-derivatives} and \eqref{e:separation-bi-lip} that the functions $\{\eta_I\}$, $\{\theta_I\}$ and the cubes $\mathcal{I}_k$ satisfy the hypotheses of Lemma \ref{l:gluing-for-regular-functions-prime} (H2). We deduce immediately that $h$ is regular, which finishes the proof.  
	\end{proof}

Lemma \ref{l:lip-to-bi-lip} essentially allow us to reduce the proof of Theorem \ref{t:lip-bi-lip} to proving that bi-Lipschitz images of $n$-dimensional space-time in $\RR^{m+1}$ satisfy the geometric lemma (more shall be said about this later). For now, we prove this latter result. We start with the following.

\begin{lemma}\label{c:dor}
	Let $L, M \geq 0$ and suppose $\psi \colon \RR^{n} \to \RR^m$ (where $\RR^m$ is equipped with Euclidean distance) is such that each component $\psi_j$ is an $(L,M)$-regular Lip(1,1/2) function. Then,  
	\begin{align}
		\sum_{\substack{I \in \pcubes \\ I \subseteq I_0 }} \fapprox_{\psi}(3I)^2\size{I}^{n+1} \lesssim_{L,M,m} \size{I_0}^{n+1}, \quad I_0 \in \pcubes. 
	\end{align}
\end{lemma}

\begin{proof} 
	For each $I \in \pcubes$ and each $j \in \{1,\dots,m\}$, let $A_{I,j} \colon \RR^{n-1}  \to \RR$ be an affine function such that 
	\begin{align}
		\fapprox_{\psi_j}(3I,A_{I,j}) \leq \fapprox_{\psi_j}(3I) + \size{I}.
	\end{align}
	Define $A_I \colon \RR^{n-1} \to \RR^{m}$ by 
	\begin{align}
		A_I(x) = \left(A_{I,1}(x),\dots,A_{I,m}(x) \right), \quad (x,t) \in \RR^{n-1} \times \RR.
	\end{align}
	Then, for each $I_0 \in \pcubes,$ it follows from the Pythagorean Theorem and the fact that each $\psi_j$ is $(L,M)$-regular parabolic that 
	\begin{align}
		\sum_{\substack{I \in \pcubes \\ I \subseteq I_0 }} \fapprox_{\psi}(3I)^2\size{I}^{n+1} &\leq \sum_{\substack{I \in \pcubes \\ I \subseteq I_0 }}  \sum_{j=1}^{m} \fapprox_{\psi_j}(3I,A_{I,j})^2\size{I}^{n+1}  \\
		&\leq \sum_{j=1}^m \sum_{\substack{I \in \Delta \\ I \subseteq I_0}} \left[ \fapprox_{\psi_j}(3I)^2\size{I}^{n+1} + \size{I}^{n+2} \right] \\
		&\lesssim_{L,M,m} \size{I_0}^{n+1} , 
	\end{align}
	as required. 
\end{proof}

We now show that regular parabolic bi-Lipschitz images of $n$-dimensional space time satisfies the geometric lemma. 

\begin{lemma}\label{l:geometric-lemma-bi-lip-images}
	Let $L,M \geq 0$ and suppose $\Psi \colon \RR^{n-1} \times \RR \to \RR^{m+1}$ an $(L,M)$-parabolic regular bi-Lipschitz map. Then, the set $\Sigma = \Psi(\RR^{n-1} \times \RR) \subseteq \RR^{m+1}$ is ADR with constant depending only on $L$ and $n$ and satisfies the geometric lemma with constants depending only on $L,M$ and $n$. 
\end{lemma}

\begin{proof}
	The statement about ADR is clear since $\Psi$ is $L$-bi-Lipschitz. To prove the geometric lemma, fix some system of cube $\cubes$ on $\Sigma$ and $Q_0 \in \cubes$. For each $Q \in \cubes(Q_0),$ choose $I_Q \in \pcubes$ such that $\Psi^{-1}(3Q) \subset 3I_Q$ and $\size{I_Q} \sim \size{Q}.$ Since $\Psi$ is bi-Lipschitz, it is easy to check that 
	\begin{align}\label{e:bounded-cubes}
		\#\{ Q \in \cubes \colon I_Q = I\} \lesssim_{L,n} 1 \quad \mbox{ for each } I \in \pcubes. 
	\end{align}

	Let $\ell_0 \coloneqq \max\{\size{I_Q} \colon Q \in \dd(Q_0)\}$ and let $\mathcal{T}$ be the set of cubes $I \in \pcubes$ such that $I \cap \Psi^{-1}(3Q_0) \neq\emptyset$ and $\size{I} = \ell_0.$ Since each $I_Q$ intersect $3Q \subseteq 3Q_0,$ and satisfies $\size{I_Q} \leq \ell_0$, it follows that 
	\begin{align}\label{e:contained-cubes}
		\mbox{ for each } Q \in \cubes(Q_0) \mbox{ there exists } I \in \mathcal{T} \mbox{ such that } I_Q \subseteq I. 
	\end{align} 
	Moreover, since $\ell_0 \lesssim \ell(Q_0)$ and $\Psi$ is $L$-bi-Lipschitz, there is a parabolic ball $B \subset \RR^n$ with radius $r_B \sim_{L,n} \ell(Q_0)$ such that 
	\begin{align}\label{e:T-in-ball}
		\bigcup_{I \in \mathcal{T}} I \subseteq B.
	\end{align}
	Since $\mathcal{T}$ is a disjoint collection of cubes, this implies 
	\begin{align}\label{e:cal-T}
		\sum_{I \in \mathcal{T}} \ell(I)^{n+1} \lesssim_{L,n} \ell(Q_0)^{n+1}. 
	\end{align}

	Recalling the definition of $(L,M)$-parabolic regular bi-Lipschitz map from Definition \ref{d:p-regular-bi-lip}, we note that $\Psi$ can be written as 
	\begin{align}
		\Psi(x,t) = (\psi(x,t),t+c_\psi ) \quad \mbox{for all } (x,t) \in \RR^n, 
	\end{align}
	for some $\psi \colon \RR^{n-1} \times \RR \to \RR^n$ which satisfies the hypotheses of Lemma \ref{c:dor}. Without loss of generality $c_\psi = 0$. For each $I \in \pcubes$, we choose $A_I \colon \RR^{n-1}\to \RR^{m}$ to be an affine function such that 
	\begin{align}\label{e:omega-psi}
		\frac{1}{|3I|} \int_{3I} \left| \frac{\psi(y,s) - A_I(y)}{\ell(3I)} \right| \, dy ds \leq \fapprox_\psi(3I) + \size{I}.
	\end{align}
	
	Consider an arbitrary $Q \in \dd(Q_0)$ for the moment. The set  
	\begin{align}
		P_Q = A_{I_Q}(\RR^{n-1}) \times \RR  
	\end{align}
	defines a $t$-independent plane in $\RR^{m+1}$ associated to $Q$. Using the properties that $\Psi^{-1}(3Q) \subseteq 3I_Q$ and $\ell(I_Q) \sim \ell(Q)$, along with a change of variables $(X,t) = \Psi(x,t)$ and the inequality \eqref{e:omega-psi}, we have
	\begin{align*}
		\sapprox(\Xbf_Q,3\ell(Q))^2\sigma(Q) &	\lesssim \int_{3Q} \left| \frac{\dist((X,t), P_Q)}{\size{Q}}\right|^2 \, d\sigma(X,t)  \\
		&\lesssim_L \int_{3I_Q} \left|\frac{\dist(\Psi(x,t) , P_Q)}{\size{I_Q}} \right|^2 \, d\meas(x,t)\\ 
		&= \int_{3I_Q} \left|\frac{\psi(x,t) - A_I(x)}{\size{I_Q}} \right|^2 \, d\meas(x,t) \\
		&\lesssim \fapprox_{\psi}(3I_Q)\size{I_Q}^{n+1} + \size{I_Q}^{n+2}.
	\end{align*}
	Combining this with \eqref{e:bounded-cubes}, \eqref{e:cal-T} and Lemma \ref{c:dor}, we get 
	\begin{align*}
		\sum_{Q \in \cubes(Q_0)} \sapprox(\Xbf_Q,3\ell(Q))^2\sigma(Q) &\lesssim \sum_{Q \in \cubes(Q_0)} \left[\fapprox_\psi(3I_Q)^2\size{I_Q}^{n+1} + \ell(I_Q)^{n+2} \right] \\
		&\lesssim_{L,n} \sum_{I_0 \in \mathcal{T}} \sum_{\substack{I \in \pcubes \\ I \subseteq I_0}} \left[\fapprox_\psi(3I)^2\size{I}^{n+1} + \size{I}^{n+2} \right] \\
		&\lesssim_{L,M,n} \sum_{I_0 \in \mathcal{T}} \ell(I_0)^{n+1} \lesssim_{L,n} \size{Q_0}^{n+1}, 
	\end{align*}
	which finishes the proof. 
\end{proof}

The previous lemma gives us a way to check the geometric lemma in a higher dimensional space time. The follow easy lemma says that any subset of $n$-dimensional space time which satisfies the geometric lemma in $m$-dimensional space time ($m \geq n$) also satisfies it in the original space. 

\begin{lemma}\label{l:dimension-reduction}
	Suppose $\Sigma \subseteq \RR^{n+1}$ is parabolic Ahlfors-David regular and satisfies the geometric lemma in $\RR^{m+1} \supseteq \RR^{n+1}$ (i.e., the planes in the definition of $\beta$ are $n$-dimensional $t$-independent planes $\RR^{m+1}$), then $\Sigma$ satisfies the geometric lemma in $\RR^{n+1}$, i.e., is P-UR. 
\end{lemma}

\begin{proof}
	Fix a point $\Xbf \in \Sigma$ and an $n$-dimensional $t$-independent plane $P$ in $\RR^{m+1}$. It suffices to find an $n$-dimensional $t$-independent plane $V$ in $\RR^{n+1}$ such that 
	\begin{align}
		\dist(\Xbf , V) \leq \dist(\Xbf,P).
	\end{align}  
	Let $\pi$ denote the projection onto $\RR^{n+1}$ which fixes the $t$-component and let $V' = \pi(P)$ denote the projection of $P$ onto $\RR^{n+1}$. Thus, $V'$ is a $k$-dimensional $t$-independent plane in $\RR^{n+1}$ for some $1 \leq k \leq n$. Take $V$ to be any $n$-dimensional $t$-independent plane which contains $V'$. Then, since $\pi$ is 1-Lipschitz and $\Xbf \in \RR^{n+1}$, we have 
	\begin{align}
		\dist(\Xbf,V) \leq \dist(\Xbf,V') = \dist(\pi(\Xbf),\pi(P)) \leq \dist(\Xbf,P), 
	\end{align}
	which finishes the proof.  
\end{proof}

We now combine the above results to conclude the proof of Theorem \ref{t:lip-bi-lip}. 

\begin{proof}[Proof of Theorem \ref{t:lip-bi-lip}]
	Let $\Sigma$ satisfies the hypotheses of Theorem \ref{t:lip-bi-lip}. By Theorem \ref{l:lip-to-bi-lip}, we know that $\Sigma$ has BPRPBI in $\RR^{m+1}$. Since $(L,M)$-regular parabolic bi-Lipschitz images of $\RR^n$ are uniformly ADR and satisfy the geometric lemma in $\RR^{m+1}$ with uniform constants (Lemma \ref{l:geometric-lemma-bi-lip-images}), it follows from the \textit{stability of the geometric lemma under big pieces}, see \cite[Proposition 2.29]{BHHLN-BP}, that $\Sigma$ satisfies the geometric lemma in $\RR^{m+1}$. Then, that $\Sigma$ is P-UR follows directly from Lemma \ref{l:dimension-reduction}. 
\end{proof}

\section{P-UR implies BPRPBI}

Here we show the other half of Theorem \ref{main.thrm}, that is, the following theorem.

\begin{theorem}\label{t:UR-implies-BPLI}
Suppose that $\Sigma$ is a parabolic Ahlfors-David regular set (see Definition \ref{ADR.def}). If $\Sigma$ is parabolic uniformly rectifiable then $E$ has big pieces of regular parabolic bi-Lipschitz images of $n$-dimensional space. 
\end{theorem}

We sketch the general strategy before beginning the proof. In \cite{BHHLN-BP}, using ideas from \cite{BH-BP}, the authors prove that corona decompositions by a class of sets $\mathcal{E}$ implies that the set is big pieces of big pieces of the sets $\mathcal{E}$, written $BP^2(\mathcal{E})$. In the context of P-UR, \cite{BHHLN-BP} shows that the (weaker unilateral) corona decomposition implies that P-UR sets have $BP^2(\mathcal{E})$, where $\mathcal{E}$ is the collection of $(L,M)$ parabolic regular {\bf graphs}. The main tool there, which is important for us as well, is a induction argument that involves, in the first step, producing for each $Q_0$ a ``large" graph and several smaller graphs for a collection of separated sub-cubes of $Q_0$, and using the larger graph to tie them together. In subsequent steps, a large graph for $Q_0$ is produced to tie together separated {\it collections} of graphs. Our scheme is not-dissimilar, but much more delicate. A fundamental issue is that, in these repeated constructions, the reference planes for various graphs may `twist" and make it impossible to produce a graphical approximation, that is, to show that $BP(\mathcal{E})$ where $\mathcal{E}$ is the collection of $(L,M)$ parabolic regular {\bf graphs}. Indeed, this is not just a parabolic issue; Hyrcak's example (as mentioned in the introduction) shows that UR sets can fail to have big pieces of Lipschitz graphs. To remedy this, we introduce a `slow' twisting mechanism, which is also present in the work of Azzam and Schul \cite{AS-HardSard}, to twist these reference planes and then carefully glue these pieces together. A challenge in the parabolic setting, as opposed to the setting of (Euclidean) UR sets, is that we must produce the additional regularity on the maps involved to obtain {\it regular} parabolic bi-Lipschitz images. 

\subsection{Smooth Parabolic Rotation via Givens Rotations and Lipschitz Images}
As hinted at above, we will need to produce a slow twisting that takes one plane to another.

\begin{lemma}\label{lem:givenrot}
	Let $P_0$ and $P_1$ be two $t$-independent planes in $\mathbb{R}^{n+1}$ which pass through the origin and $R > 0$. There exists a map $F$ from $\mathbb{R}^{n+1}$ to $\mathbb{R}^{n+1}$ such that 
	\begin{enumerate}
		\item $F|_{B(0,R)}$ is a (fixed) linear transformation that takes $P_0$ to $P_1$ and $F|_{B(0,2R)^c}$ is the identity. 
		\item $F(B(0,r)) = B(0,r)$ for all $r >0$.
		\item Writing $F(X,t) = (F_1(X,t), \dots, F_n(X,t), F_{n+1}(X,t))$, we have
		\[F_{n+1}(X,t) = t\]
		and, for $i = 1,2,\dots n$,
		\[|\nabla F_i| \lesssim 1, \quad |\nabla^2 F_i| \lesssim 1/R, \quad |\partial_t F_i| \lesssim 1/R,  \]
		where the implicit constants depend only on dimension.
		\item $F$ is an $(L, M)$-regular parabolic bi-Lipschitz map, where $L$ and $M$ depend only on dimension.
	\end{enumerate}
\end{lemma}

\begin{proof}
	First, recall $\rho(X,t) = (|X|^4 + |t|^2)^{1/4}$ and, as observed in \eqref{rhodistcomp.eq},
	\[\rho(X,t) \le \dist((X,t), (0,0)) \le  2^{3/4}\rho(X,t).\]
	Thus, if $a = 2^{1/4}$, 
	\[(X,t) \in B(0,R) \implies \rho(X,t) < R\]
	and 
	\[(X,t) \in B(0,2R)^c \implies \rho \ge aR.\]
	Let \[
	\nu_0,\nu_1\in \mathbb S^{n-1}
	\]
	be such that $(\nu_0,0)$ and $(\nu_1,0)$ are normal to $P_0$ and $P_1$, respectively, chosen such that
	\[
	\nu_0\cdot \nu_1\ge 0.
	\]
	Then the angle between $P_0$ and $P_1$ is given by 
	\[
	\theta:=\arccos(\nu_0\cdot \nu_1)\in [0,\pi/2].
	\]
	If $\theta = 0$, we can simply take $F$ to be the identity, so we will assume this is not the case.
	
	Define two unit vectors \[
	e_1:=\nu_0
	\]
	and
	\[
	e_2:=\frac{\nu_1-(\nu_1\cdot \nu_0)\nu_0}
	{|\nu_1-(\nu_1\cdot \nu_0)\nu_0|}.
	\]
	Since
	\[
	|\nu_1-(\nu_1\cdot \nu_0)\nu_0|=\sin\theta,
	\]
	we have
	\[
	\nu_1=\cos\theta\,e_1+\sin\theta\,e_2.
	\]
	Define the skew-symmetric spatial map $K:\mathbb{R}^n\to\mathbb{R}^n$ by
	\[
	KX=(X\cdot e_1)e_2-(X\cdot e_2)e_1.
	\]
	Thus $Ke_1=e_2$, $Ke_2=-e_1$, and $K=0$ on $\operatorname{span}\{e_1,e_2\}^{\perp}$.  For $\alpha\in\mathbb{R}$, define the spatial Givens rotation $g_\alpha:\mathbb{R}^n\to\mathbb{R}^n$ by
	\[
	\begin{aligned}
		g_\alpha X
		&=X+ (\cos\alpha-1)\bigl[(X\cdot e_1)e_1+(X\cdot e_2)e_2\bigr] \\
		&\qquad +\sin\alpha\bigl[(X\cdot e_1)e_2-(X\cdot e_2)e_1\bigr].
	\end{aligned}
	\]
	Equivalently, if
	\[
	X=X^\perp+a e_1+b e_2,
	\qquad X^\perp\perp \operatorname{span}\{e_1,e_2\},
	\]
	then
	\[
	g_\alpha X
	=X^\perp+(a\cos\alpha-b\sin\alpha)e_1
	+(a\sin\alpha+b\cos\alpha)e_2.
	\]
	In particular,
	\[
	g_\theta \nu_0=\nu_1,
	\]
	so the associated space-time map
	\[
	G_\alpha(X,t):=(g_\alpha X,t)
	\]
	sends $P_0$ to the $t$-independent plane with spatial normal $g_\alpha\nu_0$.  Thus,
	\begin{equation}\label{eq:fullrotgiv}
		G_\theta(P_0)=P_1.
	\end{equation}
	Then, we have a smooth parameterization of planes
	\[
	P_\alpha:=G_\alpha(P_0)=P_{g_\alpha\nu_0},
	\qquad 0\le \alpha\le \theta.
	\]
	Differentiating the formula for $g_\alpha$ gives
	\begin{equation}\label{e:diff-g_alpha}
		\frac{d}{d\alpha}g_\alpha X=K g_\alpha X,
		\qquad
		\frac{d^2}{d\alpha^2}g_\alpha X=K^2g_\alpha X.
	\end{equation}
	Also, 
	\[
	|g_\alpha X|=|X|,
	\qquad
	\|K\|_{\ell^2 \to \ell^2} = 1,
	\qquad
	\|K^2\|_{\ell^2 \to \ell^2} =  1.
	\]
	Next, we let $\eta:\RR\to[0,1]$ be a function such that
	\[
	\eta(s)=1 \quad \text{for } s\le 1,
	\qquad
	\eta(s)=0 \quad \text{for } s\ge a=2^{1/4},
	\]
	and
	\[
	\eta'(s)<0 \quad \text{for } 1<s<a.
	\]
	Now we want to smoothly turn the plane as a function of the distance to the center; however, the lack of smoothness of the parabolic metric (in particular on the plane $t = 0$) is problematic. For this reason, we employ the smoother, but comparable, function $\rho$. The angle at which we `spiral' is given by the function
	\[
	\psi(X,t):=\theta\,\eta\!\left(\frac{\rho(X,t)}{R}\right).
	\]
	Finally, we define our bi-Lipschitz map $F(X,t)$ as
	\[
	F(X,t):=G_{\psi(X,t)}(X,t)=(g_{\psi(X,t)}X,t).
	\]
	It will sometimes be useful to write
	\[F(X,t)=(f(X,t),t),
	\text{ \ \ where\ \ }
	f(X,t):=g_{\psi(X,t)}X.\]

	Now, we set out to show that $F(X,t)$ has the required properties. Starting with (1), we have observed above that, if $\dist((X,t), 0) < R$, then $\rho(X,t) < R$. Therefore, $\psi(X,t) = \theta$, which, in turn, gives $F(X,t):=G_{\theta}(X,t)$, the Givens rotation that takes $P_0$ to $P_1$. Moreover, if $\dist((X,t), 0) \geq 2R$, then $\rho(X,t) \ge aR$. Therefore, $\psi(X,t) = 0$, which, in turn, gives $F(X,t):=G_{0}(X,t) = (X,t)$. It is also clear that $F$ is smooth. We can deduce (2), that $F(B(0,r)) = B(0,r)$ for all $r > 0$, because, for fixed $s$, the points where $\dist((X,t), 0) = s$ are just rotated in space. 
	
	Next, we prove the derivative bounds on $F$ stated in item (3), which amount to direct computation. We begin by estimating the derivatives of the spatial part
	\[
	f(X,t):=g_{\psi(X,t)}X.
	\]
	The estimates are only nontrivial on the region
	\[
	\mathcal A_R:=\{(X,t):R<\rho(X,t)< a R\}.
	\]
	Outside $\mathcal A_R$, the angle $\psi$ is constant, so $f$ is either a fixed rotation or the identity.
	
	Let
	\[
	r:=\left|X\right|,
	\qquad
	\rho=(r^4+t^2)^{1/4}.
	\]
	For $\rho>0$,
	\[
	\partial_{x_j}\rho=\frac{r^2x_j}{\rho^3},
	\]
	and
	\[
	\partial_t\rho=\frac{t}{2\rho^3}.
	\]
	Thus
	\[
	\left|\nabla_X\rho\right|\le 1,
	\qquad
	\left|\partial_t\rho\right|\le \frac{C}{\rho}.
	\]
	For the second spatial derivatives,
	\[
	\partial^2_{x_jx_k}\rho
	=\frac{2x_jx_k+r^2\delta_{jk}}{\rho^3}
	-\frac{3r^4x_jx_k}{\rho^7}.
	\]
	Hence
	\[
	\left|\nabla_X^2\rho\right|\le \frac{C_n}{\rho}.
	\]
	On $\mathcal A_R$, where $\rho\approx R$, these yield
	\[
	\left|\nabla_X\rho\right|\le C,
	\qquad
	\left|\partial_t\rho\right|\le \frac{C}{R},
	\qquad
	\left|\nabla_X^2\rho\right|\le \frac{C_n}{R}.
	\]
	
	Let
	\[
	M_1:=\|\eta'\|_{L^\infty},
	\qquad
	M_2:=\|\eta''\|_{L^\infty}.
	\]
	Since
	\[
	\psi=\theta\eta(\rho/R),
	\]
	we have
	\[
	\partial_{x_j}\psi
	=\frac{\theta}{R}\eta'(\rho/R)\partial_{x_j}\rho,
	\]
	\[
	\partial_t\psi
	=\frac{\theta}{R}\eta'(\rho/R)\partial_t\rho,
	\]
	and
	\[
	\partial^2_{x_jx_k}\psi
	=\theta\left[
	\frac{1}{R^2}\eta''(\rho/R)
	\partial_{x_j}\rho\,\partial_{x_k}\rho
	+\frac{1}{R}\eta'(\rho/R)\partial^2_{x_jx_k}\rho
	\right].
	\]
	Therefore, on $\mathcal A_R$,
	\[
	\left|\nabla_X\psi\right|\le \frac{C\theta M_1}{R},
	\]
	\[
	\left|\partial_t\psi\right|\le \frac{C\theta M_1}{R^2},
	\]
	and
	\[
	\left|\nabla_X^2\psi\right|\le \frac{C_n\theta(M_1+M_2)}{R^2}.
	\]
	
	For the first spatial derivatives of $f$, the chain rule and \eqref{e:diff-g_alpha} gives
	\[
	\partial_{x_j}f
	=g_\psi e_j+\partial_{x_j}\psi\,K g_\psi X.
	\]
	Since $\left|K g_\psi X\right|\le \left|X\right|\le \rho\le  a R$ on $\mathcal A_R$,
	\[
	\left|\partial_{x_j}f\right|
	\le 1+C\theta M_1.
	\]
	Thus, for $1\le i\le n$,
	\[
	\left|\nabla_X F_i(X,t)\right|\le C,
	\]
	with $C$ depending only on the cutoff and dimension.
	
	For the second spatial derivatives, we differentiate once more and use \eqref{e:diff-g_alpha} to obtain
	\[
	\begin{aligned}
		\partial^2_{x_jx_k}f
		&=\partial_{x_k}\psi\,K g_\psi e_j
		+\partial_{x_j}\psi\,K g_\psi e_k \\
		&\quad +\partial_{x_j}\psi\,\partial_{x_k}\psi\,K^2g_\psi X
		+\partial^2_{x_jx_k}\psi\,K g_\psi X.
	\end{aligned}
	\]
	Using the bounds above and again using $\left|X\right|\le  a R$ on $\mathcal A_R$, we get
	\[
	\begin{aligned}
		\left|\partial^2_{x_jx_k}f\right|
		&\le \frac{C}{R}
		+\frac{C}{R}
		+\frac{C}{R^2}\left|X\right|
		+\frac{C}{R^2}\left|X\right| \\
		&\le \frac{C}{R}.
	\end{aligned}
	\]
	Therefore, for $1\le i\le n$,
	\[
	\left|\nabla_X^2 F_i(X,t)\right|\le \frac{C}{R}.
	\]
	
	Finally, for the time derivative,
	\[
	\partial_t f
	=\partial_t\psi\,K g_\psi X.
	\]
	Since $\left|\partial_t\psi\right|\le C/R^2$ and $\left|X\right|\le  a R$ on $\mathcal A_R$,
	\[
	\left|\partial_t f\right|\le \frac{C}{R}.
	\]
	Thus, for $1\le i\le n$,
	\[
	\left|\partial_t F_i(X,t)\right|\le \frac{C}{R}.
	\]
	
	Summarizing, the spatial components $F_i$, $1\le i\le n$, satisfy
	\[
	\left|\nabla_X F_i\right|\le C,
	\qquad
	\left|\nabla_X^2 F_i\right|\le \frac{C}{R},
	\qquad
	\left|\partial_t F_i\right|\le \frac{C}{R}.
	\]
	The constants depend on $n$ and the fixed cutoff $\eta$, and may be taken uniform in the planes because $0\le\theta\le\pi/2$.  Notice also that $F$ fixes the $t$ coordinate, so we have shown all the desired properties of $F$ aside from showing that each of the components of $F$ are regular Lip(1,1/2) functions. This is our task of the remainder of the proof.

	For the first part of (4), we show that $F$ is bi-Lipschitz. As we shall see below, $F$ has an inverse of the same form, so the heart of the matter is to show the $F$ is Lipschitz. To this end, we start by observing that 
	\[\rho(X,t) = \|(|X|, |t|^{1/2})\|_{\ell^4}.\]
	Thus, $\rho$ is $1$-Lipschitz; indeed, by the reverse triangle inequality,
	\begin{align*}
		|\rho(X,t) - \rho(Y,s)| &\le \|(|X|, |t|^{1/2}) - (|Y|, |s|^{1/2})\|_{\ell^4}
		\\& \le |X - Y| + |t -s|^{1/2} = \dist((X,t), (Y,s)).
	\end{align*}
	Consequently,
	\[
	|\psi(X,t)-\psi(Y,s)|
	\le \frac{\theta\|\eta'\|_{\infty}}{R}\dist((X,t),(Y,s)).
	\]
	\[
	\|g_\alpha-g_\beta\|_{\ell^2 \to \ell^2} = 2 \sin\left( \frac{|\alpha-\beta|}{2} \right) \le |\alpha - \beta|.
	\]
	A quick way to see this geometrically is to (without loss of generality) take $X$ to be a unit vector in $\Span\{e_1,e_2\}$ and, identifying $\Span\{e_1,e_2\}$ with $\mathbb{R}^2$, we have that $|g_\alpha X -g_\beta X|$ is just the chord distance between the vector $g_\alpha X$ and a rotation of $g_\alpha X$ with magnitude $|\alpha - \beta|$, which is the length of the arc from $g_\alpha X$ to $g_\beta X$.
	Using the previous estimate and the triangle inequality, we obtain 
	\[
	\begin{aligned}
		|g_{\psi(X,t)}X-g_{\psi(Y,s)}Y\|
		&\le |X-Y|
		+| g_{\psi(X,t)}-g_{\psi(Y,s)}| \min\{|X|,|Y|\} \\
		&\le |X-Y|
		+|\psi(X,t)-\psi(Y,s)| \min\{|X|,|Y|\}.
	\end{aligned}
	\]
	Here, in the first inequality, we have used the estimate
	\begin{align*}
		|g_{\psi(X,t)}X-g_{\psi(Y,s)}Y| &= |g_{\psi(X,t)}X-g_{\psi(Y,s)}X -g_{\psi(Y,s)}X - g_{\psi(Y,s)}Y| 
		\\ & \le |g_{\psi(X,t)}X-g_{\psi(Y,s)}X| + |X - Y| 
		\\& \le \|g_{\psi(X,t)}-g_{\psi(Y,s)}\||X| + |X - Y|
	\end{align*}
	and the analogous estimate where we add and subtract $g_{\psi(X,t)}Y$. Both of these estimates are obtained by the triangle inequality and the fact that $g_\alpha$ is an isometry. 
	If $\psi(X,t) \neq \psi(Y,s)$, then 
	\[\min\{|X|, |Y|\} \le \min\{\rho(X,t), \rho(Y,s)\} \le aR.\]
	Combining the estimates above, we have
	\[
	|g_{\psi(X,t)}X-g_{\psi(Y,s))}Y|
	\le |X -Y| + (a\theta\|\eta'\|_\infty\dist((X,t),(Y,s)).
	\]
	Since $F$ fixes the $t$-coordinate, we have 
	\begin{align*}
		\dist(F(X,t),F(Y,s)) &= |g_{\psi(X,t)}X-g_{\psi(Y,s))}Y| + |t - s|^{1/2}
		\\ & \le (1+a\theta \|\eta'\|_\infty) \dist((X,t),(Y,s))
		\\ & \le (1+2^{-1}a\pi \|\eta'\|_\infty) \dist((X,t),(Y,s)).
	\end{align*}
	This shows that $F$ is parabolic Lipschitz. Next, we show that $F$ is bi-Lipschitz. 
	
	First we observe that $\rho$ is is invariant under $F$. Indeed,
	\[
	\rho(G_\alpha(X,t))
	=\bigl(|g_\alpha X|^4+t^2\bigr)^{1/4}
	=\bigl(|X|^4+t^2\bigr)^{1/4}
	=\rho(X,t),
	\]
	which implies
	\[ \psi(F(X,t))=\psi(X,t).\]
	Define
	\[
	\widetilde F(X,t):=G_{-\psi(X,t)}(X,t)=(g_{-\psi(X,t)}X,t).
	\]
	Then
	\[
	\begin{aligned}
		\widetilde F(F(X,t))
		&=G_{-\psi(F(X,t))}G_{\psi(X,t)}(X,t) \\
		&=G_{-\psi(X,t)}G_{\psi(X,t)}(X,t) \\
		&=(X,t).
	\end{aligned}
	\]
	Similarly,
	\[
	F(\widetilde F(X,t))=(X,t).
	\]
	Thus
	\[
	F^{-1}(X,t)=G_{-\psi(X,t)}(X,t).
	\]
	In particular, $F^{-1}$ is exactly of the same form as $F$ and therefore $F^{-1}$ is parabolic Lipschitz with constant $L \le (1+2^{-1}a\pi \|\eta'\|_\infty)$. This shows that $F$ is a parabolic bi-Lipschitz map.

	The only thing left to show for (4) is that each spatial component is regular parabolic Lipschitz. From now on, we will identify $P_0$ with $\mathbb{R}^n$ and use the $\gamma$-numbers for each component of $F$. Fix $(z,\tau) \in \RR^n$, $\rho > 0$ and $j \in \{1,\dots, n\}$. Recall that we have already shown 
	\[|\nabla F_j| \lesssim \mathbbm{1}_{B(0,2R) \setminus B(0,R)}, \quad |\nabla^2 F_j| \lesssim R^{-1} \mathbbm{1}_{B(0,2R) \setminus B(0,R)}, \quad |\partial_t F_j| \lesssim R^{-1}\mathbbm{1}_{B(0,2R) \setminus B(0,R)}.\]
	For $(x,t) \in \mathbb{R}^n$, we let
	\[L_{x,t}(y,s) = F_j(x,t) + \nabla_x F(x,t)[y-x]\]
	be the spatial Taylor polynomial for $F_j$ about the point $(x,t)$. The derivative estimates above and Talyor's theorem give, for any parabolic ball $\Delta((x,t), r)$ in $\mathbb{R}^n$ with $r \le R$ and $(y,s) \in \Delta((x,t), r)$, that 
	\[|F_j(y,s) - L_{x,t}(y,s)| \lesssim_n \frac{|y-x|^2}{R} + \frac{|s-t|}{R} \lesssim \frac{r^2}{R}.\]
	Therefore,
	\[\gamma((x,t),r) \leq C_n \frac{r}{R}, \quad \forall  r < R.\]
	This gives the estimate
	\begin{equation}\label{eq:cptestup2rhobeta}
		\int_0^{\min\{\rho,R\}} \iint_{\Delta((z,\tau),\rho)} \gamma((x,t),r)^2 r^{-1} \, dxdt \, dr \leq C_n \rho^{n+1}.
	\end{equation}
	
	If $\rho \leq R$, the above estimate gives the desired Carleson estimate for the $\gamma$ numbers in the ball $B((z,\tau),\rho)$. Suppose instead that $\rho > R$. To treat the larger scales, we note that $F_j(x,t) = x_j$ for $(x,t) \not \in B(0,2R)$. Since $F_j(x,t)$ is globally Lip(1,1/2), we have the global estimate
	\[|F_j(x,t) -x_j| \le C_n R.\]
	Indeed, if $(x,t) \in B(0,2R)^c$ the bound holds trivially and, if not, we take any $(z',\tau') \in \partial B(0,2R)$ (where the estimate holds) and use the triangle inequality along with the estimate
	\[|F_j(x,t) - F_j(z',\tau')| \le C_n R.\]
	Thus, for $r > R$, using that $x_i$ is a spatial affine function,
	\[\gamma((x,t),r)^2 \le \frac{1}{r^{n+1}} \iint_{\Delta((x,t),r)} \frac{|F_j(y,s) - y_j|^2}{r^2} \, dy \, ds \leq C_n \frac{R^2}{r^{2}} .\]
	From this we deduce that 
	\begin{equation}\label{eq:cptestbgrrhobeta}
		\int_R^\rho  \iint_{\Delta((z,\tau),\rho)} \gamma((x,t),r)^2 r^{-1} \, dxdt \, dr \leq C_n \rho^{n+1}.
	\end{equation}
	Combining the estimates \eqref{eq:cptestup2rhobeta} and \eqref{eq:cptestbgrrhobeta}, we obtain the Carleson estimate for the $\gamma$ numbers in ball $B((z,\tau),\rho)$ in the case that $\rho > R$. 
\end{proof}

The next lemma is the mechanism that allows us to take a Lipschitz map that looks like one affine map outside of a ball, to another Lipschitz map that looks like a different affine map outside of a (larger) ball. The idea is to use the previous lemma to create a transition region where we spiral from looking like one plane to another.

\begin{lemma}\label{lem:givrotimages}
Suppose $P_0$ and $P_1$ are $t$-independent planes in $\mathbb{R}^{n+1}$ and
\[\Xbf_0 = (X_0,t_0) \in P_0, \quad \Xbf_1 = (X_1,t_1) \in P_1, \quad \text{ and } \quad r_* > 0.\]
Set
\[B_0 = B(\Xbf_0, r_*) \quad \text{ and } \quad B_1 = B(\Xbf_1, r_*).\]
Suppose that $F:P_1 \to \mathbb{R}^{n+1}$ is an $(L,M)$-parabolic regular bi-Lipschitz map such that
\[F(\Xbf) = \Xbf, \quad \forall \Xbf \in P_1 \setminus B_1.\]
Set 
\[\xi_L := 100 \max\{100, L_n(L +1)\},\]
where $L_n$ is from Lemma \ref{lem:givenrot}.
Then, there exists a constant $C_L < \infty$, depending on dimension and $L$, and a map $\widetilde{F}: P_0 \to \mathbb{R}^{n+1}$ with the following properties:
\begin{itemize}
\item[(1)] $\widetilde{F}$ is a $(C_L, C_n(M+1))$-parabolic regular bi-Lipschitz map.
\item[(2)] There exists an affine parabolic isometry $A: P_0 \to P_1$ such that
\[A(\Xbf_0) = \Xbf_1, \quad A(B_0 \cap P_0) = B_1 \cap P_1,\]
and
\[\widetilde{F}(\Xbf) = F(A\Xbf), \quad \forall \Xbf \in B_0 \cap P_0.\]
In particular,
\begin{equation}\label{eq:ftildesamim}
\widetilde{F}(B_0 \cap P_0) = F(B_1 \cap P_1).
\end{equation}
\item[(3)]
\[\widetilde{F}(\Xbf) = \Xbf + (\Xbf_1 - \Xbf_0), \quad \forall \Xbf \in P_0 \setminus B(\Xbf_0, \xi_L r_*).\]
In particular, if $\Xbf_0 = \pi_{P_0}(\Xbf_1)$ is the projection of $\Xbf_1$ onto $P_0$, then
\[\widetilde{F}(\Xbf) = \Xbf + c \vec{n}_{P_0} \quad \forall \Xbf \in P_0 \setminus B(\Xbf_0, \xi_L r_*),\]
where $\vec{n}_{P_0}$ is a normal vector to $P_0$.
\item[(4)] \[\widetilde{F}(B(\Xbf_0, \xi_L r_*)) \subset B(\Xbf_1, \xi_L r_*).\]
\end{itemize}
\end{lemma}
\begin{proof}
We will employ the smooth rotation map from Lemma \ref{lem:givenrot}. Let 
\[\widetilde{P}_0 = P_0 - \Xbf_0, \quad \widetilde{P}_1 = P_1 - \Xbf_1.\]
Let $G$ be the Givens rotation that takes $\widetilde{P}_0$ to $\widetilde{P}_1$, that is, $G_\theta$ in \eqref{eq:fullrotgiv}, and define $A: P_0 \to P_1$ by
\[A(\Xbf):= \Xbf_1 + G(\Xbf - \Xbf_0).\]
Notice that $A$ is a parabolic isometry and
\begin{equation}\label{eq:Amapsbis}
A(\Xbf_0) = \Xbf_1 \quad \text{ and } \quad A(B_0 \cap P_0) = B_1 \cap P_1.
\end{equation}

We let $H_R$ be the map from Lemma \ref{lem:givenrot} (called $F$ there), with $\widetilde{P}_i$ in place of $P_i$, $i = 0,1$, and $R$ to be chosen in a moment. Recall that the bi-Lipschitz constant of $H_R$ depends at most on dimension, and we called this $L_n$. We set
\[\xi_L := 100 \max\{100, L_n(L +1)\}\]
and 
\[R := \frac{\xi_L}{2}r_*.\]
Then, we define the map
\begin{equation}\label{eq:psideffromHR}
\Psi(\Xbf) = \Xbf_1 + H_R(\Xbf - \Xbf_0).
\end{equation}
By design, 
\begin{equation}\label{eq:psiAinsmallball}
\Psi(\Xbf) = A\Xbf, \quad \forall \Xbf  \in B(\Xbf_0, R)
\end{equation}
and, since $R > 50r_*$,
\[\Psi(\Xbf) = A\Xbf, \quad \forall \Xbf \in B_0 \cap P_0.\]
Also,
\begin{equation}\label{eq:psioutsidebb}
\Psi(\Xbf) = \Xbf + (\Xbf_1 - \Xbf_0), \quad \forall \Xbf \in B(\Xbf_0, 2R)^c = B(\Xbf_0, \xi_L r_*)^c. 
\end{equation}
Moreover, $\Psi$ is still $L_n$-bi-Lipschitz. We define our map $\widetilde{F}: P_0 \to \mathbb{R}^{n+1}$ now as
\begin{equation}\label{eq:tildefdef}
\widetilde{F}(\Xbf) = \Psi(\Xbf) + F(A\Xbf) - A\Xbf.
\end{equation}
We also define
\begin{equation}\label{eq:phidefinftild}
\Phi(\Xbf) = F(A\Xbf) - A\Xbf.
\end{equation}
Since $\Xbf \not\in B_0 \cap P_0$ implies $A\Xbf \not\in B_1 \cap P_1$, and $F(\Ybf) = \Ybf$ for $\Ybf \not \in B_1 \cap P_1$ (by assumption), we note that $\Phi$ is supported in $B_0 \cap P_0$. Since $\Phi(\Xbf) = 0$ outside of $B_0 \cap P_0$, using \eqref{eq:psioutsidebb}, we have
\begin{equation}\label{eq:tildefeventflat}
\widetilde{F}(\Xbf) = \Psi(\Xbf) = \Xbf + (\Xbf_1 - \Xbf_0), \quad \forall \Xbf \in P_0 \setminus B(\Xbf_0, \xi_L r_*),
\end{equation}
which is property (3). Also, since, as we have observed above, $\Psi(\Xbf) = A\Xbf$ on $B_0 \cap P_0$, it holds that
\[\widetilde{F}(\Xbf) = F(A\Xbf), \quad \forall \Xbf \in B_0 \cap P_0.\]
Considering \eqref{eq:Amapsbis}, we conclude that (2) holds. Also, we can verify that 
\[\widetilde{F}_{n+1}(X,t) = t - (t_1 -t_0),\]
where $t_i$ are the time coordinates of $\Xbf_i$, $i = 0,1$. To see this, we recall the definition of $F$, \eqref{eq:tildefdef}, all three maps $\Psi, F$ and $A$ involved in the formula act on the time variable (at most) by translation, so it must be $\widetilde{F}$ does as well, and \eqref{eq:tildefeventflat} forces the translation to be exactly $t - (t_1 -t_0)$. Thus, to prove (1), we need to show that the (overall) map is bi-Lipschitz and prove the components are regular.

We start by showing that $\widetilde{F}$ is bi-Lipschitz. Let $\Xbf, \Ybf \in P_0$. We divide into three cases, two of which are easy to handle. First, suppose $\Xbf, \Ybf \in B(\Xbf_0, R)$. In this case, we have by \eqref{eq:psiAinsmallball}, \eqref{eq:tildefdef} and the definition of $F$, that 
\[\widetilde{F}(\Xbf) = F(A\Xbf), \quad \widetilde{F}(\Ybf) = F(A\Ybf).\]
Since $A$ is a parabolic isometry, we have $\dist(F(A\Xbf), F(A\Ybf))  = \dist(F(\Xbf), F(\Ybf))$, so the fact that $F$ is $L$-bi-Lipschitz gives
\[L^{-1} \dist(\Xbf, \Ybf) \le \dist(\widetilde{F}(\Xbf), \widetilde{F}(\Ybf)) \le L \dist(\Xbf, \Ybf), \]
which is the desired estimate. 
Now suppose that $\Xbf, \Ybf \in B(\Xbf_0, r_*)^c = B_0^c$. Recall the definition of $\Phi$, \eqref{eq:phidefinftild} and our observation that $\Phi(\Xbf) = \Phi(\Ybf) = 0$ for $B_0^c$. Thus, 
\[\widetilde{F}(\Xbf) = \Psi(\Xbf), \quad \widetilde{F}(\Ybf) = \Psi(\Ybf),\]
so that, by the $L_n$-bi-Lipschitz property of $\Psi$, we have 
\[L_n^{-1} \dist(\Xbf, \Ybf) \le \dist(\widetilde{F}(\Xbf), \widetilde{F}(\Ybf)) \le L_n \dist(\Xbf, \Ybf).\]

The remaining case is when $\Xbf \in B_0$ and $\Ybf \in B(\Xbf_0, R)^c$. In this case, we let $\Zbf \in \partial B(\Xbf_1, 2r_*) \cap P_1$ be arbitrary. By \eqref{eq:Amapsbis}, we have $A\Xbf \in B(\Xbf_1, r_*) \cap P_1$, so that
\[\dist(A\Xbf, \Zbf) \le 3r_*, \quad \dist(\Zbf, \Xbf_1) = 2r_*.\]
Then, $F(\Zbf) = \Zbf$, $F$ is $L$-Lipschitz and, as we observed in the first case, $\widetilde{F}(\Xbf) = F(A\Xbf)$, so that
\begin{equation}\label{eq:dagftildlem}
\begin{aligned}
\dist(\widetilde{F}(\Xbf), \Xbf_1) &= \dist(F(A\Xbf), \Xbf_1) \le  \dist(F(A\Xbf), F(\Zbf)) +  \dist(\Zbf,  \Xbf_1) \\
& \le L\dist(A\Xbf, \Zbf) + 2r_* \le (3L + 2)r_*.
\end{aligned}
\end{equation} 
On the other hand, as we have observed in the second case above, $\widetilde{F}(\Ybf) = \Psi(\Ybf)$ and, by definition, $\Psi(\Xbf_0) = \Psi(\Xbf_1)$, so that
\[\dist(\widetilde{F}(\Ybf), \Xbf_1) = \dist(\Psi(\Ybf), \Psi(\Xbf_0)).\]
Since $\Psi$ is $L_n$-bi-Lipschitz, this gives 
\begin{equation}\label{eq:dbldagftildlem}
L_n^{-1} \dist(\Ybf, \Xbf_0) \le   \dist(\widetilde{F}(\Ybf), \Xbf_1) \le L_n \dist(\Ybf, \Xbf_0).
\end{equation}
Combining \eqref{eq:dagftildlem} and \eqref{eq:dbldagftildlem}, and using the triangle inequality, we obtain two estimates
\begin{equation}\label{eq:starftildlem}
\dist(\widetilde{F}(\Xbf), \widetilde{F}(\Ybf)) \ge L_n^{-1}\dist(\Ybf, \Xbf_0) - (3L + 2)r_*
\end{equation}
and 
\begin{equation}\label{eq:heartftildlem}
\dist(\widetilde{F}(\Xbf), \widetilde{F}(\Ybf)) \le L_n \dist(\Ybf, \Xbf_0) + (3L + 2)r_*.
\end{equation}
Also, the choice of $R = \xi_L/2  \ge 50r_*$, along with the fact that $\Xbf \in B_0 = B(\Xbf_0, r_*)$ and $\Ybf \in B(\Xbf_0, R)$, gives
\begin{equation}\label{eq:obvcompctr}
(1/2)\dist(\Xbf, \Ybf) \le \dist(\Xbf_0, \Ybf) \le \dist(\Xbf, \Ybf).
\end{equation}
Moreover, the choice of $\xi_L$ gives
\[\dist(\Ybf, \Xbf_0) \ge R = \tfrac{\xi_L}{2}r_* \ge 50(L + 1)L_n r_*.\]
Using this estimate in conjunction with \eqref{eq:starftildlem} gives 
\begin{align*}
\dist(\widetilde{F}(\Xbf), \widetilde{F}(\Ybf)) &\ge (1/2)L_n^{-1}\dist(\Ybf, \Xbf_0) + (1/2) L_n^{-1}\dist(\Ybf, \Xbf_0) - (3L + 2)r_* \\
&\ge (1/2)L_n^{-1}\dist(\Ybf, \Xbf_0) \ge_{\eqref{eq:obvcompctr}} (1/4)L_n^{-1}\dist(\Ybf, \Xbf).
\end{align*}
Since the choice of $\xi_L$ gives $(3L + 2)r_* \le R \le \dist(\Ybf, \Xbf_0)$, \eqref{eq:heartftildlem} gives,
\[\dist(\widetilde{F}(\Xbf), \widetilde{F}(\Ybf)) \le (L_n + 1) \dist(\Ybf, \Xbf_0) \le_{\eqref{eq:obvcompctr}} 2(L_n + 1)\dist(\Xbf, \Ybf), \]
which completes the bi-Lipschitz bound in this final case. Altogether, we have found that $\widetilde{F}$ is bi-Lipschitz with constant at most $\max\{L, 4L_n, 2L_n + 2\}$. 

We now check that the spatial components of $\widetilde{F}$ are regular. To simplify notation, we will identify $P_0$ with $\mathbb{R}^n$.
We note that, by \eqref{eq:tildefdef},
\[\widetilde{F}_j(\Xbf) = \Psi_j(\Xbf) + F_j(A\Xbf) - (A\Xbf)_j\]
for $j = 1,\dots n$ and $\Xbf \in \mathbb{R}^n$. The triangle inequality and the fact that $A$ is a parabolic isometry gives 
\begin{align*}
\gamma_{\widetilde{F}_j}((x,t),r) &\le \gamma_{\Psi_j}((x,t),r) + \gamma_{F_j(A\cdot)}((x,t),r) + \gamma_{(A\cdot)_j}((x,t),r) 
\\ &=   \gamma_{\Psi_j}((x,t),r) + \gamma_{F_j}((x,t),r).
\end{align*}
Here, $\gamma_{G}$ is the $\gamma$ coefficient for the function $G$. Now, we know that $F_j$ is $M$-regular by assumption, that is,
\[\gamma_{F_j}((x,t),r)^2\frac{dx \, dt\, dr}{r}\]
is a Carleson measure.
Moreover, by \eqref{eq:psideffromHR},
\[\gamma_{\Psi_j}((x,t),r) = \gamma_{(H_R)_j}((x - x_0,t - t_0),r),\]
where $H_R$ is the smooth rotation provided by Lemma \ref{lem:givenrot}. Since Lemma \ref{lem:givenrot} gives that $H_R$ is a regular parabolic bi-Lipschitz map, we have that
\[ \gamma_{(H_R)_j}((x,t),r)^2 \frac{dx \, dt\, dr}{r}\]
is a Carleson measure. Combining these estimates, it holds that
\[\gamma_{\widetilde{F}_j}((x,t),r)^2\frac{dx \, dt\, dr}{r} \]
is a Carleson measure, which completes the proof of the lemma aside from (4).

It remains to show that 
\[\widetilde{F}(B(\Xbf_0, \xi_L r_*) \cap P_0) \subset B(\Xbf_1, \xi_L r_*).\]
We have essentially seen this above, but we reproduce it here.  We first note that, if $\Xbf \not \in B_0$, then
\[\widetilde{F}(\Xbf) = \Psi(\Xbf) = \Xbf_1 + H_R(\Xbf - \Xbf_0).\]
On the other hand, the last property of $H_R$ in Lemma \ref{lem:givenrot} says $H_R(B(0,r)) = B(0,r)$ for all $r > 0$ and therefore
\[\Psi(B(\Xbf_0,r) \cap P_0) \subset B(X_1,r), \quad \forall r > 0.\]
Therefore, we only need to consider the case that $\Xbf \in B_0$. In this case, we have $\widetilde{F}(\Xbf) = F(A\Xbf)$. Let $\Zbf \in \partial B(\Xbf_1, 2r*) \cap P_1$ be arbitrary. Since $\Zbf \not \in B_1$, we have 
\[F(\Zbf) = \Zbf\]
and 
\[\dist(A\Xbf, \Zbf) = \dist(A\Xbf, \Xbf_1) + \dist(\Xbf_1, \Zbf) \le 3r_*.\]
Using that $F$ is $L$-Lipschitz, this gives 
\begin{align*}
\dist(\widetilde{F}(\Xbf), \Xbf_1) &\le \dist(F(A\Xbf), \Xbf_1)  \le \dist(F(A\Xbf), F(\Zbf) )  + \dist(\Zbf, \Xbf_1) 
\\ & \le L \dist(A\Xbf, \Zbf) + \dist(\Zbf, \Xbf_1) \le (3L + 2)r_* \le \xi_L r_*.
\end{align*}
This shows (4) and completes the proof of the lemma.
\end{proof}

\subsection{Gluing Eventually Flat Lipschitz Images Together}

We are going to want to glue several Lipschitz images together which look like planes outside of a ball. After applying the previous lemma to rotate them to be parallel, we will tie them together using a regular Lip(1,1/2) graph afforded by the corona decomposition. The following lemma is useful in this regard.

\begin{lemma}\label{lem:constreplem}
	Let $L,M,K \geq 0$. Suppose that $B^*$ is a fixed parabolic ball in $n$-dimensional space time $\mathbb{R}^n$ and $\{Q_i\}$ is a finite collection of parabolic cubes in $\mathbb{R}^n$ such that  
	\begin{align}
		4Q_i \subseteq B^* \quad \forall i 
	\end{align} 
	and
	\begin{equation}\label{eq:sepcubesprop}
		d_\infty(Q_i, Q_j) \ge 10^5\max\{\diam_\infty(Q_i), \diam_\infty(Q_j)), \quad \forall i \neq j,
	\end{equation}
	where $\diam_\infty$ is the diameter measured in $d_\infty$. Suppose further that $\theta \colon \mathbb{R}^n \to \mathbb{R}$ is an $(L,M)$-regular Lip(1,1/2) function with support in $B^*$ such that, for every $i$ there exists $c_i^*$ such that  
		\begin{equation}\label{eq:closeness}
			|\theta(x,t) - c_i^*| \le K \ell(Q_i), \quad \forall (x,t) \in 50Q_i.
		\end{equation}
		Then, there exists a $(C_n(L+K),C_n(M+L^2+K^2))$-regular parabolic Lipschitz function $g$ with support in $B^*$ such that 
		\[g(x,t) = c_i^*, \quad \forall (x,t) \in Q_i.\]
\end{lemma}

\begin{proof}
	For each $i$, we let \(\eta_i\in C_0^\infty(2 Q_i)\) be a cutoff with
	\begin{equation}\label{eq:eta}
		0\le \eta_i\le 1,
		\qquad
		\eta_i\equiv 1 \text{ on } Q_i,
	\end{equation}
	and parabolic derivative bounds
	\begin{equation}\label{eq:eta-derivatives}
		|\nabla_x^\alpha \partial_t^m\eta_i|
		\le C_{\alpha,m} \ell(Q_i)^{-|\alpha|-2m}.
	\end{equation}
	We define 
	\[h(\xbf) = \sum_i \eta_i(\xbf)(c_i^* - \theta(\xbf)).\] 
	Our goal is to show that 
	\begin{align}\label{e:regularity-of-h}
		\mbox{$h$ is $(C_n(L+K),C_n(M+L^2+K^2))$-regular parabolic Lipschitz.}
	\end{align}
	Then, we can simply let 
	\[g = \theta + h\]
	to prove Lemma \ref{lem:constreplem}. Property \eqref{e:regularity-of-h} follows directly from Lemma \ref{l:gluing-for-regular-functions-prime} once we check the hypothesis (H2) there. The separation condition \eqref{e:separation-gluing-prime} is exactly \eqref{eq:sepcubesprop}. The collection $\{\eta_i\}$ satisfy the conditions of Lemma \ref{l:gluing-for-regular-functions-prime} (H2) by definition. Finally, defining $\theta_i \coloneqq \theta - c_i^*$, the collection $\{\theta_i\}$ satisfies the conditions of Lemma \ref{l:gluing-for-regular-functions-prime} (H2) by \eqref{eq:closeness} and the fact that $\theta$ is regular Lip(1,1/2) (shifts by constant functions do not change this condition). 
	\end{proof}

\begin{lemma}\label{lem:flatpatching}
Let $L \ge 1$ and $M \ge 0$. Suppose that $\{Q_i\}$ is a finite collection of (closed) parabolic cubes in $\mathbb{R}^n$ with the property that 
\begin{equation}\label{eq.sepforflatpatch}
d_\infty(Q_i,Q_j) \ge 10^5 \max \{ \diam_\infty(Q_i), \diam_\infty(Q_j)\}= 10^5 \max \{\ell(Q_i), \ell(Q_j)\}.
\end{equation}
Suppose further that $g$ and $\{h_i\}$ are $(L,M)$-regular Lip(1,1/2) functions with
\begin{equation}\label{eq:higaggreepatchlem}
h_i = g, \quad \text{ on } 10Q_i \setminus Q_i, 
\end{equation}
and define
\[f(\Xbf) = \begin{cases}
h_i(\Xbf), & \text{ if } \Xbf \in Q_i \\
g(\Xbf) & \text{ if } \Xbf \in (\cup_i Q_i)^c.
\end{cases}
\]
Then $f$ is a $(C_nL,C_n(M+L^2))$-regular Lip(1,1/2) function.
\end{lemma}
\begin{proof}
Define
\[\varphi_i(\Xbf) = \begin{cases}
h_i(\Xbf) - g(\Xbf) & \text { if } \Xbf \in Q_i \\
0 & \text{ if } \Xbf \in Q_i^c.
\end{cases}\]
Note that, by hypothesis,
\begin{align}\label{e:phi-i}
	\varphi_i = h_i - g, \quad \text{ on } 10Q_i.
\end{align}
Then 
\[f = g + \sum_i \varphi_i\]
where, since the sum is finite, there are no issues of convergence. Thus, it suffices to show that 
\begin{align}\label{e:varphi-regular}
	\varphi \coloneqq \sum_i \varphi_i \mbox{ is } (C_nL , C_n(M+L^2))\mbox{-regular parabolic Lipschitz.} 
\end{align}
To prove this, we plan to apply Lemma \ref{l:gluing-for-regular-functions-prime}, for which we will check the hypothesis (H1). To start with, note that
\begin{align}\label{e:lip-phi}
	\mbox{$\varphi_i$ is $2L$ Lipschitz} \quad \mbox{and} \quad \supp(\varphi_i) \subseteq Q_i. 
\end{align}
Indeed, the second property follows directly from \eqref{e:phi-i}. For the first property, if $\Xbf, \Ybf \in Q_i$, then 
\[|\varphi_i(\Xbf) - \varphi_i(\Ybf)| \le |h_i(\Xbf) - h_i(\Ybf)| + |g(\Xbf) - g(\Ybf)| \le 2L\dist(\Xbf, \Ybf)\]
and, if $\Xbf, \Ybf \in (Q_i)^c$, then $\varphi_i(\Xbf) = \varphi_i(\Ybf) = 0$. Finally, in the case that $\Xbf \in Q_i$ and $\Ybf \in Q_i$, let $\Zbf \in \partial Q_i$ lie on the line segment from $\Xbf$ to $\Ybf$, so that
\[|\varphi_i(\Xbf) - \varphi_i(\Ybf)| = |\varphi_i(\Xbf)| \le |\varphi_i(\Xbf) - \varphi_i(\Zbf)| \le 2L\dist(\Xbf, \Zbf) \le 2L\dist(\Xbf, \Ybf).\]
Finally, for any parabolic cubes $Q,R_1$ such that $Q \subseteq 4R_1$ and $R_1 \subseteq 6Q_i$, it follows from the triangle inequality that $3Q \subseteq 10Q_i$. Since $\phi_i$ is the sum of two $(L,M)$-regular functions on $10Q_i$, we have that \eqref{e:loca-regular-lipschitz} holds for the cube $R_1$. This and \eqref{e:lip-phi} imply 
\begin{align}
	\varphi_i \mbox{ is } (2L,2M)\mbox{-regular parabolic Lipschitz on } 10Q_i
\end{align}
(recall Definition \ref{d:local-regular-lipschitz}). 

It follows from the above observations and \eqref{eq.sepforflatpatch} that the cubes $\{Q_i\}$ and the function $\{\varphi_i\}$ satisfy the hypotheses of Lemma \ref{l:gluing-for-regular-functions-prime} (H1) with parameters $``(L,M)"$ equal to $(2L,2M)$. Property \eqref{e:varphi-regular} now follows immediately. 
\end{proof}

Now we are ready to prove the Main Lemma below. 

\begin{lemma}[Main Lemma]\label{lem:main}
Let $\Sigma \subset \mathbb{R}^{n+1}$ be a (parabolic) ADR set. Let $M, M_\theta > 0$ and $L,\tilde{\kappa} > 1$. Then there exists constants $\kappa^*,N,L_* > 1$ and $M_* > 0$, depending on $L$, $\tilde{\kappa}$, $M$, $M_\theta$, dimension, and the ADR constant, such that the following holds.

Suppose that $Q_0 \in \mathbb{D}(\Sigma)$ and that $\{Q_i\}$ is a finite, disjoint collection dyadic proper\footnote{Meaning $Q_i \neq Q_0$.} subcubes of $Q_0$  with the property that
\begin{equation}\label{eq:sepinmainlem}
\dist(Q_i, Q_j) \ge N \max\{\diam(Q_i), \diam(Q_j)\}.
\end{equation}
Suppose, for each $Q_i$, that there exists a $t$-independent plane $P_{Q_i}$ and a mapping $F_{Q_i}: P_{Q_i} \to \mathbb{R}^{n+1}$ with the following properties:
\begin{enumerate}
\item[(i)] $\dist((X_{Q_i},t_{Q_i}), P_{Q_i}) < \tilde{\kappa}\diam(Q_i)$.
\item[(ii)] $F_{Q_i}$ is a $(L,M)$-regular parabolic bi-Lipschitz map of $n$-dimensional space-time $\mathbb{R}^n$ (upon identifying $P_{Q_i}$ with $\mathbb{R}^n$).
\item[(iii)] If $\Xbf_{i,1} = \pi_{P_{Q_i}}(X_{Q_i},t_{Q_i})$ is the projection of the `center' of $Q_i$ onto $P_{Q_i}$, then 
\begin{enumerate}
\item[(a)] $F_{Q_i}$ is the identity in $B({\Xbf}_{i,1}, \tilde{\kappa}\diam(Q_i))^c \cap P_{Q_i}$,
\item[(b)] $F_{Q_i}(B(\Xbf_{i,1}, \tilde{\kappa}\diam(Q_i))) \subseteq B(\Xbf_{i,1}, \tilde{\kappa}\diam(Q_i))$. 
\end{enumerate}
\end{enumerate}
Suppose further that there exists a $t$-independent plane $P_0 = P_{Q_0}$ and an $(1, M_\theta)$-regular function $\theta$ on $P_0$ such that the graph of $\theta$ (in the coordinates $P_0^\perp \times P_0$, written $(\theta(x,t), x,t)$), which we denote by $\Gamma$, satisfies 
\begin{equation}\label{eq:closegraphcoronalem}
\sup_{(X,t)\in 10 Q} \dist(X,t,\Gamma )\, + \,  \sup_{(Y,s)\in B((X_{Q},t_{Q}),10 \diam(Q)) \cap \Gamma} \dist(Y,s, \Sigma) \,
\leq\, \diam(Q),
\end{equation}
for all $Q \in \{Q_i^*\}_i \cup Q_0$, where $Q_i^*$ is the dyadic grandparent of $Q_i$.
Further suppose that 
\begin{equation}\label{eq:closeplanecoronalem} 
	P_0 \cap B((X_{Q_0},t_{Q_0}),\tilde{\kappa} \diam({Q_0})) \neq \emptyset
\end{equation}
and
\begin{equation}\label{eq:supportgraphlem1}
\supp(\theta) \subseteq P_0 \cap B((X_{Q_0},t_{Q_0}),\tilde{\kappa} \diam({Q_0})).
\end{equation}

Then there exists a map $ F^* :P_0 \to \mathbb{R}^{n+1}$ such that the following holds:
\begin{itemize}
\item[(i)$^*$] $\dist((X_{Q_0},t_{Q_0}), P_0) < \kappa^*\diam(Q_0)$. 
\item[(ii)$^*$]  $F^*$ is an $(L_*, M_*)$-regular parabolic bi-Lipschitz map.
\item[(iii)$^*$] If $\Xbf_{1} = \pi_{P_0}(X_{Q_0},t_{Q_0})$ is the projection of the `center' of $Q_0$ onto $P_0$, then 
\begin{enumerate}
\item[(a)] $F^*$ is the identity in $B({\Xbf}_1, \kappa^*\diam(Q_0))^c \cap P_0,$
\item[(b)] $F^*(B(\Xbf_{1}, \kappa^*\diam(Q_0))\cap P_0) \subseteq B(\Xbf_1, \kappa^*\diam(Q_0))$. 
\end{enumerate}
\item[(iv)$^*$] \[\cup_i F_{Q_i}(B({\Xbf}_{i,1}, \tilde{\kappa}\diam(Q_i)) )\cap P_{Q_i})) \subset F^*(B(\Xbf_1, \kappa^*\diam(Q_0) \cap P_0).\]
\end{itemize}

\end{lemma}
\begin{proof} 
Without loss of generality (change of coordinates) we assume that 
\[P_0 = \{0\} \times \mathbb{R}^n,\]
with normal vector chosen to be $n_{P_0} = e_1$.
We will often abuse notation identifying points in $P_0$ with points in $\mathbb{R}^n$. For instance, we will make the identification
\[\theta(0,x_2, \dots, x_n, t) = \theta(x_2, \dots, x_n, t)\]
for the approximating graph function.

We do not need to construct $F^*$ to observe that (i)$^*$ holds. Indeed, this just follows from \eqref{eq:closegraphcoronalem} with $Q = Q_0$,  \eqref{eq:supportgraphlem1}, and the fact that $\theta$ is $1$-Lipschitz, provided that we take $\kappa^*$ sufficiently large, depending on $\tilde{\kappa}$. 
Let
\[\Xbf_{i,0} = \pi_{P_0}(\Xbf_{i,1})\]
be the projection of $\Xbf_{i,1}$ onto the plane $P_0$. We set
\[B_{i,1} = B({\Xbf}_{i,1}, \tilde{\kappa}\diam(Q_i)),\]
\[B_{i,0} = B({\Xbf}_{i,0}, \tilde{\kappa}\diam(Q_i)),\]
and
\[P_{i,1} = P_{Q_i},\]
with the intention of using Lemma \ref{lem:givrotimages}. We let $\widetilde{F}_i$ be the map provided by Lemma \ref{lem:givrotimages} using the map $F_{Q_i}$ as the function $F$, $B_{i,k}$ as the balls $B_k$ for $k = 0,1$, and the plane $P_{i,1}$ for $P_1$. From that lemma, we deduce that $\widetilde{F}_i: P_0 \to \mathbb{R}^{n+1}$ is an $(\widetilde{L}, \widetilde{M})$-regular bi-Lipschitz map, 
\begin{equation}\label{eq:widetildeimpres}
\widetilde{F}_i(B_{i,0} \cap P_0) = F_{Q_i}(B_{i,1} \cap P_1),
\end{equation}
and that
\begin{equation}\label{eq:tildeshiftoutball}
\widetilde{F}_i(\Xbf) = \Xbf + c_i^* n_{P_0}, \quad \forall \Xbf \in (\xi_L B_{i, 0})^c \cap P_0, 
\end{equation}
where $c_i^*$ is such that
\begin{equation}\label{eq:shiftrel}
\Xbf_{i,0} + c^*_i n_{P_0} = \Xbf_{i,1}.
\end{equation}
Moreover, we have that 
\begin{equation}\label{eq:imcontainimp1}
\widetilde{F}_i(\xi_L B_{i, 0} \cap P_0)  \subseteq \xi_L B_{i, 1}.
\end{equation}
Using \eqref{eq:imcontainimp1} and \eqref{eq:tildeshiftoutball} it holds that if $E \subset P_0$ with $\xi_L B_{i, 0} \cap P_0 \subset E$ then
\begin{equation*}
\widetilde{F}_i(E) \subset \xi_L B_{i, 1} \cup (E +  c_i^* n_{P_0}).
\end{equation*}
In particular,
\begin{equation}\label{eq:projcontain}
\pi_{P_0}(\widetilde{F}_i(E)) \subseteq E. 
\end{equation}
Indeed, \eqref{eq:imcontainimp1} says exactly 
\[\widetilde{F}(\xi_L B_{i, 0} \cap P_0)  \subseteq \xi_L B_{i, 1},\]
and, for $\xbf \in E \setminus \xi_L B_{i, 0}$, \eqref{eq:tildeshiftoutball} says
\[F(\xbf) = \xbf + c_i^* n_{P_0}.\]

Our next goal is to use Lemma \ref{lem:constreplem} to create a new graph to patch the $F_i$'s together. We want to ensure first that the $c_i^*$ we have defined above are suitable. To begin, we recall that we have assumed $P_0 = \{0\} \times \mathbb{R}^n$ so that using \eqref{eq:shiftrel} we can write
\begin{equation}\label{eq:1isshift}
\Xbf_{i,1} = (c_i^*, x_2^{(i,1)}, \dots, x_n^{(i,1)}, t^{(i,1)}) = (c_i^*, x_2^{(i,0)}, \dots, x_n^{(i,0)}, t^{(i,0)}),
\end{equation}
where 
\begin{equation}\label{eq:x0coords}
\Xbf_{i,0} = (0, x_2^{(i,0)}, \dots, x_n^{(i,0)}, t^{(i,0)}). 
\end{equation}
The hypothesis  of graphical approximation by $\theta$, \eqref{eq:closegraphcoronalem} gives that 
\begin{equation}\label{eq:centergraphcloseness}
\dist((X_{Q_i},t_{Q_i}), \theta(\pi_{P_0}(X_{Q_i},t_{Q_i})))  \approx \dist((X_{Q_i},t_{Q_i}), \Gamma)  \lesssim \diam(Q_i),
\end{equation}
where the implicit constant depends only on dimension. Here, we have used the fact that the ``vertical" distance to a Lipschitz graph is comparable to the distance to a Lipschitz graph. 
Note also, by (i), 
\begin{equation}\label{eq:iiiaagain}
\dist(\Xbf_{i,1}, (X_{Q_i}, t_{Q_i})) \le \tilde{\kappa} \diam(Q_i). 
\end{equation}
Then, using this fact again (actual distance versus vertical distance to graphs), the estimates \eqref{eq:centergraphcloseness} and \eqref{eq:iiiaagain}, and the observation \eqref{eq:1isshift}  (see \eqref{eq:x0coords}), we have
\begin{equation}\label{eq:ciclosever1}
\begin{aligned}
|\theta( x_2^{(i,0)}, \dots, x_n^{(i,0)}, t^{(i,0)}) - c_i^*| &= |\theta( x_2^{(i,1)}, \dots, x_n^{(i,1)}, t^{(i,1)}) - c_i^*| \approx \dist(\Gamma, \Xbf_{i,1})
\\ & \le \dist(\Gamma, (X_{Q_i},t_{Q_i})) + \dist((X_{Q_i},t_{Q_i}), \Xbf_{i,1}) \\
&\lesssim \diam(Q_i),
\end{aligned}
\end{equation}
where the implicit constants depends at most on dimension and $\tilde{\kappa}$. Now we set 
\begin{equation}\label{eq:tildeQ-contains-ball}
	\mathcal{Q}_i \coloneqq Q(\Xbf_{i,0}, \xi_L \tilde{\kappa}\diam(Q_i)) \cap P_0 \supseteq \xi_L B_{i, 0} \cap P_0,
\end{equation}
which we view an $n$-dimensional space-time cube in $P_0$. Since $\theta$ is $1$-Lipschitz, $\diam(\mathcal{Q}_i) \approx \diam(Q_i)$ and 
\[(0, x_2^{(i,0)}, \dots, x_n^{(i,0)}, t^{(i,0)})\]
is the center of $\mathcal{Q}_i$, the estimate \eqref{eq:ciclosever1} yields the estimate
\begin{equation}\label{eq:ciclosever2}
|\theta( x_2, \dots, x_n, t) - c_i^*| \lesssim \diam(\mathcal{Q}_i), \quad \forall (0,x_2, \dots, x_n, t) \in 10^5\mathcal{Q}_i.
\end{equation}
Set 
\[\mathcal{Q}_i^* \coloneqq 10^2\mathcal{Q}_i.\]

We want to apply Lemma \ref{lem:constreplem} to the with the cubes $\mathcal{Q}_i^*$ in place of $Q_i$ therein. The estimate \eqref{eq:ciclosever2} implies \eqref{eq:closeness} with a constant only depending on dimension and $\tilde{\kappa}$. Let $\kappa \in (1,\kappa^*)$ (to be chosen sufficiently large) and consider the ball 
\begin{align}
	B^* \coloneqq B(\pi_{P_0}(X_{Q_0},t_{Q_0}), \kappa \diam(Q_0)). 
\end{align} 
For $\kappa$ large enough, depending on $\tilde{\kappa}$, it follows by assumption that $\theta$ has compact support in $B^*$. Moreover, it is clear that if we choose $\kappa$ sufficiently large, depending on dimension, $L$ and $\tilde{\kappa}$, it holds that
\begin{equation}\label{eq:4mcqicontain}
	10\mathcal{Q}_i^* \subseteq B^* \subseteq B({\Xbf}_1, \kappa^* \diam(Q_0)).
\end{equation}
Indeed, recall that $\Xbf_{i,0}$ is the center of $\mathcal{Q}_i^*$ a cube with radius comparable to the $\diam(Q_i)$ and it holds
\begin{align*}
	\dist(\Xbf_{i,0}, \pi_{P_0}(X_{Q_0},t_{Q_0})) &=  \dist(\pi_{P_0}(\Xbf_{i,1}), \pi_{P_0}(X_{Q_0},t_{Q_0})) \\
	& \le \dist(\Xbf_{i,1}, (X_{Q_0},t_{Q_0})) \le \dist(\Xbf_{i,1}, (X_{Q_i},t_{Q_i})) 
	\\ & \qquad  + \dist( (X_{Q_i},t_{Q_i}), (X_{Q_0},t_{Q_0})) \\
	& \le \tilde{\kappa} \diam(Q_i) + \diam(Q_0),
\end{align*}
where we used (i) and that $(X_{Q_i},t_{Q_i}) \in Q_i \subseteq Q_0$. From these facts and the triangle inequality, it is easy to deduce \eqref{eq:4mcqicontain} holds. 

It only remains to verify \eqref{eq:sepcubesprop} holds for the collection $\{\mathcal{Q}_i^*\}$. This is where we start employing the constant $N$, which we can choose at our disposal (provided it only depends on $L$, $\tilde{\kappa}$ and dimension). In particular, we show that if $N$ is taken sufficiently large
\begin{equation}\label{eq:calqsep1}
d_\infty(\mathcal{Q}_i^*,\mathcal{Q}_j^*) \ge 10^5 \max\{\diam_\infty(\mathcal{Q}_i^*)  , \diam_\infty(\mathcal{Q}_j^*)\} \quad \text{ for } i \neq j.
\end{equation}
Note that this also implies 
\begin{equation}\label{eq:calqsep2}
d_\infty(10\mathcal{Q}_i,10\mathcal{Q}_j) \ge 10^5 \max\{\diam_\infty(10\mathcal{Q}_i), \diam_\infty(10\mathcal{Q}_j)\} \quad \text{ for } i \neq j.
\end{equation}
Let $i \neq j$ below and define
\[\Theta(x_2,\dots, x_n, t) \coloneqq (\theta(x_2,\dots, x_n, t), x_2,\dots, x_n, t),\]
the function whose image is the graph of $\theta$. Since $\theta$ is $1$-Lipschitz, $\Theta$ is $2 = 1 + 1$ bi-Lipschitz. In particular, using that $\Theta$ is $2$-bi-Lipschitz and the triangle inequality along with \eqref{eq:centergraphcloseness} (for both $Q_i$ and $Q_j$) we obtain 
\begin{equation}\label{eq:flatcentsep}
\begin{aligned}
& \dist(\pi_{P_0}(X_{Q_i}, t_{Q_i}), \pi_{P_0}(X_{Q_j}, t_{Q_j})) \ge (1/2)\dist(\Theta(\pi_{P_0}(X_{Q_i}, t_{Q_i})), \Theta(\pi_{P_0}(X_{Q_j}, t_{Q_j})))
\\ & \quad \ge (1/2)\dist((X_{Q_i}, t_{Q_i}), (X_{Q_j}, t_{Q_j})) - (1/2)[\dist((X_{Q_i}, t_{Q_i}), \Theta(\pi_{P_0}(X_{Q_i}, t_{Q_i}))) 
\\ & \qquad + \dist((X_{Q_j}, t_{Q_j}), \Theta(\pi_{P_0}(X_{Q_j}, t_{Q_j})))
\\ & \quad \ge_{\eqref{eq:centergraphcloseness}} (N/2)\max\{\diam(Q_i),\diam(Q_j)\} - C[\diam(Q_i) + \diam(Q_j)]
\\ & \quad \ge (N/4) \max\{\diam(Q_i),\diam(Q_j)\},
\end{aligned}
\end{equation}
provided $N$ is chosen large enough. Thus, using \eqref{eq:iiiaagain} (for both $Q_i$ and $Q_j$), we have
\begin{equation}\label{eq:centsep1}
\begin{aligned}
\dist(\Xbf_{i,0}, \Xbf_{j,0}) &= \dist(\pi_{P_0}(\Xbf_{i,1}), \pi_{P_0}(\Xbf_{j,1}))
\\ & \ge \dist(\pi_{P_0}(X_{Q_i}, t_{Q_i}), \pi_{P_0}(X_{Q_j}, t_{Q_j})) - \dist(\pi_{P_0}(\Xbf_{i,1}), \pi_{P_0}(X_{Q_i}, t_{Q_i})) 
\\ & \qquad -  \dist(\pi_{P_0}(\Xbf_{j,1}), \pi_{P_0}(X_{Q_j}, t_{Q_j}))
\\ & \ge_{\eqref{eq:flatcentsep}, \eqref{eq:iiiaagain}} (N/4) \max\{\diam(Q_i),\diam(Q_j)\} - C[\diam(Q_i) + \diam(Q_j)]
\\ & \ge (N/8) \max\{\diam(Q_i),\diam(Q_j)\},
\end{aligned}
\end{equation}
provided $N$ is sufficiently large. Now, since $\diam_\infty(\mathcal{Q}_i^*) \approx \diam(Q_i)$ and $\diam_\infty(\mathcal{Q}_j^*) \approx \diam(Q_j)$ with constants depending on dimension, $L$ and $\tilde{\kappa}$, it holds that
\[\dist_\infty(\Xbf_{i,0}, \Xbf_{j,0}) \ge cN\max\{\diam_\infty(\mathcal{Q}_i^*),\diam_{\infty}(\mathcal{Q}_j^*)\}.\]
Since $\Xbf_{i,0}$ and $\Xbf_{j,0}$ are the centers of  $\mathcal{Q}_i^*$ and $\mathcal{Q}_j^*$ (resp.), it is not hard to demonstrate that \eqref{eq:calqsep1} holds. 

Now we have verified that we may use Lemma \ref{lem:constreplem} with the constants $c_i^*$, the ball $B^*$ and the cubes $\mathcal{Q}_i^*$ in place of $Q_i$ therein. Doing so produces a regular $(L_g,M_g)$-regular Lip(1,1/2) function $g$, where $L_g$ and $M_g$ depend on $M_\theta$, $L$, $\tilde{\kappa}$ and dimension, with the properties that
\begin{equation}\label{eq:giscioncqstr}
g = c_i^* \quad \text{ on } \mathcal{Q}_i^* = 10^2\mathcal{Q}_i
\end{equation}
and
\begin{equation}\label{eq:gsup4mcqiconc}
\supp g \subset B^*.
\end{equation}

We are now prepared to define $F^*$. We define 
\[G(x_2,\dots, x_n, t) \coloneqq (g(x_2,\dots,x_n, t), x_2, \dots x_n, t),\]
the function whose image is the graph of $g$. The spatial components of $G$ are regular Lip(1,1/2) functions since $g$ is regular Lip(1,1/2) and the other components are trivially regular Lip(1,1/2) functions.
We set
\[F^*(\Xbf) \coloneqq
\begin{cases}
\widetilde{F}_i(\Xbf) & \text{ if } \Xbf \in 10\mathcal{Q}_i; \\
G(\Xbf) & \text{ if } \Xbf \in (\cup 10\mathcal{Q}_j)^c.
\end{cases} \]

We now check (ii)$^*$--(iv)$^*$ for $F^*$, starting with (ii)$^*$. The first thing to check is that the spatial components of $F^*$ are regular Lip(1,1/2) functions. To do so, we use Lemma \ref{lem:flatpatching} on the components of $F^*$. 
Set
\[\widetilde{\mathcal{Q}}_i \coloneqq 10\mathcal{Q}_i.\]
Observe that \eqref{eq:tildeshiftoutball}, \eqref{eq:tildeQ-contains-ball} and \eqref{eq:giscioncqstr} give
\begin{equation}\label{eq:Gfiagreement}
G(\Xbf) = \widetilde{F}_i(\Xbf), \quad \text{ on } 10\widetilde{\mathcal{Q}}_i \setminus \mathcal{Q}_i \supseteq 10\widetilde{\mathcal{Q}}_i \setminus \widetilde{\mathcal{Q}}_i.
\end{equation}
Thus, we have verified \eqref{eq:higaggreepatchlem} for the spatial components of $F^*$. Moreover, \eqref{eq:calqsep2} demonstrates that \eqref{eq.sepforflatpatch} holds for $Q_i$ replaced with $\widetilde{\mathcal{Q}}_i$. Therefore, we may apply Lemma \ref{lem:flatpatching} to obtain that the spatial components of $F^*$ are regular Lip(1,1/2) functions. Note that this also shows that $F^*$ is a Lipschitz map.

One can check that the map $F^*$ fixes the $t$-variable by first observing that it is obviously true for $G$, but also that the maps $\widetilde{F}_i$ fix $t$. In verifying $\widetilde{F}_i$ has this property, it is useful to note that the various projection operators fix $t$, since the planes are $t$-independent. Thus, to conclude (ii)$^*$, we need only check that $F^*$ is parabolic bi-Lipschitz. Since we have already observed that $F^*$ is a Lipschitz map, this amounts to proving only the lower bounds. In particular, we must show 
\begin{equation}\label{eq:fstarlowerbd}
	\dist(F^*(\Xbf), F^*(\Ybf)) \gtrsim \dist(\Xbf, \Ybf).
\end{equation}
We break this estimate into several cases.

{\bf Case 1:} $\Xbf, \Ybf \in 10\mathcal{Q}_i$ for some $i$. Here we use that
\[F^*(\Xbf) = \widetilde{F}_i(\Xbf) \quad \text{ and } \quad F^*(\Ybf) = \widetilde{F}_i(\Ybf)\]
and hence \eqref{eq:fstarlowerbd} holds because $\widetilde{F}_i$ is bi-Lipschitz. 

{\bf Case 2:}  $\Xbf, \Ybf \in (\cup_j \mathcal{Q}_j)^c$. Here we use \eqref{eq:Gfiagreement} to deduce
\[F^*(\Xbf) = \G(\Xbf) \quad \text{ and } \quad F^*(\Ybf) = \G(\Ybf),\]
so that \eqref{eq:fstarlowerbd} since $G$ is bi-Lipschitz. 

{\bf Case 3:} $\Xbf \in \mathcal{Q}_i$ and $\Ybf \in 10\mathcal{Q}_j$ for some $j \neq i$. In this case we note that by definition 
\[\mathcal{Q}_i \supseteq \xi_L B_{i, 0} \cap P_0\]
and 
\[10 \mathcal{Q}_j \supseteq \xi_L B_{j, 0} \cap P_0.\]
Then, \eqref{eq:projcontain} gives that 
\begin{equation}\label{eq:projcalqjimp} 
	\pi_{P_0}(F^*(\mathcal{Q}_i)) \subseteq \mathcal{Q}_i
\end{equation}
and 
\[\pi_{P_0}(F^*(10\mathcal{Q}_j)) \subseteq 10\mathcal{Q}_j.\]
Then, by our choice of $N$, the triangle inequality and \eqref{eq:calqsep2}, it holds that
\[\dist(F^*(\Xbf), F^*(\Ybf)) \ge \dist(\pi_{P_0}(F^*(\Xbf)),\pi_{P_0}( F^*(\Ybf))) \ge \dist(\mathcal{Q}_i, 10\mathcal{Q}_j) \approx \dist(\Xbf,\Ybf).\]

{\bf Case 4:} $\Xbf \in \mathcal{Q}_i$ and $\Ybf \in (\cup_j 10\mathcal{Q}_j)^c$. In this case, we use that $F^*(\Ybf) = G(\Ybf)$ and, since $G(\Ybf)$ is a function that generates a graph over $P_0$, we have
\[\pi_{P_0}(F^*(\Ybf)) = \Ybf.\]
Then, using \eqref{eq:projcalqjimp}, we have
\[\dist(F^*(\Xbf), F^*(\Ybf)) \ge \dist(\pi_{P_0}(F^*(\Xbf)),\pi_{P_0}( F^*(\Ybf))) \ge \dist(\mathcal{Q}_i, \Ybf) \ge \dist(\Xbf,\Ybf),\]
where we used that $\Ybf \in (10Q_i)^c$. 
Having handled all of the possible cases, we have shown that $F^*$ is parabolic bi-Lipschitz.

Using \eqref{eq:4mcqicontain}, \eqref{eq:gsup4mcqiconc} and the fact that $\kappa \in (1,\kappa^*)$, we immediately see that (iii)$^*$(a); indeed, 
\begin{equation}\label{eq:flexiiia}
	\text{$F^*$ is the identity in $(B^*)^c \cap P_0 \supseteq B({\Xbf}_1, \kappa^* \diam(Q_0))^c \cap P_0$.}
\end{equation} 
To see that (iii)$^*$(b) holds, note that, since $F^*$ is $L^*$-parabolic Lipschitz, $B^* = B(\Xbf_1,\kappa \diam(Q_0))$ and \eqref{eq:flexiiia} holds, we have 
\begin{align*}
	F^*(B(\Xbf_1,\kappa \diam(Q_0)) \cap P_0) &\subseteq B(\Xbf_1,4L^*\kappa \diam(Q_0)).
\end{align*} 
Property \eqref{eq:flexiiia} also gives 
\begin{align*}
	F^*(A(\Xbf_1,\kappa^*\diam(Q_0),\kappa \diam(Q_0)) \cap P_0) \subseteq B(\Xbf_1,\kappa^* \diam(Q_0)),
\end{align*}
where $A(\Xbf,R,r)$ denotes the annulus with outer radius $R$ and inner radius $r$. As long as $\kappa^* \geq 4L^*\kappa$, the above two inclusions give (iii)$^*$(b).

Finally, we show that (iv)$^*$ holds, that is, 
\[\cup_i F_{Q_i}(B({\Xbf}_{i,1}, \tilde{\kappa}\diam(Q_i))\cap P_{Q_i}) \subset F^*(P_0 \cap B(X_1, \kappa^*\diam(Q_0)).\]
To this end, note that, by construction of $\widetilde{F}_i$, specifically \eqref{eq:widetildeimpres}, it holds
\[F_{Q_i}(B({\Xbf}_{i,1}, \tilde{\kappa}\diam(Q_i))\cap P_{Q_i}) = F_{Q_i}(B_{i,1} \cap P_{Q_i}) =  \widetilde{F}_i(B_{i,0} \cap P_0 ).\]
Since $B_{i,0} \cap P_0 \subset 10\mathcal{Q}_i$, we have 
\[\cup_i F_{Q_i}(B({\Xbf}_{i,1}, \tilde{\kappa}\diam(Q_i))\cap P_{Q_i}) \subset F^*(\cup 10\mathcal{Q}_i).\]
Recalling that we have already shown \eqref{eq:4mcqicontain}, we easily see that (iv)$^*$ holds.

\end{proof}

\subsection{Discrete Carleson Measures and the Set-up of the Inductive Argument}
Here we closely follow the definitions in \cite{BH-BP, BHHLN-BP}.

Given $\mathcal{F}$, a collection of pairwise-disjoint dyadic cubes in $\dd$, we define 
\[\dd_{\mathcal{F}} = \dd \setminus (\cup_{Q \in \mathcal{F}} \dd(Q)).\]
These are the cubes that are ``above" the cubes in $\mathcal{F}$ with respect to ancestory.

\begin{definition}[Discrete Measures and Discrete Carleson Norms]
Suppose that $\Sigma \subseteq \mathbb{R}^{n+1}$ is ADR with constant $C'$ and $\mathbb{D} = \mathbb{D}(\Sigma)$ be as above. Let $\{\alpha_Q\}_{Q \in \mathbb{D}}$, where $\alpha_Q \in [0, \infty)$.
We let $\mut$ be the discrete measure associated to $\{\alpha_Q\}_{Q \in \mathbb{D}}$ defined by
\[\mut(\mathbb{D}') \coloneqq \sum_{Q \in \mathbb{D}'} \alpha_Q,\]
for any collection of cubes $\mathbb{D}' \subseteq \mathbb{D}$. If $\mathcal{F} = \{Q_j\}$ is a countable collection of pairwise disjoint cubes in $\mathbb{D}$, we define $\mut_{\mathcal{F}}$ by
\[\mut_{\mathcal{F}}(\mathbb{D}') \coloneqq \mut(\mathbb{D}' \cap \mathbb{D}_{\mathcal{F}}).\]
If $\mathcal{F} = \{Q_j\}$ is a countable collection of pairwise disjoint cubes in $\mathbb{D}$, we define the global Carleson norm of $\mut_{\mathcal{F}}$ as
\[\|\mut_\mathcal{F}\|_{\mathcal{C}} \coloneqq \sup_{Q \in \mathbb{D}} \frac{\mut_\mathcal{F}(\mathbb{D}(Q))}{\sigma(Q)}\]
and, for $Q_0 \in \mathbb{D}$, the localized Carleson norm of $\mut_\mathcal{F}$ (with respect to $Q_0$) as
\[\|\mut_\mathcal{F}\|_{\mathcal{C}(Q_0)} \coloneqq \sup_{Q \in \mathbb{D}(Q_0)} \frac{\mut_\mathcal{F}(\mathbb{D}(Q))}{\sigma(Q)}.\]
Here, if $\mathcal{F}= \emptyset$ we write $\mut$ in place of $\mut_{\mathcal{F}}$ in the notation above.
\end{definition}

An important ingredient in the proof of this direction of Theorem \ref{main.thrm} is the
following decomposition of a discrete Carleson region.

\begin{lemma}[{\cite[Lemma 7.2]{HM-I}}]\label{extraplem.lem}
Suppose that $\Sigma \subseteq \mathbb{R}^{n+1}$ is ADR with constant $C'$ and $\mathbb{D}(\Sigma)$ is as above. Suppose that $\mut$ is a discrete measure associated to $\{\alpha_Q\}_{Q \in \mathbb{D}}$. There exists $C$ depending on $d$ and $C'$ such that the following holds. Given $a\geq 0$, $b>0$, and $Q \in \mathbb{D}$ such that
$\mut(\dd(Q))\leq (a+b)\,\sigma(Q)$,
there is a family $\widetilde{\F}=\{Q_j\}\subset\dd(Q)$
of pairwise disjoint cubes such that
\begin{equation} \label{Corona-sawtooth}
\|\mut_{\widetilde{\F}}\|_{\mathcal{C}(Q)}
\leq C b
\end{equation}
and
\begin{equation}
\label{Corona-bad-cubes}
\sigma(B)
\leq \frac{a+b}{a+2b}\, \sigma(Q)\,,
\end{equation}
where $B$ is the union of those $Q_j\in\widetilde{\F}$ such that
$\mut\left(\dd(Q_j) \setminus \{Q_j\}\right)>a\,\sigma(Q_j)$. We set
\[\widetilde{\mathcal{B}} = \left\{Q \in \widetilde{\F}: \mut\left(\dd(Q_j)\setminus \{Q_j\}\right)>a\,\sigma(Q_j)\right\}. \] 
\end{lemma}

From now on, we are going to assume that $\Sigma$ is P-UR. Using Theorem \ref{PURiffcorona.thrm}, we have a corona decomposition for $\Sigma$ by regular Lip(1,1/2) graphs. We apply Theorem \ref{PURiffcorona.thrm} with $L = 1$ and we let $(\mathcal{G},\mathcal{B},\mathcal{M})$ denote the resulting collections of \textit{good, bad} and \textit{maximal} cubes, and $\mathcal{S}$ denote the resulting collection of stopping-time regimes.

We define
\begin{equation}\label{eq4.0}
\alpha_Q:= 
\begin{cases} \sigma(Q)\,,&{\rm if}\,\, Q\in \M\cup\B, \\
0\,,& {\rm otherwise}.\end{cases}
\end{equation}
and we let $\mut$ be the discrete measure with respect to $\{\alpha_Q\}_{Q \in \mathbb{D}}$. Note that, by assumption, 
\begin{equation}\label{carlconst.eq}
\|\mut\|_{\mathcal{C}} \le C_0. 
\end{equation}

As in \cite{BH-BP, BHHLN-BP}, we make some important observations.

\begin{lemma}\label{impob1.lem}
Fix $x \in \Sigma$ and $\sbf \in \mathcal{S}$. If there exists an (infinite) nested sequence of cubes
$Q_0 \supsetneq Q_1 \supsetneq Q_2 \dots$, with $x \in Q_k$ and $Q_k \in \sbf$, then $x \in \Gamma_{\sbf}$.
\end{lemma}
\begin{proof} The proof of this lemma is simple. Since $Q_{k+1} \subsetneq Q_k$,
it follows that $\ell(Q_k) \le 2^{-k}\ell(Q_0)$. Then \eqref{closegraphcorona.eq} gives that $\dist(x,\Gamma_{\sbf}) \lesssim 2^{-k}\ell(Q_0)$ for all $k \in \mathbb{N}$. Since $\Gamma_{\sbf}$ is closed,
$x \in \Gamma_{\sbf}$.
\end{proof}

\begin{lemma}\label{impob2.lem}
If $Q_0 \in \mathbb{D}(\Sigma)$, $\widetilde{\mathcal{F}}$ is a collection of pairwise disjoint sub-cubes of $Q_0$ and
    \[\|\mut_{\widetilde{\F}}\|_{\mathcal{C}(Q_0)} \le 1/2,\]
    then there exists $\sbf \in \mathcal{S}$ such that $Q \in \sbf$ whenever $Q \in \mathbb{D}_{\widetilde{\mathcal{F}}}(Q_0)$.
\end{lemma}
\begin{proof}
This proof is also simple but requires chasing a few definitions. We first note that we can assume that $\widetilde{\mathcal{F}} \neq \{Q_0\}$, as otherwise the lemma is vacuously true. For $Q \in \mathbb{D}_{\widetilde{\mathcal{F}}}(Q_0)$, we have \[\alpha_Q/\sigma(Q) \le \mut(Q)/\sigma(Q) \le \|\mut_{\widetilde{\F}}\|_{\mathcal{C}(Q_0)}  \le 1/2.\]
By definition $\alpha_Q/\sigma(Q) \in \{0,1\}$, hence, $Q \in \mathbb{D}_{\widetilde{\mathcal{F}}}(Q_0)$ can never be a maximal or a bad cube.

Let $\sbf_0$ be the stopping time regime such that $Q_0 \in \sbf_0$. Suppose, for the sake of contradiction, that $Q \in \mathbb{D}_{\widetilde{\mathcal{F}}}(Q_0)$ but $Q \not \in \sbf_0$. Since $Q$ is not maximal or bad, it must be the case that $Q\in \sbf$ for some $\sbf \neq \sbf_0$. It can't be the case that $Q_0 \subseteq Q(\sbf)$ (the maximal cube for $\sbf$) as by coherency of the stopping time regimes, $Q_0 \in \sbf$, which would yield a contradiction. On the other hand, if $Q(\sbf) \subset Q_0$ then since $Q \subset Q(\sbf)$ we have $Q(\sbf) \in \mathbb{D}_{\widetilde{\mathcal{F}}}(Q_0)$. This is contradiction to the fact that $\mathbb{D}_{\widetilde{\mathcal{F}}}(Q_0)$ contains no maximal cubes.
\end{proof}

Combining the two lemmas above, we obtain the following.

\begin{lemma}\label{impobcor.lem}
Let $Q_0 \in \mathbb{D}(\Sigma)$ and $\widetilde{\mathcal{F}} = \{Q_j\}_{j}$ be a collection of pairwise disjoint subcubes of $Q_0$, with
$\widetilde{\mathcal{F}} \neq \{Q_0\}$ and
    \[\|\mut_{\widetilde{\F}}\|_{\mathcal{C}(Q_0)} \le 1/2.\]
    Let $\sbf_0$ be the stopping time regime such that $Q_0 \in \sbf_0$, which exists by Lemma \ref{impob2.lem}. Let $\sbf'_0 = \sbf_0 \cap \mathbb{D}(Q_0)$, a semi-coherent stopping time regime.
    Set $A = Q_0 \setminus \cup_{j} Q_j$. If $x \in A$ then $x \in \Gamma_{\sbf_0'}$.
\end{lemma}
\begin{proof}
Let $x \in A$. By the properties of dyadic cubes, for any $Q \in \mathbb{D}(Q_0)$ such that $x \in Q$ we
have that $Q$ is not contained in $\mathbb{D}(Q_j)$ for any $j$ (otherwise this would imply that $x \in Q_j$). Thus, $Q \in \mathbb{D}_{\widetilde{\mathcal{F}}}(Q_0)$, and it follows from Lemma \ref{impob2.lem} that $Q \in \sbf_0'$. Let $R_i$, $i = 0,1,2 \dots$, be such that $R_{i + 1}$ is the unique sub-cube of $R_i$ such that $x \in R_{i +1}$. Then, $x \in R_i$ and the collection $R_i$ satisfy the hypothesis of Lemma \ref{impob1.lem}. Hence, $x \in \Gamma_{\sbf_0'}$.
\end{proof}

\subsection{Setting Up the Induction Hypothesis and Induction Argument} 

Our goal is to show that for every $Q$ has a (fixed, uniform) coincidence with a regular parabolic bi-Lipschitz image of $n$-dimensional space-time $\mathbb{R}^n$. We recall that \eqref{carlconst.eq} holds and hence 
\begin{equation}\label{mutiscarl.eq}
	\mut(\dd(Q)) \le C_0\sigma(Q), \quad \forall Q \in \dd.
\end{equation} 
We set up a finite induction scheme that lets us say, whenever $a \in [0,C_0]$ and
\[\mut(\dd(Q)) \le a\sigma(Q), \quad \forall Q \in \dd, \]
then $Q$ has a (fixed, uniform) coincidence with a regular parabolic bi-Lipschitz image of $n$-dimensional space-time $\mathbb{R}^n$. We will show that having this hold for $a$ implies it holds for $a + b$, where $b$ is a sufficiently small constant depending only on dimension and ADR. This is roughly what we do below. However, the complexity of showing the statement holds for $a$ implies it holds for $a + b$ requires us to introduce additional quantitative parameters.

For $a \ge 0$, let $H(a)$ be the following statement: There exists positive constants $M_a,L_a, \kappa_a^*, \eta_a$ such that, if
$\mut(\dd(Q_0)) \le a\sigma(Q_0)$, then there exists a $t$-independent plane $P_{Q_0}$ and a function $F_{Q_0}: P_{Q_0} \to \mathbb{R}^{n+1}$ with the following properties. 
\begin{enumerate}
\item[(i)] $\dist((X_{Q_0},t_{Q_0}), P_{Q_0}) < \kappa_a^*\diam(Q_0)$.
\item[(ii)] $F_{Q_0}$ is a $(L_a,M_a)$ regular parabolic bi-Lipschitz map of $n$-dimensional space-time $\mathbb{R}^n$ (upon identifying $P_{Q_0}$ with $\mathbb{R}^n$).
\item[(iii)] If ${\Xbf}_1 = \pi(X_{Q_0},t_{Q_0})$ is the projection of the `center' of $Q_0$ onto $P_{Q_0}$, then 
\begin{enumerate}
\item[(a)] $F_{Q_0}$ is the identity in $B({\Xbf}_1, \kappa_a^*\diam(Q_0))^c \cap P_{Q_0}$,
\item[(b)] $F_{Q_0}(B(\Xbf_1, \kappa_a^*\diam(Q_0))) \subseteq B(\Xbf_1, \kappa_a^*\diam(Q_0))$.
\end{enumerate}
\item[(iv)] 
\[\sigma(Q_0 \cap F_{Q_0}(B({\Xbf}_1, \kappa_a^*\diam(Q_0)) \cap P_{Q_0})) \ge \eta_a\sigma(Q_0).\]
\end{enumerate}

The statement $H(0)$ is essentially trivial from Lemma \ref{impobcor.lem} and Theorem \ref{PURiffcorona.thrm}, see \cite{BHHLN-BP}. For this we note that, if $\mut(\dd(Q)) = 0$, then every cube in $\dd(Q_0)$ is in a single stopping time regime $\sbf$ and therefore $\sbf' = \dd(Q_0)$ is a semi-coherent stopping time regime contained in $\sbf$. Theorem \ref{PURiffcorona.thrm} furnishes an $(L,M)$-regular Lip(1,1/2) function $\psi = \psi_{\sbf'}$, whose graph coincides with $Q_0$ by Lemma \ref{impobcor.lem}. Setting $F(x,t) = (\psi(x,t),x,t)$, we clearly have that $F$ is a regular parabolic bi-Lipschitz map of $n$-dimensional space-time $\mathbb{R}^n$ (upon identifying $P_{Q_0}$ with $\mathbb{R}^n$).

We now begin the induction scheme in earnest. Fix $b > 0$, depending on dimension and ADR, such that $Cb \le 1/2$, where $C$ is from Lemma \ref{extraplem.lem}. We show that if $H(a)$ holds then $H(a + b)$ holds. Since $b$ is fixed, this is enough to show that $H(C_0)$ holds, and completes the proof of the theorem. 

Let $Q_0$ be such that $\mut(\mathbb{D}(Q_0)) \le (a + b)\sigma(Q_0)$. We apply Lemma \ref{extraplem.lem} to $Q_0$ to obtain $\widetilde{\mathcal{F}} = \{Q_j\}_j$, a collection of pairwise disjoint sub-cubes of $Q_0$ with the properties stated in Lemma \ref{extraplem.lem}. An important observation is that, by our choice of $b$, we have
\[\|\mut_{\widetilde{\F}}\|_{\mathcal{C}(Q_0)} \le 1/2.\]
This allows us to utilize Lemmas \ref{impob2.lem} and \ref{impobcor.lem} (in the case $\widetilde{\mathcal{F}} \neq \{Q_0\}$).

We define the following objects. Let 
\begin{equation*}
	\widetilde{\mathcal{G}} \coloneqq \widetilde{\mathcal{F}} \setminus \widetilde{\mathcal{B}}, \quad G \coloneqq \cup_{Q_j \in \widetilde{\mathcal{G}}} Q_j \quad \mbox{ and } \quad A = Q_0 \setminus (\cup_{Q_j \in \widetilde{\mathcal{F}}} Q_j),
\end{equation*}
where $\widetilde{\mathcal{B}}$ is from Lemma \ref{extraplem.lem}, and set 
\begin{align*}
	\gamma_a =1 - \frac{a + b}{a + 2b} > 0.
\end{align*}
By Lemma \ref{extraplem.lem}, we have
$\sigma(B) \le (1- \gamma_a) \sigma(Q_0)$. Hence, 
\[\sigma(A \cup G) = \sigma(Q_0 \cap B^c) \ge \gamma_a \sigma(Q_0).\]

There are a few cases to consider. Case 1 and Case 2a are `easy', and Case 2b is really the heart of the matter.
\\ 

\noindent{\bf Case 1:} $\sigma(A) > (\gamma_a/2)\sigma(Q_0)$. In this case, we use Lemma \ref{impobcor.lem} to say that there exists a stopping time regime $\sbf_0$ such that $Q_0 \in \sbf_0$ 
and, if we set $\sbf_0' = \sbf_0 \cap \dd(Q_0)$, then 
$A \subseteq \Gamma_{\sbf_0'}$. Here, $\Gamma_{\sbf_0'}$ is the graph of an $(L,M)$-regular Lip(1,1/2) function $\psi = \psi_{\sbf'}$ from Theorem \ref{PURiffcorona.thrm}. Setting, $F_{Q_0}(x,t) = (\psi(x,t),x,t)$, the properties (i)--(iii) hold by Theorem \ref{PURiffcorona.thrm} (provided $\kappa_{a+b}$ is sufficiently large depending on $\kappa$) and (iv) holds provided $\eta_{a + b} \le  (\gamma_a/2)$. \\

\noindent {\bf Case 2:} $\sigma(G) \ge (\gamma_a/2)\sigma(Q_0)$. We break down this case further.\\

\noindent {\bf Case 2a:} Case 2 holds and $\widetilde{\mathcal{F}} = \{Q_0\}$. In this case, recalling that we have used Lemma \ref{extraplem.lem}, we must have $Q_0 \notin \widetilde{\mathcal{B}}$, otherwise we would violate \eqref{Corona-bad-cubes}. Examining the definition of $\widetilde{\mathcal{B}}$, it follows that
\[ \mut\big(\dd(Q_0)\setminus \{Q_0\}\big) \le a\,\sigma(Q_0).\]
Thus, by pigeonholing, there exists a dyadic child of $Q_0$, which we denote by $Q_0'$, such that  
\[\mut\big(\dd(Q_0')\big) \le a\,\sigma(Q_0).\]
In particular, the induction hypothesis holds for $Q_0'$. Therefore, we may apply $H(a)$ to $Q_0'$ to find a bi-Lipschitz map satisfying all of the desired properties, provided that $\kappa_{a + b}^* \ge C \kappa_{a}$, where $C$ depends on dimension and ADR -- this accounts for the fact that 
\[\diam(Q_0') \le \diam(Q_0) \lesssim \diam(Q_0').\]
Also, we use that 
\[\sigma(Q_0' \cap F_{Q_0'}(B({\Xbf}_1, \kappa_a^*\diam(Q_0')) \cap P_{Q_0'})) \ge \eta_a\sigma(Q_0') \approx \eta_a\sigma(Q_0).\]

\noindent{\bf Case 2b:} Case 2 holds and $\widetilde{\mathcal{F}} \neq \{Q_0\}$. In this case, our intention is to use Lemma \ref{lem:main}. By a basic covering lemma, we can extract a family of cubes $\mathcal{G}' = \{Q_i'\} \subseteq \widetilde{\mathcal{G}}$ such that 
\[\dist(\widetilde{Q}_i, \widetilde{Q}_j)  \ge N \max\{\diam(\widetilde{Q}_i), \diam(\widetilde{Q}_j)\} \quad \mbox{ for } i \neq j, \]
and
\[\sigma(\cup_i \widetilde{Q}_i) \gtrsim \sigma(\cup_i Q_i) = \sigma(G),\]
where the implicit constant depends on $N$, dimension and ADR. Moreover, by the definition of $\widetilde{G} = \widetilde{\mathcal{F}} \setminus \widetilde{\mathcal{B}}$, we have (as in Case 2a) that, for every $\widetilde{Q}_i$ there exists a child $\widetilde{Q}_i'$ for which the induction hypothesis $H(a)$ holds. Moreover, it holds that  
\[\dist(\widetilde{Q}_i', \widetilde{Q}_j')  \ge N \max\{\diam(\widetilde{Q}_i'), \diam(\widetilde{Q}_j')\}\]
and 
\begin{equation}\label{eq:piunionlwbd}
\sigma(\cup_i \widetilde{Q}_i') \gtrsim \sigma(\cup_i Q_i) = \sigma(G),
\end{equation}
where the implicit constant depends on $N$, dimension and ADR. For the sake of notation we will write $\{R_i\} = \{\widetilde{Q}_i'\}$. Clearly, the collection $\{R_i\}$ satisfies the separation condition in Lemma \ref{lem:main}, that is, \eqref{eq:sepinmainlem}. Applying the induction hypothesis H(a) to each $\{R_i\}$, there exists a $t$-independent plane $P_{R_i}$ and bi-Lipschitz map such $F_{R_i} \colon P_{R_i} \to \RR^{n+1}$ such that Lemma \ref{lem:main} (i)-(iii) are satisfied with parameters $(L, M) = (L_a,M_a)$ and $\tilde{\kappa} =\kappa_a^*$. Moreover, if $\Xbf_{i,1} = \pi_{P_{R_i}}(X_{R_i},t_{R_i})$ denotes the projection of the `center' of $R_i$ onto $P_{R_i}$, we have 
\[\sigma(R_i \cap F_{R_i}(B({\Xbf}_{i,1}, \kappa_a^*\diam(R_i)) \cap P_{R_i})) \ge \eta_a\sigma(R_i).\]
By \eqref{eq:piunionlwbd} and the disjointedness of $\{R_i\}$, this yields
\begin{equation}\label{eq:Riunionimbd}
\sigma\left(\bigcup_i \left( R_i \cap F_{R_i}\left(B({\Xbf}_{i,1}, \kappa_a^*\diam(Q_0)) \cap P_{R_i}\right) \right)  \right) \gtrsim \eta_a\sigma(G) \gtrsim \sigma(Q_0),
\end{equation}
where, in the final inequality, we used that we are in Case 2.

The last ingredient we need for Lemma \ref{lem:main} is a $t$-independent plane $P_0$ and regular Lip(1,1/2) function $\theta \colon P_0 \to \RR$ which satisfies the hypothesis therein.  As in Case 1, we use Lemma \ref{impobcor.lem} to say that there exists a stopping time regime $\sbf_0$ such that  $Q \in \sbf_0$ whenever $Q \in \mathbb{D}_{\widetilde{\mathcal{F}}}(Q_0)$.
If we set $\sbf_0' = \sbf_0 \cap \dd(Q_0)$, then $\sbf_0'$ is a stopping time regime that is a subset of $\sbf_0$ and 
$Q \in \sbf_0'$ whenever $Q \in \mathbb{D}_{\widetilde{\mathcal{F}}}(Q_0)$. Theorem \ref{PURiffcorona.thrm} provides a $t$-independent plane $P_0$ and regular Lip(1,1/2) function $\theta \coloneqq \psi_{\sbf_0'} \colon P_0 \to \RR$ which satisfy \eqref{eq:closeplanecoronalem} and \eqref{eq:supportgraphlem1} with constant $\tilde{\kappa} = \kappa^*_a$ (so long as $\kappa^*_a \geq \kappa$, with $\kappa$ as in Theorem \ref{PURiffcorona.thrm}), and whose graph $\Gamma_{\sbf_0'}$ satisfies \eqref{closegraphcorona.eq}, i.e., 
\[\sup_{(X,t)\in 10 Q} \dist(X,t,\Gamma_{\sbf'_0} )\, + \,  \sup_{(Y,s)\in B((X_Q,t_Q),10 \diam(Q)) \cap \Gamma_{\sbf'_0}} \dist(Y,s, \Sigma) \,
\leq \diam(Q)\]
for all $Q \in \mathbb{D}_{\widetilde{\mathcal{F}}}(Q_0)$ (recall that we have applied Theorem \ref{PURiffcorona.thrm} with $L=1$). Thus, \eqref{eq:closegraphcoronalem} clearly holds for $Q = Q_0$. Furthermore, since the parent of $R_i$ is $\widetilde{Q}_i$, which is in $\widetilde{\mathcal{F}}$, the grandparent of $R_i$ (which is the parent of $\widetilde{Q}_i$) is in $\mathbb{D}_{\widetilde{\mathcal{F}}}(Q_0)$. Thus, \eqref{eq:closegraphcoronalem} holds for the grandparent of each $R_i$. 

Now, choosing $N$ large enough, we may apply Lemma \ref{lem:main} to a obtain a map $F^*:P_0 \to \mathbb{R}^{n+1}$ that satisfies (i)$^*$-(iii)$^*$ in Lemma \ref{lem:main} with parameters $(\kappa^*,L^*,M^*)$. In particular, all but property (iv) of $H(a+b)$ are satisfied for $F_{Q_0} = F^*$. To obtain item (iv) of $H(a+b)$, note that Lemma \ref{lem:main} (iv)$^*$ implies
\[\cup_i F_{R_i}(B({\Xbf}_{i,1}, \kappa_a^*\diam(R_i)) )\cap P_{R_i})) \subset F^*(P_0 \cap B(\Xbf_1, \kappa^*\diam(Q_0)),\]
where $\Xbf_1$ is the projection of $(X_{Q_0},t_{Q_0})$ onto $P_0$. 
Therefore, by \eqref{eq:Riunionimbd}, it holds that
\[F^*(P_0 \cap B(\Xbf_1, \kappa^*\diam(Q_0) \cap Q_0) \gtrsim \sigma(Q_0).\]

Now having shown that $H(a)$ holds for all $a$, we have proved Theorem \ref{t:UR-implies-BPLI}. Indeed, as we have observed in \eqref{carlconst.eq}, $\mut(\dd(Q)) \le C_0\sigma(Q)$ holds for {\it all} cubes $Q$. Therefore, (ii) and (iv) in $H(C_0)$ hold for all cubes $Q$.

{\bf Statement on AI use:} AI tools (ChatGPT and Claude) were used in the editing and writing of the manuscript. ChatGPT was used to identify the name of the Givens rotation and their explicit definition. The proof scheme and the arguments in the manuscript were provided by the authors.

\bibliography{ParaBPLIrefs}

\newcommand{\etalchar}[1]{$^{#1}$}
\begin{thebibliography}{BHH{\etalchar{+}}23b}

\bibitem[AS12]{AS-HardSard}
Jonas Azzam and Raanan Schul.
\newblock Hard {S}ard: quantitative implicit function and extension theorems
  for {L}ipschitz maps.
\newblock {\em Geom. Funct. Anal.}, 22(5):1062--1123, 2012.

\bibitem[Azz21]{azzam2021semi}
Jonas Azzam.
\newblock Semi-uniform domains and the {{\(A_\infty\)}} property for harmonic
  measure.
\newblock {\em Int. Math. Res. Not.}, 2021(9):6717--6771, 2021.

\bibitem[BH17]{BH-BP}
Simon Bortz and Steve Hofmann.
\newblock Harmonic measure and approximation of uniformly rectifiable sets.
\newblock {\em Rev. Mat. Iberoam.}, 33(1):351--373, 2017.

\bibitem[BHH{\etalchar{+}}22]{BHHLN-BP}
Simon Bortz, John Hoffman, Steve Hofmann, Jose~Luis Luna-Garcia, and Kaj
  Nystr\"om.
\newblock Coronizations and big pieces in metric spaces.
\newblock {\em Ann. Inst. Fourier (Grenoble)}, 72(5):2037--2078, 2022.

\bibitem[BHH{\etalchar{+}}23a]{BHHLN-corona}
S.~Bortz, J.~Hoffman, S.~Hofmann, J.~L. Luna-Garcia, and K.~Nystr\"om.
\newblock Corona decompositions for parabolic uniformly rectifiable sets.
\newblock {\em J. Geom. Anal.}, 33(3):Paper No. 96, 67, 2023.

\bibitem[BHH{\etalchar{+}}23b]{BHHLN-CME}
Simon Bortz, John Hoffman, Steve Hofmann, Jos\'e{}~Luis Luna~Garc\'ia, and Kaj
  Nystr\"om.
\newblock Carleson measure estimates for caloric functions and parabolic
  uniformly rectifiable sets.
\newblock {\em Anal. PDE}, 16(4):1061--1088, 2023.

\bibitem[BHH{\etalchar{+}}26]{BHHLN-SIO}
Simon Bortz, John Hoffman, Steve Hofmann, Jos\'e~Luis Luna~Garc\'ia, and Kaj
  Nystr\"om.
\newblock Parabolic singular integrals with nonhomogeneous kernels.
\newblock {\em Bull. Lond. Math. Soc.}, 2026.
\newblock To appear.

\bibitem[BHMN25]{BHMN1}
Simon Bortz, Steve Hofmann, Jos\'e{}~Mar\'ia Martell, and Kaj Nystr\"om.
\newblock Solvability of the {${\rm L}^p$} {D}irichlet problem for the heat
  equation is equivalent to parabolic uniform rectifiability in the case of a
  parabolic {L}ipschitz graph.
\newblock {\em Invent. Math.}, 239(1):165--217, 2025.

\bibitem[Chr90]{Christ-cubes}
Michael Christ.
\newblock A {$T(b)$} theorem with remarks on analytic capacity and the {C}auchy
  integral.
\newblock {\em Colloq. Math.}, 60/61(2):601--628, 1990.

\bibitem[CMM82]{CMM}
R.~R. Coifman, A.~McIntosh, and Y.~Meyer.
\newblock L'int\'egrale de {C}auchy d\'efinit un op\'erateur born\'e{} sur
  {$L\sp{2}$}\ pour les courbes lipschitziennes.
\newblock {\em Ann. of Math. (2)}, 116(2):361--387, 1982.

\bibitem[Dah77]{Dahl-L2}
Bj\"orn E.~J. Dahlberg.
\newblock Estimates of harmonic measure.
\newblock {\em Arch. Rational Mech. Anal.}, 65(3):275--288, 1977.

\bibitem[Dav88]{David-cubes}
Guy David.
\newblock Morceaux de graphes lipschitziens et int\'egrales singuli\`eres sur
  une surface.
\newblock {\em Rev. Mat. Iberoamericana}, 4(1):73--114, 1988.

\bibitem[DJ90]{DJ}
G.~David and D.~Jerison.
\newblock Lipschitz approximation to hypersurfaces, harmonic measure, and
  singular integrals.
\newblock {\em Indiana Univ. Math. J.}, 39(3):831--845, 1990.

\bibitem[DS91]{DS-Ast}
G.~David and S.~Semmes.
\newblock Singular integrals and rectifiable sets in {${\bf R}^n$}: {B}eyond
  {L}ipschitz graphs.
\newblock {\em Ast\'erisque}, (193):152, 1991.

\bibitem[DS93a]{DS-AMS}
Guy David and Stephen Semmes.
\newblock {\em Analysis of and on uniformly rectifiable sets}, volume~38 of
  {\em Mathematical Surveys and Monographs}.
\newblock American Mathematical Society, Providence, RI, 1993.

\bibitem[DS93b]{david1993quantitative}
Guy David and Stephen Semmes.
\newblock Quantitative rectifiability and {Lipschitz} mappings.
\newblock {\em Trans. Am. Math. Soc.}, 337(2):855--889, 1993.

\bibitem[Eng17]{Eng-parafbp}
Max Engelstein.
\newblock A free boundary problem for the parabolic {P}oisson kernel.
\newblock {\em Adv. Math.}, 314:835--947, 2017.

\bibitem[HJ26]{HJ-SIO}
John Hoffman and Ben Jaye.
\newblock On singular integrals and quantitative rectifiability in parabolic
  space and the heisenberg group.
\newblock {\em arXiv preprint}, 2026.
\newblock arXiv:2510.26934.

\bibitem[HK12]{HK-cubes}
Tuomas Hyt\"onen and Anna Kairema.
\newblock Systems of dyadic cubes in a doubling metric space.
\newblock {\em Colloq. Math.}, 126(1):1--33, 2012.

\bibitem[HL96]{HL-ann}
Steve Hofmann and John~L. Lewis.
\newblock {$L^2$} solvability and representation by caloric layer potentials in
  time-varying domains.
\newblock {\em Ann. of Math. (2)}, 144(2):349--420, 1996.

\bibitem[HLN03]{HLN1}
Steve Hofmann, John~L. Lewis, and Kaj Nystr\"om.
\newblock Existence of big pieces of graphs for parabolic problems.
\newblock {\em Ann. Acad. Sci. Fenn. Math.}, 28(2):355--384, 2003.

\bibitem[HLN04]{HLN2}
Steve Hofmann, John~L. Lewis, and Kaj Nystr\"om.
\newblock Caloric measure in parabolic flat domains.
\newblock {\em Duke Math. J.}, 122(2):281--346, 2004.

\bibitem[HM14]{HM-I}
Steve Hofmann and Jos\'e{}~Mar\'ia Martell.
\newblock Uniform rectifiability and harmonic measure {I}: {U}niform
  rectifiability implies {P}oisson kernels in {$L^p$}.
\newblock {\em Ann. Sci. \'Ec. Norm. Sup\'er. (4)}, 47(3):577--654, 2014.

\bibitem[Hof97]{Hof-SIO}
Steve Hofmann.
\newblock Parabolic singular integrals of {C}alder\'on-type, rough operators,
  and caloric layer potentials.
\newblock {\em Duke Math. J.}, 90(2):209--259, 1997.

\bibitem[Jon88]{jones1988lipschitz}
Peter~W. Jones.
\newblock Lipschitz and bi-{Lipschitz} functions.
\newblock {\em Rev. Mat. Iberoam.}, 4(1):115--121, 1988.

\bibitem[Jon90]{Jones-salesman}
Peter~W. Jones.
\newblock Rectifiable sets and the traveling salesman problem.
\newblock {\em Invent. Math.}, 102(1):1--15, 1990.

\bibitem[KT97]{KT1}
Carlos~E. Kenig and Tatiana Toro.
\newblock Harmonic measure on locally flat domains.
\newblock {\em Duke Math. J.}, 87(3):509--551, 1997.

\bibitem[KT99]{KT2}
Carlos~E. Kenig and Tatiana Toro.
\newblock Free boundary regularity for harmonic measures and {P}oisson kernels.
\newblock {\em Ann. of Math. (2)}, 150(2):369--454, 1999.

\bibitem[KT03]{KT3}
Carlos~E. Kenig and Tatiana Toro.
\newblock Poisson kernel characterization of {R}eifenberg flat chord arc
  domains.
\newblock {\em Ann. Sci. \'Ecole Norm. Sup. (4)}, 36(3):323--401, 2003.

\bibitem[KW80]{KW-counter}
Robert Kaufman and Jang~Mei Wu.
\newblock Singularity of parabolic measures.
\newblock {\em Compositio Math.}, 40(2):243--250, 1980.

\bibitem[LM95]{Lew-Mur-Mem}
John~L. Lewis and Margaret A.~M. Murray.
\newblock The method of layer potentials for the heat equation in time-varying
  domains.
\newblock {\em Mem. Amer. Math. Soc.}, 114(545):viii+157, 1995.

\bibitem[MM97]{mattila1997measure}
Pertti Mattila and R.~Daniel Mauldin.
\newblock Measure and dimension functions: {Measurability} and densities.
\newblock {\em Math. Proc. Camb. Philos. Soc.}, 121(1):81--100, 1997.

\bibitem[Rig19]{Rigot}
S\'everine Rigot.
\newblock Quantitative notions of rectifiability in the {H}eisenberg groups.
\newblock 2019.
\newblock Preprint, arXiv:1904.06904.

\end{thebibliography}
\bibliographystyle{alpha}

\end{document}